\documentclass[12pt]{scrartcl}
\usepackage[
bookmarks,
  bookmarksopen=true,
  bookmarksnumbered=true,
  pdfusetitle,
  pdfcreator={},
  colorlinks,
  linkcolor=black,
  urlcolor=black,
  citecolor=black,
  plainpages=false,
  ]{hyperref}
	\usepackage[
  left=2.5cm,
  right=2.5cm,
  top=2cm,
]{geometry}

\usepackage{amsfonts}
\usepackage{amsmath}
\usepackage{amssymb}
\usepackage{amsthm}
\usepackage{algorithm}
\usepackage{algorithmic}
\usepackage{graphicx}
\usepackage{color}
\usepackage{subcaption}

\usepackage{tikz,pgfplots}
\usetikzlibrary{plotmarks}

\usepackage{prettyref}
\usepackage{mathrsfs}
\usepackage[utf8]{inputenc}
\usepackage{floatflt} 
\usepackage{float} 
\usepackage{wrapfig}

\usepackage{chngcntr}
\counterwithin{figure}{section}
\counterwithin{table}{section}

\renewcommand{\vec}[1]{\ensuremath \mathbf{\boldsymbol{#1}}}

\renewcommand{\phi}{\varphi}
\newcommand{\N}{\ensuremath{\mathbb{N}}}
\newcommand{\T}{\ensuremath{\mathbb{T}}}

\newcommand{\I}{\ensuremath{\mathcal{I}}}
\newcommand{\J}{\ensuremath{\mathcal{J}}}

\newcommand{\Z}{\ensuremath{\mathbb{Z}}}

\newcommand{\R}{\ensuremath{\mathbb{R}}}
\newcommand{\C}{\ensuremath{\mathbb{C}}}
\newcommand{\E}{\ensuremath{\mathbb{E}}}

\newcommand{\per}{\mathrm{per}}

\renewcommand{\O}{\ensuremath{\mathcal{O}}}
\renewcommand{\P}{\ensuremath{\mathbb{P}}}

\newcommand{\dx}{\mathrm{d}}
\newcommand{\e}{\mathrm{e}}
\newcommand{\im}{\mathrm{i}}

\newcommand{\mix}{\mathrm{mix}}

\DeclareMathOperator*{\diag}{diag}
\DeclareMathOperator*{\supp}{supp}
\DeclareMathOperator*{\cir}{circ}

\DeclareMathOperator*{\lin}{span}

\renewcommand{\d}{\, \mathrm{d}}
\newcommand{\Dx}{\mathrm{D}}

\newcommand{\norm}[1]{\left\lVert \smash{#1} \right\rVert}

\newcommand\multiset[2]
{\mathchoice{\left(\kern-0.5em{\binom{#1}{#2}}\kern-0.5em\right)}
	{\bigl(\kern-0.3em{\binom{#1}{#2}}\kern-0.3em\bigr)}
	{\bigl(\kern-0.3em{\binom{#1}{#2}}\kern-0.3em\bigr)}
	{\bigl(\kern-0.3em{\binom{#1}{#2}}\kern-0.3em\bigr)}}

\newcommand{\bB}{\mathbf{\boldsymbol{B}}}
\newcommand{\bj}{\mathbf{\boldsymbol{j}}}
\newcommand{\bh}{\mathbf{\boldsymbol{h}}}
\newcommand{\bk}{\mathbf{\boldsymbol{k}}}

\newcommand{\uinv}[2]{\uparrow\vec{#1}_{\vec{#2}}}  \newcommand{\X}{\ensuremath{\mathcal{X}}}

\newtheorem{theorem}{Theorem}[section]
\newtheorem{lemma}[theorem]{Lemma}

\newtheorem{corollary}[theorem]{Corollary}
\newtheorem{proposition}[theorem]{Proposition}
\newtheorem{problem}[theorem]{Problem}

\theoremstyle{definition}
\newtheorem{definition}[theorem]{Definition}
\newtheorem{example}[theorem]{Example}
\newtheorem{remark}[theorem]{Remark}

\newenvironment{Remark}[1][noisnotdefined]{ \ifthenelse{\equal{#1}{noisnotdefined}}{\begin{remark}}{\begin{remark}[#1]}\normalfont\rmfamily}{\bend\end{remark}}

\newenvironment{Definition}{ \begin{definition}\normalfont\rmfamily}{\end{definition}}

\numberwithin{equation}{section}

\newcommand{\bend}{\hspace*{0ex} \hfill \hbox{\vrule height
    1.5ex\vbox{\hrule width 1.4ex \vskip 1.4ex\hrule  width 1.4ex}\vrule
    height 1.5ex}}

\long\def\symbolfootnote[#1]#2{\begingroup
\def\thefootnote{\fnsymbol{footnote}}\footnote[#1]{#2}\endgroup}

\newcounter{todocounter}
\newcommand{\todo}[2][noisnotdefined]{
 \marginpar{\fcolorbox{black}{yellow}{\footnotesize\textbf{todo}}
 \ifthenelse{\equal{#1}{noisnotdefined}}{}{\textcolor{black}{\newline\tiny #1}}}
 \textbf{\ifthenelse{\equal{#2}{.}}
   {\fcolorbox{red}{white}{\textcolor{red}{$\maltese$}}}{{\textcolor{red}{#2}}}}
 \refstepcounter{todocounter}}

\newrefformat{def}{Definition \ref{#1}}
\newrefformat{sub}{Subsection \ref{#1}}
\newrefformat{prop}{Proposition \ref{#1}}
\newrefformat{fig}{Figure \ref{#1}}
\newrefformat{rem}{Remark \ref{#1}}
\newrefformat{cor}{Corollary \ref{#1}}

\title{Fast Hyperbolic Wavelet Regression meets ANOVA }

\date{\today}
\date{}
\author{Laura Lippert\thanks{Faculty of Mathematics, Chemnitz University of Technology, D-09107 Chemnitz, Germany.\newline E-mail:		\href{mailto:laura.lippert@math.tu-chemnitz.de}{\{laura.lippert, daniel.potts, tino.ullrich\}@math.tu-chemnitz.de}} 
\and 
Daniel Potts\footnotemark[1]
\and 	 
Tino Ullrich\footnotemark[1]
}

\begin{document}
\maketitle
\begin{abstract}
 We use hyperbolic wavelet regression for the fast reconstruction of high-dimensional functions having only low dimensional 
variable interactions. Compactly supported periodic Chui-Wang wavelets are used for the tensorized hyperbolic wavelet basis. 
In a first step we give a self-contained characterization of tensor product Sobolev-Besov spaces on the $d$-torus with arbitrary 
smoothness in terms of the decay of such wavelet coefficients. In the second part we perform and analyze scattered-data 
approximation using a hyperbolic cross type truncation of the basis expansion for the associated least squares method. 
The corresponding system matrix is sparse due to the compact support of the wavelets, which leads to a significant 
acceleration of the matrix vector multiplication. In case of i.i.d. samples we can even bound the approximation error 
with high probability by loosing only $\log$-terms that do not depend on $d$ compared to the best approximation. In addition, if the function has low effective 
dimension (i.e. only interactions of few variables), we qualitatively determine the variable interactions and omit ANOVA terms with low variance in a second step in order to increase the accuracy. This allows us to suggest an adapted model for the approximation. 
Numerical results show the efficiency of the proposed method. 
\end{abstract}


\section{Introduction}
We consider the problem of reconstructing a multivariate periodic function $f$ on the $d$-torus $\T^d$ from discrete function samples on the set of nodes $\X=\{\vec x_1,\ldots,\vec x_M\}\subset \T^d$. As a function model we use periodic Sobolev-Besov 
spaces with dominating mixed smoothness as they have proven useful in several multivariate approximation problems, see \cite{Ys10},  \cite[Chapt.\ 9]{DuTeUl16} and \cite[Chapter III,IV]{Tem93}. We provide fast algorithms which recover an individual function $f:\T^d\to \C$ from unstructured samples (scattered data), where the error is measured in $L_2(\T^d)$ and $L_\infty(\T^d)$. In the sense of \cite[Chapt.\ 22,23]{NW12} our recovery operator is a \textit{Monte-Carlo method} referring to the \textit{randomized setting}, since the nodes in $\X$ are drawn individually for each function $f$ and are not supposed to work simultaneously for the whole class of functions like in \cite[Sect.\ 9]{KaUlVo19}. We will extend the idea 
in \cite{Bo17,BoGr17, Bo18} to higher order smoothness and give statements that work with overwhelming probability. In fact, 
the core of the recovery is a wavelet based least squares algorithm. Throughout the paper we work in a rather general context 
with a minimal set of assumptions, namely in the context of (semi-)orthogonal compactly supported wavelets being a Riesz basis 
on a fixed level and having $m$ vanishing moments, see~\eqref{eq:support},~\eqref{eq:moments} and~\eqref{eq:Riesz}. This 
allows for considering two main examples, periodic orthogonal Daubechies wavelets as well as semi-orthogonal Chui-Wang wavelets.  

In this paper we give a self-contained and rather elementary proof for the characterization of Sobolev and Besov-Nikolskij spaces $H^s_\mix(\T^d)$ and $\bB^s_{2,\infty}(\T^d)$ in terms of such wavelet coefficients. There is an interesting qualitative phenomenon happening if $s$ equals $m$, the order of vanishing moments, see Theorem~\ref{thm:norm-psi*}.  In case of Chui-Wang wavelets this is reflected by the order of the underlying B-spline.    

For a finite dimensional hyperbolic cross type index set $I$, 
given as in~\eqref{eq:defPn} and cardinality $|I|$, we study the projection operator
$$P_I f =\sum_{j\in I}\langle f,\psi_j^*\rangle\psi_j.$$ 
As a first result wee state in Corollary~\ref{cor:l2error_Pn} that for $s<m$ the $L_2(\T^d)$-error is bounded by
$$\norm{f-P_If}_{L_2(\T^d)}\lesssim \frac{(\log |I|)^{(s+1/2)(d-1)}}{|I|^s}\norm{f}_{\bB^s_{2,\infty}(\T^d)}.$$
From this bound we infer that the operator $P_I$ yields the asymptotic optimal error among all operators 
with rank $|I|$, see \cite[Thm.\ 4.3.10]{DuTeUl16}. In case $s=m$, the above bound remains true if we replace the space $\bB^m_{2,\infty}(\T^d)$ by the smaller space $H^m_{\mix}(\T^d)$, see Corollary~\ref{cor:error_m=s}. One would actually expect a better rate. However, numerical experiments based on the tools developed in this paper indicate, that there must be an additional $\log$-term, see Figure~\ref{fig:RMSE_ex2}. Additionally, we give in Theorem~\ref{thm:L_infty} a bound for the $L_\infty{(\T^d)}$-error 
of the projection operator in both cases. Note, that these problems have some history. They belong to the field ``hyperbolic wavelet approximation'' and have already been investigated by several authors in the literature, e.g. \cite{DKT98,GrOsSc99,Gr00,SiUl09,Bo18}, see Section~\ref{sec:HWA}. \\

In order to investigate the scattered data problem, where we are engaged with the sample set $\X$ and the corresponding function values, 
we construct a recovery operator $S_I^\X$. This operator computes  
a best least squares fit 
$$S_I^\X f = \sum_{j\in I}a_j \psi_j,$$ 
to the given data 
$(f(\vec x))_{\vec x\in \X}$ from the finite dimensional subspace spanned by the wavelets with indices in the hyperbolic cross type set $I$.
We derive the coefficients $a_j$ by minimizing the error $\sum_{\vec x\in \X}|f(\vec x)-S_I^\X f(\vec x)|^2$ by 
using an LSQR algorithm, see~\eqref{eq:def_S_n^X}. This results in the hyperbolic wavelet regression in~Algorithm~\ref{alg:1}.
Assuming that the sample points in $\X$ are drawn i.i.d. and equally distributed at random with $|\X|\gtrsim |I|\,\log |I|$, 
we show in Corollary~\ref{cor:one_f} for $1/2<s<m$ and $r>1$ that there is a constant $C(r,d,s)>0$ such that for fixed $\|f\|_{\bB^s_{2,\infty}} \leq 1$
$$\P\Big(\norm{f-S_I^\X f}_{L_2(\T^d)}\leq C(r,d,s) \,\frac{(\log |\X|)^{(d-1)(s+1/2)+s}}{|\X|^s}\Big) \geq 1-2|\X|^{-r}\,.$$ For details to the constant $C(r,d,s)$ see Corollary~\ref{cor:one_f}. The proof is based on the well-known elementary Bernstein 
inequality for the deviation of a sum of random variables ~\cite[Theorem 6.12]{Stein08}. It turns out that the error of the least squares approximation asymptotically coincides with the behavior of the projection operator $P_I$ on the wavelet space with index set $I$. 
Even in the case $s=m$ we show in Corollary~\ref{cor:one_f} that the error of the least squares approximation inherits the error bound of the projection operator.
Such operators as well as the approximation error are also considered in~\cite{ChCoMi15,CoDaLe13,CoMi17,KaUlVo19}.
Note that, in contrast to the expected error, which has been considered in those references, we show a new 
concentration inequality for the approximation error $\norm{f-S_I^\X f}_{L_2(\T^d)}$. 
The often considered case where $f\in H^2_{\mix}(\T^d)$, \cite{GaGr13,Bo17}, can be covered in our results with wavelets with vanishing moments of order $m\geq 2$, where we benefit from piecewise quadratic wavelets what concerns the convergence rate, see Figure~\ref{fig:RMSE_ex2}.  \\

We use the parameter $n$ to determine the index-set $I$, for details see Section~\ref{sec:HWA}. In Lemma~\ref{lem:N} we state the well-known asymptotic bound 
 $|I| = \O(2^n n^{d-1})$ for hyperbolic cross type wavelet index sets $I$. This cardinality of the index set $|I|$ in the hyperbolic wavelet regression drastically reduces the complexity compared to a full grid approximation, which requires an index set of order $2^{d n}$. However, since the dimension $d$ still appears in the exponent of a logarithmic term, we aim to further reduce the index set $I$, while keeping 
the same approximation rate. To this end we introduce the analysis of variance (ANOVA) decomposition, see \cite{CaMoOw97,LiOw06,Holtz11}, \cite[Section 3.1.6]{NW08}, 
which decomposes the $d$-variate function into $2^d$ ANOVA terms, i.e.
$$f(\vec x) = \sum_{\vec u \subset\{1,\ldots,d\}}f_{\vec u}(\vec x_{\vec u}).$$
Each term corresponding to $\vec u$ depends only 
on variables $x_i$, where $i\in \vec u$. The number of these variables is called \textit{order} of the ANOVA term. 
However, in practical applications with high-dimensional functions, often 
only the ANOVA terms of low order play a role in order to describe the function well, see~\cite{CaMoOw97,KuSlWaWo09, DePeVo10, Wu2011, PoSc21}. For a rigorous mathematical treatment of this observation we introduce ANOVA inspired Sobolev spaces of dominating mixed derivatives with superposition dimension $\nu$
$$H^{s,\nu}_{\mix}(\T^d)=\{f\in H_{\mix}^s(\T^d)\mid f_{\vec u}=0 \text{ for all } \vec u\notin U_{\nu} \},$$
see also \cite{DuUl13} and \eqref{eq:HsU} for a natural generalization.
These function spaces describe functions, 
which only consist of ANOVA terms of order less than the superposition dimension $\nu$.
Note that, in contrast to \cite{Sc18} we do not restrict our model to ANOVA terms of order $\nu=1$.
We propose the new Algorithm~\ref{alg:anova}, which approximates functions in the space $H^{s,\nu}_{\mix}(\T^d)$ very well, by using the connection between
the ANOVA terms and the corresponding wavelet functions. This allows us to reduce the cardinality of the needed index set $I$ significantly to
$|I|=\O(\binom{d}{\nu}\, 2^n n^{{\nu}-1})$. Furthermore, we gain on the approximation error for $r>1$ to
$$\mathbb P\left(\norm{f-S_I^\X f}_{L_2(\T^d)}\leq C(r,\nu,s)\,\frac{(\log |\X|)^{s\nu}}{|\X|^s}\norm{f}_{H^s_{\mix}(\T^d)}\right)\geq 1-2|\X|^{-r},$$ 
with a constant $C(r,\nu,s)>0$, for details see Section~\ref{sec:truncate_ANOVA}.

Our further strategy, finding the unimportant dimension interactions and then building an adapted model, 
allows an interpretation of the data. This strategy is based on sensitivity 
analysis in combination with computing sensitivity indices, see~\eqref{eq:def_gsi}, analytically for the related wavelets. 
A similar algorithm for the approximation with Fourier methods is described in~\cite{PoSc19a}. 
Note that in this case one has to use a full 
index set in frequency domain, which gives the larger cardinality of the index set $|I|=\O(\binom{d}{\nu}\, 2^{n\nu})$ for the 
same approximation error, for details see Remark~\ref{rem:fourier}. \\

This paper is organized as follows. In Section \ref{sec:ANOVA} we recall the well-known ANOVA decomposition of a function on the 
$d$-dimensional torus. Section~\ref{sec:wavelets} is dedicated to the approximation using wavelets. 
In Section~\ref{sec:chui} we introduce periodic wavelet spaces, especially Chui-Wang wavelets which are semi-orthogonal, 
piecewise polynomial and compactly supported wavelets. We formulate three fundamental properties of wavelets which represent the only requirements for our theory. 
In Section~\ref{sec:bounded} we give a rather simple and elementary proof of the one-sided sharp wavelet 
characterization in our periodic setting. 
In Section~\ref{sec:HWA} we introduce the operator which truncates the wavelet decomposition by orthogonal 
projection on the wavelet spaces and we determine the number of necessary parameters. 
Two results from probability theory are summarized in Section~\ref{sec:prob}. We use these results in Section~\ref{sec:HWR} 
to bound the approximation error from scattered data approximation. This results in the concentration inequalities
in Corollary~\ref{cor:one_f}. Algorithm~\ref{alg:1} summarizes the hyperbolic wavelet regression, which gives us the approximant $S_I^\X f$.
In Section~\ref{sec:wav_anova} we show the connection between the ANOVA decomposition and the hyperbolic wavelet regression. Therefore, 
we determine in Theorem~\ref{thm:anova_psi} the ANOVA decomposition of our approximant $S_I^\X f$, which we use to improve our algorithm 
to Algorithm~\ref{alg:anova}.  
Finally, Section~\ref{sec:num} is dedicated to some numerical examples, where we apply our algorithms to confirm our theory. The relevant facts about Sobolev-Besov spaces of mixed smoothness have been collected in the appendix.

\subsection{Notation}
In this paper we consider multivariate periodic functions $f\colon \T^d\to \C$, $d\in \N$, where we identify the torus $\T$ with ${[-\tfrac 12,\tfrac 12)}$. 
As usually, we define the function spaces
$$L_p(\T^d):=\left\{f\colon \T^d\rightarrow \C\left| \norm{f}_{L_p(\T^d)}<\infty \right.\right\},$$
normed by
\begin{equation*}
\norm{f}_{L_p(\T^d)}=\begin{cases}
(\int_{\T^d}|f(\vec x)|^p\d \vec x)^{1/p} & \text{ if } p<\infty,\\
\sup_{\vec x\in \T^d} |f(\vec x)| & \text{ if } p=\infty.
\end{cases}
\end{equation*}
The overall aim is to approximate a square integrable function $f\in L_2(\T^d)$ by a function from a finite dimensional subspace of $L_2(\T^d)$. We study the 
scattered-data problem, i.e. we have given some sample points $\vec x\in \T^d$, where we denote the set of all sample points by $\X$, and 
we have given the function values $\vec y = (f(\vec x))_{\vec x\in \X}\in \C^M$.
We denote the number of sample points by $M=|\X|$.

Let us first introduce some notation. In this paper we denote by $[d]$ the set $\{1,\ldots,d\}$ and its power set by $\mathcal{P}([d])$. The $d$-dimensional input variable of the function $f$ is $\vec x$, where we denote the subset-vector by $\vec x_{\vec u}=(x_i)_{i\in \vec u}$ 
for a subset $\vec u\subseteq [d]$. The complement of those subsets is always with respect to $[d]$, i.e., $\vec u^c=[d]\backslash \vec u$.
For an index set $\vec u\subseteq [d]$ we define $|\vec u|$ as the number of elements in $\vec u$.
Since we will often use vector-notation, the relations $<$ and $>$ for vectors are always meant pointwise. 
As usual, we use the Kronecker delta 
$$\delta_{j,\ell} =\begin{cases}1 &\text{ if } j=\ell,\\
0& \text{otherwise.}
\end{cases}  $$ 
The vector-valued version is also 
a pointwise generalization $\delta_{\vec j,\vec \ell}=\prod_{i=1}^d\delta_{j_i,\ell_i}$.
Additionally, if we write $a+\vec b$, where $a$ is a scalar and $\vec b$ is a vector, we mean that we add $a$ to every component of the vector $\vec b$.  
Furthermore, the notation $X\lesssim Y$ means that $X\leq CY$ for some constant $C$ which does not depend on the relevant parameters.
The inner product for functions and vectors is defined by
$$\langle f,g\rangle=\int_{\T^d}f(\vec x)\overline{g(\vec x)}\d \vec x,\quad \quad\langle \vec x,\vec y\rangle = \vec y^*\cdot\vec x.$$
For $p\in\{1,2\}$ we introduce the norm 
$\norm{\vec x}_p=\left(\sum_{i=1}^d |\vec x_i|^p\right)^{1/p}.$
Moreover, we introduce the multi-dimensional Fourier coefficients on the torus by 
\begin{equation}\label{eq:Fourier}
c_{\vec k}(f)=\int_{\T^d}f(\vec x)\,\e^{-2\pi \im\langle\vec k,\vec x\rangle}\d \vec x.
\end{equation}

\section{The ANOVA decomposition}\label{sec:ANOVA}
The aim of sensitivity analysis is to describe the structure of multivariate periodic functions $f$ and analyze the influence of each 
variable. A traditional approach is to study the variance, defined in \eqref{eq:sigma}, of $f(\vec x_{\vec u})$ for $\vec u\subseteq [d]$, 
to find out which terms contribute how much to the total variance of $f$. A concept used frequently, see for instance \cite{CaMoOw97, Holtz11, LiOw06}, is the following.
\begin{definition}
The\textit{ ANOVA decomposition }(Analysis of variance) of a function ${f\colon \T^d\to \C}$ is given by
\begin{align}\label{eq:anova-decomp}
f(\vec x)&=f_{\varnothing}+\sum_{i=1}^d f_{\{i\}}(x_i)+\sum_{i\neq j=1}^d f_{\{i,j\}}(x_i,x_j)+\cdots+f_{[d]}(\vec x)\notag\\
&=\sum_{\vec u\in \mathcal{P}([d])}f_{\vec u}(\vec x_{\vec u}).
\end{align}
\end{definition}
For $\vec u=\varnothing$ the function $f_\varnothing$ is a constant that is equal to the grand mean
\begin{equation*}\label{eq:fzero}
f_{\varnothing}=\int_{\T^d}f(\vec x)\d \vec x.
\end{equation*}
The one-dimensional terms for $j=1,\ldots,d$, the so-called main effects, can be estimated by the integral 
$$f_{\{j\}}(x_j)=\int_{\T^{d-1}}(f(\vec x)-f_\varnothing)\d \vec x_{j^c},$$
which only depends on $\vec x$ through $x_j$, all other components $\vec x_{j^c}$ have been integrated out. 
The corresponding two-factor interactions are 
$$f_{\{i,j\}}( x_j,x_i)=\int_{\T^{d-2}}f(\vec x)\d \vec x_{\{j,i\}^c}-f_{\{j\}}(x_j)-f_{\{i\}}(x_i)-f_\varnothing.$$
In general, we define the following.
\begin{definition}\label{def:anova-terms}
Let $f$ be in $L_2(\T^d)$. For a subset $\vec u\subseteq[d]$ we define the \textit{ANOVA terms} by 
\begin{equation}\label{eq:anova-terms}
f_{\vec u}(\vec x_{\vec u})=\int_{\T^{d-|\vec u|}}f(\vec x)\d \vec x_{\vec u^c}-\sum_{\vec v\subset \vec u}f_{\vec v}(\vec x_{\vec v}).
\end{equation}
\end{definition}                                                                  
We do not want to attribute anything to $\vec x_{\vec u}$ that can be explained by $\vec x_{\vec v}$ for
strict subsets $\vec v\subset\vec u$, so we subtract the corresponding $f_\vec v(\vec x_{\vec v})$. 
By averaging over all other variables not in $\vec u$, we receive functions that depend only on $\vec x_{\vec u}$. The definition of $f_{[d]}(\vec x)$
ensures that the functions defined in Definition \ref{def:anova-terms} satisfy \eqref{eq:anova-decomp}.
There are many ways to make a decomposition of the form \eqref{eq:anova-decomp}. Indeed, an
arbitrary choice of $f_{\vec u}$ for all $|\vec u| < d$ can be accommodated by taking $f_{[d]}$ to be $f$
minus all the other terms. The terms in Definition \ref{def:anova-terms} are the unique decomposition \eqref{eq:anova-decomp}, 
such that they have additionally mean zero.
\begin{lemma}\label{lem:meanzero}
For all $\vec u\neq \varnothing $ the decomposition in Definition~\ref{def:anova-terms} fulfills
$$\int_{\T^{|\vec u|}}f_{\vec u}(\vec x_{\vec u})\d \vec x_{\vec u}=0.$$
\end{lemma}
\begin{proof}
The result follows by induction over $|\vec u|$ by
\begin{align*}
\int_{\T^{|\vec u|}}f_{\vec u}(\vec x_{\vec u})\d \vec x_{\vec u}&=\int_{\T^{|\vec u|}}\int_{\T^{d-|\vec u|}}f(\vec x)\d \vec x_{\vec u^c}\d \vec x_{\vec u}-
\int_{\T^{|\vec u|}}\sum_{\vec v\subseteq \vec u}f_{\vec v}(\vec x_{\vec v})\d \vec x_{\vec u}\\
&=\int_{\T^d}f(\vec x)\d \vec x-\sum_{\vec v\subseteq \vec u}\int_{\T^{|\vec v|}}f_{\vec v}(\vec x_{\vec v})\d \vec x_{\vec v}
\\
&=f_\varnothing-f_\varnothing=0.\qedhere
\end{align*}
\end{proof}
In order to show orthogonality of the ANOVA terms, we use the following lemma.
\begin{lemma}\label{lem:integral_x_j}
Let $f\in L_2(\T^d)$ and $\vec u\neq \varnothing$. Then for the ANOVA terms from Definition~\ref{def:anova-terms} we have
$$\int_{\T} f_{\vec u}(\vec x_{\vec u})\d x_j=0$$ for $j\in \vec u$. 
\end{lemma}
\begin{proof}
The proof is by induction on $|\vec u|$. For $\vec u=\{j\}$ this follows from Lemma~\ref{lem:meanzero}. 
Now suppose that $\int_\T f_{\vec v}(\vec x_{\vec v})\d x_j=0$ for $j\in \vec v$ whenever $1\leq |\vec v|\leq r<d$. 
Choose $\vec u$ with $|\vec u|=r+1$ and pick $j\in \vec u$. To complete the induction, we calculate
\begin{align*}
\int_{\T}f_{\vec u}(\vec x_\vec u)\d x_j&=\int_\T\int_{\T^{d-|\vec u|}}f(\vec x)\d \vec x_{\vec u^c}-\sum_{\vec v\subseteq \vec u}f_{\vec v}(\vec x_{\vec v})\d x_j\\
&=\int_\T\int_{\T^{d-|\vec u|}}f(\vec x)\d \vec x_{\vec u^c}\d x_j-\sum_{\vec v\subseteq \vec u}\int_\T f_{\vec v}(\vec x_{\vec v})\d x_j\\
&=f_{\vec u\backslash j}(\vec x_{\vec u\backslash j})+\sum_{\stackrel{\vec v\subset \vec u}{j\notin\vec v}}f_\vec v(\vec x_\vec v)-\sum_{\stackrel{\vec v\subseteq \vec u}{j\notin\vec v}}f_{\vec v}(\vec x_{\vec v})
=0.\qedhere
\end{align*}
\end{proof}
This lemma establishes $L_2(\T^d)$-orthogonality of the ANOVA-terms:
\begin{lemma}
Let $f\in L_2(\T^d)$. Then the ANOVA-terms from Definition~\ref{def:anova-terms} are orthogonal, i.e. for $\vec u\neq\vec v\subseteq [d]$ 
$$\langle f_{\vec u},f_{\vec v}\rangle=\int_{\T^d}f_{\vec u}(\vec x_\vec u)\overline{f_\vec v(\vec x_\vec v)}\d \vec x=0.$$
\end{lemma}
\begin{proof}
Since $\vec u \neq \vec v$, there either exists $j \in \vec u$ with $j \notin \vec v$, or $j \in v$ with $j \notin u$.
Without loss of generality suppose that $j \in u$ and $j \notin v$. We integrate $x_j$ out of $f_\vec u f_\vec v$ as follows:
\begin{align*}
\int_{\T^d}f_{\vec u}(\vec x_\vec u)\overline{f_\vec v(\vec x_\vec v)}\d\vec x&=\int_{\T^{d-1}}\int_\T f_{\vec u}(\vec x_\vec u)\overline{f_\vec v(\vec x_\vec v)} \d x_j\d\vec x_{j^c}\\
&=\int_{\T^{d-1}}\int_\T f_{\vec u}(\vec x_\vec u)\d x_j \overline{f_\vec v(\vec x_\vec v)} \d\vec x_{j^c}=0,
\end{align*}
using Lemma~\ref{lem:integral_x_j} for the inner integral.
\end{proof}

In order to get a notion of the importance of single terms compared to the entire function, we define the \textit{variance} of a function by 
\begin{equation}\label{eq:sigma}\sigma^2(f):=\int_{\T^d}\left|f(\vec x)-\int_{\T^d}f(\vec x)\d \vec x\right|^2\d \vec x=\int_{\T^d}|f(\vec x)|^2\d \vec x-f_\varnothing^2.\end{equation}
The idea of the ANOVA decomposition is to analyze which combinations of the input variables $x_j$ play a role for 
the approximation of $f$, i.e. which ANOVA terms are necessary to approximate the function $f$ and which terms can be omitted. The variances 
of the ANOVA terms indicate their importance, i.e. if an ANOVA-term has high variance, this term contributes 
much to the variance of $f$. For that reason we do the following.
For subsets $\vec u\subseteq [d]$ with $\vec u\neq \varnothing$ the \textit{global sensitivity indices} (gsi)
\cite{So01} are then defined as
\begin{equation}\label{eq:def_gsi}
\rho(\vec u,f):=\frac{\sigma^2(f_{\vec u})}{\sigma^2(f)}\in [0,1],
\end{equation}
where the variance of the ANOVA term $f_{\vec u}$ is 
\begin{equation*}\label{eq:sigma_fu}
\sigma^2(f_{\vec u})= \int_{\T^{|\vec u|}}|f_{\vec u}(\vec x_{\vec u})|^2\d \vec x_{\vec u},
\end{equation*}
since the mean of the ANOVA terms is zero.
The $L_2(\T^d)$-orthogonality of the ANOVA terms implies that the variance of $f(\vec x)$ for $L_2(\T^d)$-functions $f$ can be decomposed as
$$\sigma^2(f)=\sum_{\stackrel{\vec u\subseteq [d]}{\vec u\neq \varnothing}}\sigma^2(f_{\vec u}).$$
This implies
$$\sum_{\stackrel{\vec u\subseteq [d]}{\vec u\neq \varnothing}}\rho(\vec u,f)=1.$$
The global sensitivity index $\rho(\vec u,f)$ represents the proportion of variance of $f(\vec x)$ explained by the interaction between 
the variables indexed by $\vec u$. The knowledge of the indices is very helpful for understanding the influence 
of the inputs, but the computation of these relies on the 
computation of the integrals of equation~\eqref{eq:anova-terms}. Following \cite{DuGiRoCa13} we want to perform 
the sensitivity analysis on a function $g$, which approximates $f$.
\begin{remark}
Instead of using the conventional Lebesgue measure $\d \vec x$, we can instead use another measure, which 
leads to a possible different ANVOVA-decomposition. But $L_2(\T^d)$-orthogonality of the ANOVA terms still 
holds. Using for instance the Dirac measure located at a point $\vec a$, which is a simple evaluation of $f$ at $\vec a$, 
yields to the \textit{anchored ANOVA decomposition}, see \cite{KuSlWaWo09}. 
\end{remark} 

\section{Approximation with wavelets}\label{sec:wavelets}
In this section we will use periodized, translated and dilated wavelets to approximate periodic functions. 
We will give some theoretical results, which apply on general wavelets that fulfill the properties in Definition~\ref{def:ppp}.
The main result of this section is Theorem~\ref{thm:norm-psi*}, which generalizes Theorem~\ref{thm:strong_characterization_nd} to the multivariate and dual case. 
In contrast to Theorem~\ref{thm:d_charakpsi} this is a sharp characterization of the decay of wavelet coefficients for functions in $H^s_\mix(\T^d)$ or $\bB_{2,\infty}^s(\T^d)$.

\subsection{Wavelet spaces}\label{sec:chui}
In order to get basis functions on $\T$, we use $1$-periodized functions. 
We first introduce the B-splines.
\begin{definition}
For $m\in \N$ we define the cardinal B-spline $B_m\colon \R\to \R$ of order $m$ as a piecewise polynomial function recursively by
\begin{equation}\label{eq:BSpline}
B_1(x)= 
\begin{cases}
1, & -\tfrac 12<x<\tfrac 12\\
0, &\text{otherwise},
\end{cases} \quad \text{and } \quad B_m(x) = \int_{x-\tfrac 12}^{x+\tfrac 12} B_{m-1}(y)\d y.
\end{equation}
\end{definition}
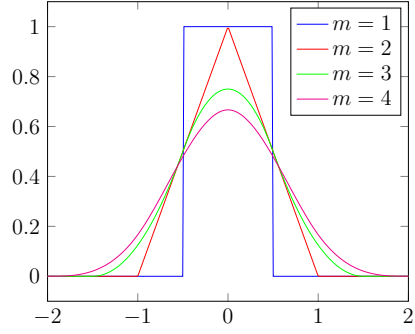
\begin{figure}
\centering
    \begin{tikzpicture}[scale=0.7,
declare function={
    func1(\x)= and(\x >= -0.5, \x < 0.5) * (1)   +0;
    func2(\x)= and(\x >= 0, \x <1) * (1-\x)   +and(\x >= -1, \x <0) * (1+\x)+0;
		func3(\x)= and(\x >= -3/2, \x <-0.5) * (1.125+\x*(1.5+\x*(0.5))) +
								and(\x >= -0.5, \x <0.5) * (0.75+\x*(-\x)) +
								and(\x >= 0.5, \x <1.5) * (1.125+\x*(-1.5+\x*(0.5)));
		func4(\x)=	and(\x >= -2, \x <-1) * (4/3+\x*(2+\x*(1+\x*(1/6)))) +
								and(\x >= -1, \x <0) * (2/3+\x*(0+\x*(-1+\x*(-0.5)))) +
								and(\x >= 0, \x <1) * (2/3+\x*(0+\x*(-1+\x*(+0.5)))) +
								and(\x >= 1, \x <2) * (4/3+\x*(-2+\x*(1+\x*(-1/6))))
		;				
  }			
]

\begin{axis}[xmin=-2,xmax=2,
legend entries={$m=1$,$m=2$,$m=3$,$m=4$}
]

\addplot[domain = -2:2,blue,samples = 300]{func1(x)};
\addplot[domain = -2:2,red,samples = 300]{func2(x)};
\addplot[domain = -2:2,green,samples = 300]{func3(x)};
\addplot[domain = -2:2,magenta,samples = 300]{func4(x)};

\end{axis}
\end{tikzpicture}
 \caption{Cardinal B-splines of order $m=1,2,3,4$} 
\end{figure}
The function $B_m(x)$ is a piecewise polynomial function of order $m-1$. Furthermore, the support 
of $B_m(x)$ is $(-\tfrac m2,\tfrac m2)$ and they are normalized by $\int_{-\tfrac m2}^{\tfrac m2}B_m(x)\dx x =1$. 
We introduce the function spaces $V_j$ by 
$V_j=\lin \{\phi(2^j\cdot -k)\mid k\in \Z\}$,
where the function $\phi$ is some scaling function, for example one can use here the $m$-th order cardinal B-spline $B_m(x)$ as scaling function $\phi$.
Consequently, we have
$$\cdots \subset V_{-1}\subset V_{0}\subset V_{1}\subset V_2\subset \cdots.$$
From the nested sequence of spline subspaces $V_j$, we build the the orthogonal complementary subspaces $W_j$, namely
$$V_{j+1}=V_j\oplus W_j,\quad j\in \N_0.$$
These subspaces are mutually orthogonal. Hence, we write 
\begin{equation}\label{eq:decomp} 
L_2(\R) = V_0\oplus W_0 \oplus W_1 \oplus \cdots
\end{equation}
We are interested in a \textit{wavelet function} $\psi\in W_0$ that generates the subspaces $W_j$ in the sense that 
$$W_j=\lin\{\psi_{j,k}\mid k\in \Z\},$$
where 
\begin{equation*}\label{eq:def_psis}
\psi_{j,k}(x)=2^{j /2}\psi(2^jx-k), \quad j\in \N_0, k\in \Z.
\end{equation*}
Thus, we use a normalization such that $\norm{\psi_{j,k}}_{L_2(\T)}=1$. 
\begin{example}\label{ex:Chui-Wang}
An example for the scaling function is the cardinal B-spline, $\phi = B_m$ of order $m$. The corresponding wavelet functions are the  
Chui-Wang wavelets \cite{chui92}, which are given by
\begin{equation*}
\psi(x) = \sum_{n}q_n B_m(2x-n-\tfrac m2),
\end{equation*}
where 
\begin{equation*}
q_n=\frac{(-1)^n}{2^{m-1}}\sum_{k=0}^m\binom{m}{k}B_{2m}(n+1-k-m).
\end{equation*}
\end{example}

\begin{figure}[ht]
\begin{subfigure}[c]{0.32\textwidth}
\includegraphics[width=\textwidth]{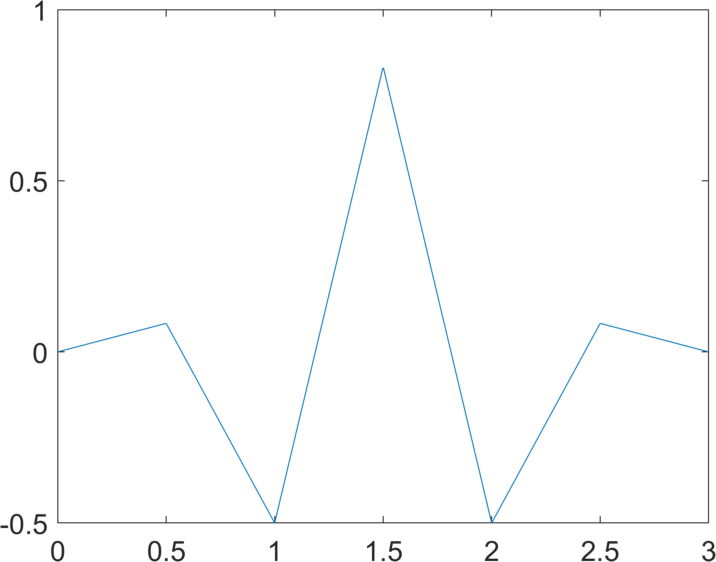}
\end{subfigure}
\begin{subfigure}[c]{0.32\textwidth}
\includegraphics[width=\textwidth]{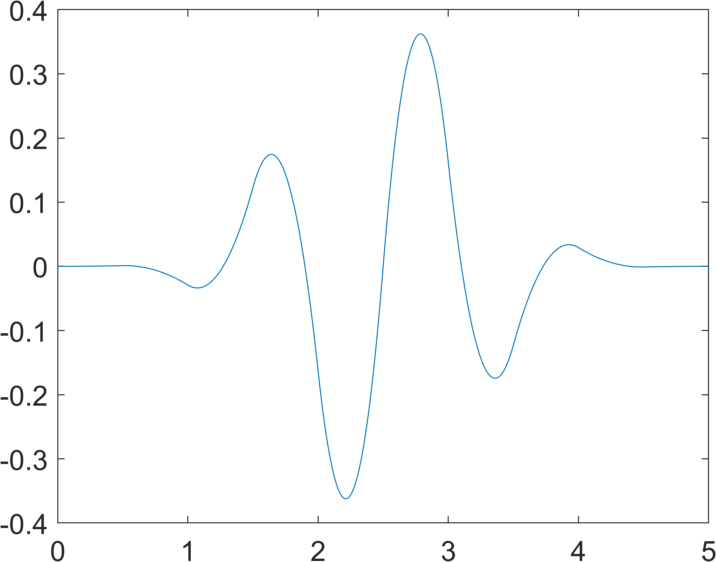}
\end{subfigure}
\begin{subfigure}[c]{0.32\textwidth}
\includegraphics[width=\textwidth]{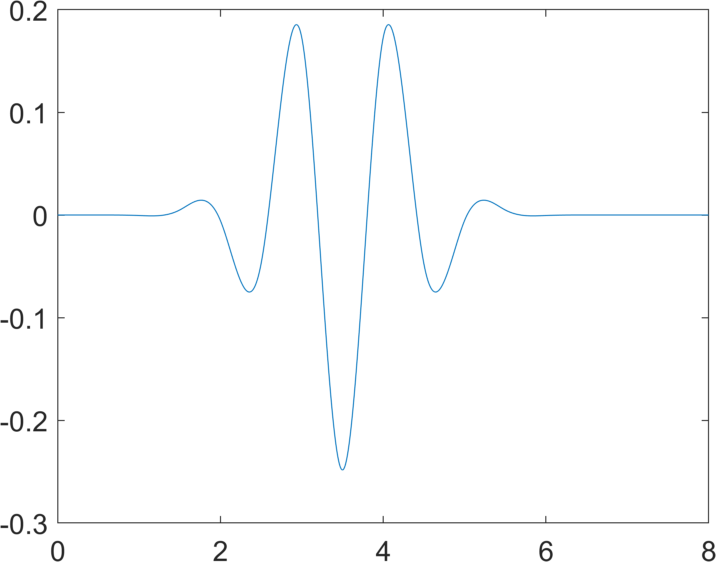}
\end{subfigure}
\caption{Chui-Wang-wavelets of order $m=2$ (left), $m=3$ (middle) and $m=4$ (right).}
\end{figure}
In order to approximate periodic functions, we use $1$-periodized versions of these wavelets,
$$\psi_{j,k}^\per(x)=\sum_{\ell\in \Z}\psi_{j,k}(x+\ell),$$
where we first have to dilate and then periodize. 
To get a periodic decomposition of the form \eqref{eq:decomp}, we have to periodize the scaling function $\phi\in V_0$ by
$$\phi^\per(x)= \sum_{\ell\in \Z}\phi(x+\ell)=1,$$
see \cite[Section 9.3]{dau92}, which is the constant function, denoted by $1_\T$. 
To simplify notation, we denote the scaling function $\phi(\cdot-k)$ by $\psi_{-1,k}$ and $\psi^\per_{-1,0}=\phi^\per = 1_\T$.

We want to approximate a function $f\in L_2(\T)$ by some function in $V_j$. 
If the wavelets do not induce an orthonormal basis, but a Riesz-basis for every index $j$, we have to consider the dual basis $\psi_{j,k}^*$, which has the property
$$\langle \psi_{j,k}, \psi_{i,\ell}^*\rangle=\delta_{i,j}\delta_{k,\ell}.$$
The usage of wavelets allow us a decomposition of a non-periodic $f\in L_2(\R)$ in terms of $\psi_{j,k}$, using \eqref{eq:decomp} by 
\begin{equation*}
f = \sum_{k\in \Z}\langle f ,\psi_{-1,k}^*\rangle\psi_{-1,k}+\sum_{j\geq 0}\sum_{k\in \Z}\langle f,\psi_{j,k}^* \rangle \psi_{j,k}. 
\end{equation*}
If $\{\psi_{j,k}(x)\}_k$ is an orthonormal basis for $W_j$, the dual wavelet and the wavelet coincide, i.e. $\psi_{j,k}=\psi_{j,k}^*$. 
The periodized version of \eqref{eq:decomp} is
\begin{equation}\label{eq:decomp_per} 
L_2(\T) = V_0^\per \oplus W_0^\per \oplus W_1^\per \oplus \cdots.
\end{equation} 
Hence, we can decompose every function $f\in L_2(\T)$ as
\begin{equation}\label{eq:decompf}
f = \langle f ,\psi_{-1,0}^\per \rangle \psi_{-1,0}^\per+\sum_{j\geq 0}\sum_{k=0}^{2^j-1}\langle f,\psi^{\per *}_{j,k} \rangle \psi^\per_{j,k}, 
\end{equation}
since the periodization of $\psi_{j,k}$ reduces the parameter $k$ to the range $\{0,\ldots 2^j-1\}$ and the 
periodization of $\psi_{-1}$ makes the parameter $k$ in the first term obsolete. 
The periodization inherits indeed the orthogonality in \eqref{eq:decomp_per}, which can be seen by setting $y = x+\ell'$ and $\tilde{\ell}=\ell-\ell'$ in 
\begin{align}\label{eq:sca_psi_per}
\langle\psi_{j,k}^\per,\psi_{j',k'}^\per\rangle &= 2^{\tfrac{j+j'}{2}}\int_\T \sum_{\ell\in \Z}\sum_{\ell'\in \Z} \psi(2^jx+2^j\ell-k)\psi(2^{j'}x+2^j\ell'-k')\d x\notag\\
&= 2^{\tfrac{j+j'}{2}} \sum_{\tilde{\ell}\in \Z}\sum_{\ell'\in \Z} \int_{-\tfrac 12+\ell'}^{\tfrac 12+\ell'}\psi_{2^jy+2^j\ell''-k}\psi_{2^{j'}y-k'}\d y\notag\\
&= \sum_{\tilde{\ell}\in \Z}2^{\tfrac{j+j'}{2}}\int_{-\infty}^{\infty}\psi(2^jy+2^j\tilde{\ell}-k)\psi(2^{j'}y-k')\d y\notag\\
&=\sum_{\tilde{\ell}\in \Z}\langle\psi_{j,-2^j\tilde{\ell}+k},\psi_{j',k'}\rangle
=\begin{cases} 0 &\text{ if } j\neq j'\\
\langle\psi_{j,k},\psi_{j,k'}\rangle&\text{ if } j = j'\text{ and } 2^j>\supp\psi,\\
\sum_{\ell\in \Z}\langle \psi_{j,k-2^j\ell}, \psi_{j,k'}\rangle &\text{ if }j = j'\text{ and }2^j\leq \supp\psi.\end{cases}
\end{align}
 
In order to give theoretical error estimates, we require some properties of the wavelets.
\begin{definition}\label{def:ppp}
We define the following properties of a wavelet $\psi:\R\to \R$.
\begin{itemize}
	\item The wavelet $\psi(x)$ has compact support, i.e.,
	\begin{equation}\label{eq:support}
	\supp \psi=[0,S]. \tag{P1}
	\end{equation}
	\item The wavelet has vanishing moments of order $m$, i.e.
	\begin{equation}\label{eq:moments}
\int_{-\infty}^{\infty} \psi(x)x^\beta \d x=0,\quad \beta = 0,\ldots ,m-1.\tag{P2}
\end{equation}
	\item The periodized wavelets form a Riesz-Basis for every index $j$ with
	\begin{equation}\label{eq:Riesz}
\gamma_m \sum_{k=0}^{2^j-1} |d_{j,k}|^2\leq \left\|\sum_{k=0}^{2^j-1}d_{j,k}\psi_{j,k}^\per(x)\right\|_{L_2(\T)}^2\leq  \delta_m \sum_{k=0}^{2^j-1} |d_{j,k}|^2\tag{P3}.
\end{equation}
\end{itemize}
\end{definition}
For example, the Chui-Wang wavelets of order $m$ in Example \ref{ex:Chui-Wang} fulfill all three properties. 
These wavelets have compact support with
$\supp \psi=[0,2m-1]$,
they have vanishing moments of order $m$ and the periodic Chui-Wang wavelets represent a Riesz-basis for every index $j$, 
\cite[Theorem 3.5.]{PlTa94} with the constants from there, see Table~\ref{tab:Riesz-const}.
\begin{table}[hbt]\centering
\begin{tabular}{c|ccccc}
    \hline
    $m$   & $1$ &$2$&$3$&$4$ &$5$\\
    \hline
		$\gamma_m$ & $1$ &$0.14814815$	&$0.03792593$ &$0.01005993$&$0.00267766$\\
    $\delta_m$ & $1$ &$0.33333333$	&$0.13386795$ &$0.05938886$&$0.02785522$\\
		\hline

  \end{tabular}
\caption{Riesz constants for the Chui-Wang wavelets from \cite{PlTa94}.}
\label{tab:Riesz-const}
\end{table}

To generalize the one-dimensional wavelets to higher dimensions, we use the following tensor-product approach. To this end we define the multi-dimensional wavelets
\begin{equation*}
\psi_{\vec j,\vec k}(\vec x)=\prod_{i=1}^d\psi_{j_i,k_i}(\vec x_i), \quad \vec x=(x_1,\ldots,x_d),
\end{equation*}
where $\vec j\in \Z^d$ and $\vec k=(k_i)_{i=1}^d\in \Z^d$ are multi-indices. Analogously, we define the $1$-periodized versions
\begin{equation}\label{eq:psi_per_multi}
\psi^\per_{\vec j,\vec k}(\vec x)=\prod_{i=1}^d\psi^\per_{j_i,k_i}(\vec x_i),
\end{equation}
where $\vec j=(j_i)_{i=1}^d,\, j_i\in \{-1,0,2,\ldots\}$ and $\vec k=(k_i)_{i=1}^d$ are multi-indices $\vec k\in \I_{\vec j}$. Hence, we define the sets 
\begin{equation*}
\I_{\vec j}=\times_{i=1}^d\begin{cases} \{0,1,\ldots 2^{j_i}-1\}&\text{ if } j_i\geq 0,\\
\{0\} &\text{ if } j_i=-1.
\end{cases}
\end{equation*}
In an analogous way we define the multi-variate dual wavelets and their periodization.

 \subsection{Boundedness of wavelet coefficients for mixed regularity}\label{sec:bounded}
The following results are essentially known and appear in several papers \cite{DKT98,GrOsSc99,SiUl09} 
in various different settings. We decided to give a rather simple and elementary proof of the one-sided sharp wavelet 
characterization in our periodic setting. We would like to point out that 
the vanishing moments of order $m$ of these wavelets play a crucial role for the partial characterization which we have in mind.  
Our proof can be easily extended to $1<p<\infty$. Note, that the analysis in \cite{GrOsSc99} relies on proper 
Jackson and Bernstein inequalities. 

The relevant function spaces are defined in the appendix. In order to analyze a best-approximation error of the function space $V_j$, we characterize the $H^s(\T)$-norm of a 
function by a sequence-norm of the wavelet-coefficients.
\begin{lemma}\label{lem:1d_charakpsi}
Let $f\in H^m(\T)$ and $\psi$ a wavelet, which is compactly supported, see~\eqref{eq:support}, and has vanishing moments of order $m$, see~\eqref{eq:moments}. Then there exists a constant $C$, which depends on $m$, such that
\begin{equation*}
\sup_{j\geq -1} 2^{ j m}\left(\sum_{k\in \I_j}|\langle f,\psi_{j,k}^\per\rangle|^2\right)^{1/2}\leq C \norm{f}_{H^m(\T)}.
\end{equation*}
\end{lemma}
\begin{proof}
We define the function 
\begin{equation*}
\Psi_m(x)=\int_{-\infty}^x \frac{\psi(t)(x-t)^{m-1}}{(m-1)!}\d t.
\end{equation*}
Note that this function is defined using the non-periodic wavelet function $\psi$, which has compact 
support on $[0,S]$. Because of the moment condition \eqref{eq:moments} and the fact 
that $(x-t)^{m-1}$ is a polynomial of degree at most $m-1$, we have $\Psi_m(x)\rightarrow 0$ for $x\rightarrow \pm \infty$. 
Hence, $\Psi_m(x)$ has also support $[0,S]$.  
Furthermore, $m$-times differentiation yields
$$\Dx^m\Psi_m(x)=\psi(x).$$
The periodization of $\Psi_m$ is
$$\Psi_m^\per(x) = \sum_{\ell\in \Z}\Psi_m(x+\ell)=\sum_{\ell\in \Z}\int_{-\infty}^{x+\ell} \frac{\psi(t)(x+\ell-t)^{m-1}}{(m-1)!}\d t.$$
Since $\Psi_m$ has compact support, the summation over $\ell$ is finite and we can interchange differentiation and summation, which yields 
$\frac{\d^m }{\d x^m}\Psi_m^\per(x)=\psi^\per(x).$
Analogously, we confirm that
\begin{equation}\label{eq:ddm_psi_per}
\frac{\d^m }{\d x^m} \left(\sum_{\ell\in \Z}\int_{-\infty}^{x+\ell} \frac{\psi_{j,k}(t)(x+\ell-t)^{m-1}}{(m-1)!}\d t\right)=\psi_{j,k}^\per(x).
\end{equation}
Now we calculate the inner products by using variable substitutions
\begin{align}\label{eq:sca_f_psi}
|\langle f, \psi_{j,k}^\per\rangle|
&=\left|\int_\T \overline{f(x)}\psi_{j,k}^\per(x)\d x\right|
=\left|\int_\T \overline{f(x)}\frac{\d^m }{\d x^m}\sum_{\ell\in \Z}\int_{-\infty}^{x+\ell} \frac{\psi_{j,k}(t)(x+\ell-t)^{m-1}}{(m-1)!}\d t\d x\right|\notag\\
&=2^{j/2}\,2^{-j}\left|\int_\T \overline{f^{(m)}(x)}\sum_{\ell\in \Z}\int_{-\infty}^{2^jx+2^j\ell-k} \frac{\psi(t')(x+\ell-\tfrac{t'+k}{2^j})^{m-1}}{(m-1)!}\d t'\d x\right|\notag\\
&=2^{j/2}\,2^{-j}\,2^{-j(m-1)} \left|\int_\T \overline{f^{(m)}(x)}\sum_{\ell\in \Z}\int_{-\infty}^{2^jx+2^j\ell-k} \frac{\psi(t')(2^jx+2^j\ell-k -t)^{m-1}}{(m-1)!}\d t'\d x\right|\notag\\
&=2^{j/2}\,2^{-jm}\left|\int_\T \overline{f^{(m)}(x)}\sum_{\ell\in \Z}\Psi_m(2^j x+2^j\ell-k)\d x\right|\notag\\
&\leq  2^{j/2}\,2^{-jm}\sum_{\ell\in \Z}\int_{I_{\ell,j,k}}\left|\overline{f^{(m)}(x)}\Psi_m(2^j x+2^j\ell-k)\right|\d x,
\end{align}
where $I_{\ell,j,k}=\supp V(2^j\cdot +2^j\ell -k)$. Since $\Psi_m$ is supported on $[0,S]$, the interval 
$I_{\ell,j,k}$ is $[2^{-j}k-\ell,2^{-j}(S +k)-\ell]\cap [-\frac 12,\frac 12)$. We denote the union $\cup_{\ell\in \Z}I_{\ell,j,k}$ 
by $I_{j,k}$. Note that in this proof we put points that are in multiple $I_{\ell,j,k}$, 
multiple times in $I_{j,k}$. This is the case if $S>2^{j}$, hence $|I_{j,k}|=S\,2^{-j}$, which can be greater than $1$.
\begin{align}\label{eq:sca_f_psi2}
|\langle f, \psi_{j,k}^\per\rangle|&\leq 2^{j/2}\,2^{-jm}\int_{I_{j,k}}\left|\overline{f^{(m)}(x)}\Psi_m(2^j x-k)\right|\d x\notag\\
&\leq 2^{j/2}\,2^{-jm}\,\left(\int_{I_{j,k}}\left| f^{(m)}(x)\right|^2\d x\right)^{1/2} \left(\int_{I_{j,k}}\left|\Psi_m( x)\right|^2\d x\right)^{1/2}\notag\\
&\leq 2^{j/2}\,2^{-jm}\,\left(\int_{I_{j,k}}\left| f^{(m)}(x)\right|^2\d x\right)^{1/2} \max_{x\in\R}\Psi_m(x) \, |I_{j,k}|^\frac 12,\notag\\
&\leq 2^{-jm}\, S^{1/2}\,\max_{x\in\R}\Psi_m(x)\left\| f^{(m)}\right\|_{L_2(I_{j,k})}.
\end{align}
Summation over $k$ means to unite $I_{j,k}$ for all $k\in\I_j$. The intervals $I_{j,k}$ for fixed $j$  
overlap for different $k$, but at most $\left\lceil S\right\rceil$-times. Hence for  $I_j=\displaystyle\cup_{k=0}^{2j-1}I_{j,k}$, we have
$\int_{I_j}|g(x)|^2\d x\leq S\int_\T |g(x)|^2\d x$.
All together, this yields
\begin{equation}\label{eq:wavcoefpsi}
2^{jm}\left(\sum_{k=0}^{2^j-1}|\langle f,\psi_{j,k}^\per\rangle|^2\right)^{1/2}\leq C \norm{f^{(m)}}_{L_2(\T)}.
\end{equation}
For the scaling function we have
\begin{align}\label{eq:wavcoefphi}
|\langle f,1_\T\rangle|=\left|\int_\T f(x)\d x\right|\leq \norm{f}_{L^1(\T)}\leq\norm{f}_{L_2(\T)}.
\end{align}
Now the assertion follows from \eqref{eq:wavcoefpsi} and \eqref{eq:wavcoefphi}.
\end{proof} 
We use this one-dimensional Lemma to give a characterization for multi-dimensional functions. 
Instead of the function space $H^m(\T)$ we have to use functions from $H^m_\mix(\T^d)$. 
\begin{theorem}\label{thm:d_charakpsi}
Let $f\in H^m_{\mix}(\T^d)$ and $\psi^\per_{\vec j,\vec k}$ be the periodized, dilated and translated versions 
of a wavelet $\psi$, which is compactly supported, see~\eqref{eq:support}, and has vanishing moments of order $m$, see~\eqref{eq:moments}. 
Then there exists a constant $C$, which depends on $m$ and $d$, such that
\begin{equation*}\label{eq:weak_characterization_nd}
\sup_{\vec j\geq -\vec 1} 2^{ |\vec j|_1 m}\left(\sum_{\vec k\in \I_\vec j}|\langle f,\psi_{\vec j,\vec k}^\per\rangle|^2\right)^{\tfrac 12}\leq C \norm{f}_{H^m_\mix(\T^d)},
\end{equation*}
where we define the index norm $|\vec j|_1=\sum_{i,j_i\geq 0} j_i$. 
\end{theorem}
\begin{proof}
We use a multi-variate version of the function $\Psi_m$ of the proof of Lemma~\ref{lem:1d_charakpsi}, which we get by tensorizing 
the one-dimensional functions, $\Psi_m(\vec x)= \prod_{i=1}^d \Psi_m(x_i)$. This function is supported on $[0,S]^d$. Furthermore, we have 
$$\Dx^{(m\cdot\vec 1)}\Psi_m (\vec x)=\psi(\vec x),$$
where $\vec 1$ is the $d$-dimensional vector of ones. We also derive multi-dimensional identities like in \eqref{eq:ddm_psi_per}, 
where we tensorize the functions and differentiate in every dimension $m$ times. 
Let the multi-index $\vec j$ be fixed, with $\vec u = \{i\in [d]\mid j_i\geq 0\}$.
We apply the partial integration in the dimensions $i$, where $j_i\geq 0$, i.e. in the dimensions $\vec u$. In these dimensions 
we use \eqref{eq:ddm_psi_per} with $j=j_i$ and $x=x_i$. Therefore we get
\begin{align*}
&|\langle f,\psi_{\vec j,\vec k}^\per\rangle|=\int_{\T^d}\overline{f(\vec x)}\, \psi_{\vec j,\vec k}^\per(\vec x)\d\vec x
=\int_{\T^{|\vec u^c|}}\int_{\T^{|\vec u|}}\overline{f(\vec x)} \,\psi_{\vec j,\vec k}^\per(\vec x)\d\vec x_\vec u\d \vec x_{\vec u^c}\\
= &\int_{\T^{|\vec u^c|}}\int_{\T^{|\vec u|}}\overline{f(\vec x)} \,\prod_{i=1}^{|\vec u|}\psi_{\vec j_{u_i},\vec k_{u_i}}^\per( x_{u_i})\d\vec x_\vec u\d \vec x_{\vec u^c}\\
\leq & 2^{|\vec j_\vec u|_1/2}2^{-|\vec j|_1 m}\int_{\T^{|\vec u^c|}}\sum_{\vec \ell_{\vec u}\in \Z^{|\vec u|}}\int_{I_{\vec \ell,\vec j,\vec k}}|(D^{(\vec 1_{(\vec u)}\,m)}\overline{f(\vec x)} )\,\Psi_m(2^{\vec j_{\vec u}}\vec x_{\vec u}+2^{j_{\vec u}}\vec l_{\vec u}-\vec k_{\vec u})|\d \vec x_\vec u \d \vec x_{\vec u^c},
\end{align*}
where the vector $\vec 1_{(\vec u)}$ is the vector which is $1$ at the indices $\vec u$ and all other entries are zero. 
The intervals for integration are given by the support of $\Psi_m$, i.e. 
$$I_{\vec \ell,\vec j,\vec k}=\supp \Psi_m(2^{\vec j_{\vec u}}\vec x_{\vec u}+2^{j_{\vec u}}\vec l_{\vec u}-\vec k_{\vec u})=[2^{\vec j_{\vec u}}\vec k_{\vec u}-\vec l_{\vec u},2^{-\vec j_{\vec u}}(S+\vec k_{\vec u})-\vec \ell_{\vec u}]\subset \T^{|\vec u|},$$ 
where expressions over multi-indices are always meant component-wise. That means $(2^{-\vec j}\vec k)_i=2^{-{j_i}}k_i$.
We denote the union $\cup_{\vec \ell\in \Z^d}I_{\vec \ell,\vec j,\vec k}$ by $I_{\vec j,\vec k}$.
Note that in this proof we put points that are in multiple $I_{\vec \ell,\vec j,\vec k}$, multiple 
times in $I_{\vec j,\vec k}$. This is the case if $S>2^{j_i}$ for some $i$, hence 
$|I_{\vec j,\vec k}|=S\,2^{-|\vec j_{\vec u}|_1}$, which can be greater than $1$. Like in \eqref{eq:sca_f_psi2} 
we use Cauchy-Schwarz-inequality and get
\begin{align*}
|\langle f ,\psi_{\vec j,\vec k}^\per\rangle|\leq C 2^{-|j_\vec u|_1m}\norm{D^{(\vec 1_{(\vec u)}\,m)}f }_{L_2(I_{\vec j_\vec u,\vec k_\vec u}\otimes\T^{d-|\vec u|})},
\end{align*} 
where the constant depends on $m$ and $|\vec u|$.
Summation over $\vec k\in \I_\vec j$ means to unite $I_{\vec j,\vec k}$ for all $\vec k\in I_\vec j$. 
The intervals $I_{\vec j_,\vec k}$ have length $2^{-|\vec j_\vec u|_1}S^{|\vec u|}$ and they are 
centered at the points $2^{\vec j_\vec u}\vec k_\vec u$ for $\vec k_\vec u\in\I_{\vec j_\vec u}$. Hence, 
the intervals overlap at most $\left\lceil S\right\rceil^{|\vec u|}$-times, i.e. we have
\begin{align*}
 2^{ 2|\vec j_\vec u|_1 m}\left(\sum_{\vec k\in \I_\vec j}|\langle f,\psi_{\vec j,\vec k}^\per\rangle|^2\right)\leq C \norm{D^{(\vec 1_{(\vec u)}\,m)}f}_{L_2(\T^{d})}^2\leq C\norm{f}_{H^m_\mix(\T^d)}^2,
\end{align*}
where the constant depends on $m$ and $d$. This finishes the proof.
\end{proof}
In the previous proof the function $f$ had to have the same smoothness as the order $m$ of the vanishing moments property \eqref{eq:moments}. If we require 
only slightly less smoothness, we get a much better characterization of functions in $V_j$, which uses the sum 
instead of the supremum of the wavelet coefficients if the fractional smoothness parameter $s$ satisfies $s<m$. Again, we prepare the multi-dimensional result by proving the following one-dimensional result first.
\begin{lemma} \label{lem:strong_characterization}
Let $f\in H^s(\T)$, $0<s<m$ and $\psi^\per_{j,k}$ an $1$-periodized wavelet, which is compactly supported, see~\eqref{eq:support}, has vanishing moments of order $m$, see~\eqref{eq:moments}, and forms a Riesz-basis, see~\eqref{eq:Riesz}. Then there exists a fixed constant $C$, which depends on $m$, such that
\begin{equation*}
\left(\sum_{j= -1}^{\infty} 2^{2 |j|_1 s}\sum_{k\in \I_j}|\langle f,\psi_{ j, k}^\per\rangle|^2\right)^{1/2}\leq \left(\delta_m\,\frac{2^s}{2^s-1}+C\,\frac{1}{2^{(m-s)}-1}\right) \norm{f}_{H^s(\T)},
\end{equation*}
where $\delta_m$ is the Riesz constant from \eqref{eq:Riesz} and where we use $|j|_1=j$ if $j\geq 0$ and $0$ otherwise.
\end{lemma}
\begin{proof}
The first summand for $j=-1$ is $|\langle f, 1_\T \rangle|^2=\norm{f}^2_{L^1(\T)}\leq C\norm{f}_{H^s(\T)}^2$. Now we consider a fixed index $j\geq 0$. 
We will use the equivalent norm given in \eqref{dyadic} in the appendix (univariate version).  This yields in particular for the block $f_{q}$
\begin{equation}\label{eq:norm_fq}
\norm{f_q}_{H^s(\T)}\leq 2^{qs}\norm{f_q}_{L_2(\T)}.
\end{equation}
The decomposition of $f$ in dyadic blocks and triangle inequality yields
\begin{align*}
\left(\sum_{k\in \I_j}|\langle f,\psi_{j,k}^\per\rangle|^2\right)^{1/2}
&\leq \sum_{\ell\in \Z}\left(\sum_{k\in \I_j }|\langle f_{j+\ell},\psi_{j,k}^\per\rangle|^2\right)^{1/2}.
\end{align*}
Let now $\ell\in \Z$ be fixed. We distinguish two cases and we begin with $\ell>0$. Here we have
\begin{align}\label{eq:1d-1}
\sum_{k\in \I_j }|\langle f_{j+\ell},\psi_{j,k}^\per\rangle|^2
\leq  \delta_m\norm{f_{j+\ell}}^2_{L_2(\T)}.
\end{align}
Note that at this point we need the property that $\psi_{j,k}$ form a Riesz basis for fixed $j$, i.e \eqref{eq:Riesz}, and every Riesz basis is a frame with the same constants.
Using the Riesz-basis property~\eqref{eq:Riesz}, summation over the weighted wavelet coefficients yields
\begin{align}\label{eq:proof_end_1}
\sum_{\ell> 0}\left(\sum_{j=0}^\infty 2^{2js}\sum_{k\in \I_j }|\langle f_{j+\ell},\psi_{j,k}^\per\rangle|^2\right)^{1/2}
&\leq\delta_m \sum_{\ell> 0}\left(\sum_{j=0}^\infty 2^{2(j+\ell)s}\,2^{-2\ell s}\norm{f_{j+\ell}}^2_{L_2(\T)}\right)^{1/2}\notag\\
&\leq \delta_m\sum_{\ell> 0} 2^{-\ell s} \norm{f}_{H^s(\T)} = \delta_m\frac{2^{ s}}{2^{ s}-1}\norm{f}_{H^s(\T)} .
\end{align}
For the remaining case $\ell\leq 0$ we use Lemma~\ref{lem:1d_charakpsi} and~\eqref{eq:norm_fq}, i.e we have
\begin{equation}\label{eq:1d-2}
\sum_{k\in \I_j }|\langle f_{j+\ell},\psi_{j,k}^\per\rangle|^2
\leq C 2^{-2jm}\norm{f_{j+\ell}}^2_{H^m(\T)}
\leq C 2^{-2jm}2^{2(j+\ell)m}\norm{f_{j+\ell}}^2_{L_2(\T)}
=C 2^{2\ell m}\norm{f_{j+\ell}}^2_{L_2(\T)},
\end{equation}
where the constant is from Lemma~\ref{lem:1d_charakpsi}  and depends on $m$.
This yields 
\begin{align}\label{eq:proof_end_2}
\sum_{\ell\leq 0}\left(\sum_{j=0}^\infty 2^{2js}\sum_{k\in \I_j }|\langle f_{j+\ell},\psi_{j,k}^\per\rangle|^2\right)^{1/2}
&\lesssim  \sum_{\ell\leq 0}\left(\sum_{j=0}^\infty 2^{2(j+\ell)s}\,2^{-2\ell s} 2^{2\ell m} \norm{f_{j+\ell}}^2_{L_2(\T)}\right)^{1/2}\notag\\
&= \sum_{\ell\leq 0} 2^{\ell(m-s)}\left(\sum_{j=0}^\infty 2^{(j+\ell)s}\norm{f_{j+\ell}}^2_{L_2(\T)}\right)^{1/2}\notag\\
&=\sum_{\ell\leq 0} 2^{\ell(m-s)}\norm{f}_{H^s(\T)}
= \frac{1}{2^{(m-s)}-1}\norm{f}_{H^s(\T)}.
\end{align}
The assertion follows from \eqref{eq:proof_end_1} and \eqref{eq:proof_end_2}.
\end{proof}

By a similar and straight-forward direction-wise analysis as in Theorem \ref{thm:d_charakpsi} we get the following multivariate version from \eqref{eq:1d-1} and \eqref{eq:1d-2},
\begin{equation}\label{multiv}
    \sum_{\vec k\in \I_\vec j}|\langle f_{\bj+{\vec \ell}},\psi_{ \vec j, \vec k}^\per\rangle|^2 
    \lesssim 2^{-2m|{\vec \ell}_-|_1}\|f_{\bj+{\vec \ell}}\|^2_{L_2(\T^d)}\,,
\end{equation}
where we define ${\vec \ell}_-:=((\ell_1)_-,...,(\ell_d)_-)$ with $x_{-}=\min\{0,x\}$.

The following result represents a multivariate version of  Lemma~\ref{lem:strong_characterization}. 

\begin{theorem}\label{thm:strong_characterization_nd}
Let $f\in L_2(\T^d)$, $0<s<m$ and $\psi$ a wavelet, which is compactly supported, see~\eqref{eq:support}, has vanishing moments of order $m$, see~\eqref{eq:moments}, and forms a Riesz-basis, see~\eqref{eq:Riesz}.
Then there 
exists a constant $C$, which depends on $m$, $s$ and $d$, such that
\begin{equation}\label{eq:strong_characterization_nd}
\left(\sum_{\vec j\geq  -\vec 1} 2^{2 |\vec j|_1 s}\sum_{\vec k\in \I_\vec j}|\langle f,\psi_{ \vec j, \vec k}^\per\rangle|^2\right)^{1/2}\leq C \norm{f}_{H_\mix^s(\T^d)},
\end{equation}
and 
\begin{equation}\label{eq:strong_characterization_Besov}
\sup_{\vec j\geq  -\vec 1} 2^{|\vec j|_1 s}\left(\sum_{\vec k\in \I_\vec j}|\langle f,\psi_{ \vec j, \vec k}^\per\rangle|^2\right)^{1/2}\leq C \norm{f}_{\bB^s_{2,\infty}(\T^d)},
\end{equation}
where we use $|\vec j|_1=\sum_{i,j_i\geq 0} j_i$.
\end{theorem}
\begin{proof} The relation in \eqref{eq:strong_characterization_nd} can be shown along the lines of Lemma \ref{lem:strong_characterization} using \eqref{multiv} instead of \eqref{eq:1d-1} and \eqref{eq:1d-2} at the respective place. However, let us additionally give a different proof argument based on an abstract tensor product result. For this end we need the sequence space 
$$b_2^s:=\left\{(a_{j,k})\subset \C \left\vert \left(\sum_{j=-1}^\infty\sum_{k\in \I_j}2^{2js}|a_{j,k}|^2\right)^{1/2}<\infty\right.\right\}.$$
Corollary 3.6.(i) for the case $p=2$ from~\cite{SiUl09} 
gives us a result of the multivariate versions of these one-dimensional sequence spaces.  
It was shown that the multivariate sequence spaces are the tensor products of the one-dimensional sequence spaces. 
In our case we have to consider the sequence spaces of the wavelet coefficients, where $a_{\vec j,\vec k}=\langle f,\psi_{ \vec j, \vec k}^{*,\per}\rangle$.
Theorem 2.1 also from~\cite{SiUl09} shows that the spaces $H^s_\mix(\T^d)$ coincide with the tensor products $\otimes_{i=1}^d H^s(\T)$. 
Our one-dimensional Lemma~\ref{lem:strong_characterization} bounds the operator which maps a function from 
$H^s(\T)$ to $b_2^s$. Hence, the tensor product operator is also bounded between the tensor-product spaces. 

As for \eqref{eq:strong_characterization_Besov} we again need a direct argument since a counterpart of the mentioned tensor product result is not available. The modification is straightforward and again based on \eqref{multiv}. 
\end{proof}
\begin{Remark}
Note that the constant in the previous theorem is the $d$-th power of the constant in Theorem~\ref{lem:strong_characterization}. 
We receive the same constant in an elementary proof which uses multi-dimensional ideas of the proof of Lemma~\ref{lem:strong_characterization}.
\end{Remark}

The version in Theorem~\ref{thm:strong_characterization_nd} is not suitable for our purpose. We want to approximate a function $f\in L_2(\T^d)$ in terms of multi-dimensional tensor products of dilated and translated versions of the wavelet $\psi$, given in~\eqref{eq:psi_per_multi}, i.e.
\begin{equation*}\label{eq:decompf2}
f = \sum_{\vec j\geq -\vec 1}\sum_{\vec k\in \I_\vec j}\langle f,\psi^{\per *}_{\vec j,\vec k} \rangle \psi^\per_{\vec j,\vec k}, 
\end{equation*}
which is the multi-dimensional version of \eqref{eq:decompf}. For that reason we need a characterization with 
the scalar products $\langle f, \psi_{\vec j,\vec k}^{\per *}\rangle$ instead of 
$\langle f, \psi_{\vec j,\vec k}^{\per}\rangle$. 
\begin{theorem}\label{thm:norm-psi*}
With the assumptions like in the previous theorem and letting $\psi^{\per *}_{\vec j,\vec k}$ denote the dual wavelets corresponding to the wavelets $\psi^{\per}_{\vec j,\vec k}$. There 
exists a constant $C$, which depends on $m$, $s$ and $d$, such that
\begin{equation*}
\left(\sum_{\vec j\geq  -\vec 1} 2^{2 |\vec j|_1 s}\sum_{\vec k\in \I_\vec j}|\langle f,\psi_{ \vec j, \vec k}^{\per, *}\rangle|^2\right)^{1/2}\leq C \norm{f}_{H_\mix^s(\T^d)},
\end{equation*}
and
\begin{equation*}
\sup_{\vec j\geq  -\vec 1} 2^{|\vec j|_1 s}\left(\sum_{\vec k\in \I_\vec j}|\langle f,\psi_{ \vec j, \vec k}^{\per,*}\rangle|^2\right)^{1/2}\leq C \norm{f}_{\bB^s_{2,\infty}(\T^d)},
\end{equation*}
where we define the index norm $|\vec j|_1=\sum_{i,j_i\geq 0} j_i$.
\end{theorem}
\begin{proof}
In this proof we use the property \eqref{eq:Riesz}, i.e. that $\{\psi^\per_{\vec j,\vec k}\}_{\vec k\in \I_\vec j}$ as well as their duals $\{\psi^{\per,*}_{\vec j,\vec k}\}_{\vec k\in \I_\vec j}$ are a Riesz-basis for every fixed $\vec j$. That means
\begin{align*}
\sum_{\vec k\in \I_\vec j}|\langle f,\psi_{ \vec j, \vec k}^{\per *}\rangle|^2\lesssim \left\|\sum_{\vec k\in \I_\vec j}\langle f,\psi_{ \vec j, \vec k}^{\per *}\rangle\psi_{ \vec j, \vec k}^{\per}\right\|^2_{L_2(\T^d)}
=\left\|\sum_{\vec k\in \I_\vec j}\langle f,\psi_{ \vec j, \vec k}^{\per}\rangle\psi_{ \vec j, \vec k}^{\per *}\right\|^2_{L_2(\T^d)}
\lesssim \sum_{\vec k\in \I_\vec j}|\langle f,\psi_{ \vec j, \vec k}^{\per }\rangle|^2.
\end{align*}
Hence, this theorem follows immediately from Theorem~\ref{thm:strong_characterization_nd}. 
\end{proof}
Note that the converse inequality for orthogonal wavelets in case $0<s<m-\tfrac 12$ was shown in 
\cite[Prop. 2.8 ii)]{SiUl09}.

\subsection{Hyperbolic wavelet approximation} \label{sec:HWA}
In the sequel we always deal with multi-dimensional periodic wavelets $\psi_{\vec j,\vec k}^\per$, which are 
compactly supported, see~\eqref{eq:support}, have vanishing moments of order $m$, see~\eqref{eq:moments}, and 
form a Riesz basis, see~\eqref{eq:Riesz}. The last subsection motivates to introduce an approximation operator $P_n$, which truncates the wavelet decomposition, by
\begin{equation}\label{eq:defPn}
P_nf := \sum_{\vec j\in \J_n } \sum_{\vec k\in \I_{\vec j}}\langle f,\psi_{\vec j,\vec k}^{\per *} \rangle \psi_{\vec j,\vec k}^\per,
\end{equation}
in order to approximate a function $f\in L_2(\T^d)$. To do so, we define the index sets
\begin{equation}\label{eq:J_n}
\J_n  = \{\vec j\in \Z^d\mid \vec j\geq -\vec 1, |\vec j|_1\leq n\} .
\end{equation}
The operator $P_n$ is the projection of a function in $L_2(\T^d)$ onto the space 
\begin{equation}\label{eq:Vnper}
V_ n^\per(\T^d)=\sum_{\vec j\in \J_n}\bigotimes_{i=1}^d V_{j_i}^\per.
\end{equation}
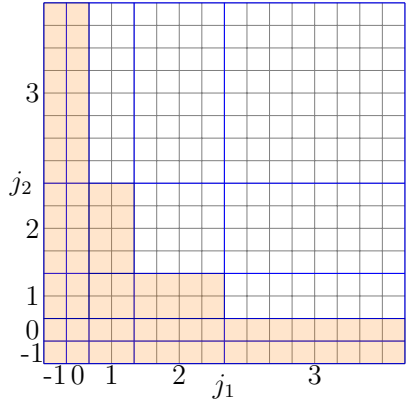
\begin{figure}[ht]\centering
\begin{tikzpicture} [scale=0.3]
\draw[step=1, gray, very thin] (0,0) grid (16,16);
\draw[blue] (0,0) -- (16,0);
\draw[blue] (0,1) -- (16,1);
\draw[blue] (0,2) -- (16,2);
\draw[blue] (0,4) -- (16,4);
\draw[blue] (0,8) -- (16,8);
\draw[blue] (0,16) -- (16,16);
\draw[blue] (0,0) -- (0,16);
\draw[blue] (1,0) -- (1,16);
\draw[blue] (2,0) -- (2,16);
\draw[blue] (4,0) -- (4,16);
\draw[blue] (8,0) -- (8,16);
\draw[blue] (16,0) -- (16,16);
\filldraw[fill=orange,opacity=0.2] (0,0) rectangle (16,2);
\filldraw[fill=orange,opacity=0.2] (0,2) rectangle (2,16);
\filldraw[fill=orange,opacity=0.2] (2,2) rectangle (8,4);
\filldraw[fill=orange,opacity=0.2] (2,4) rectangle (4,8);
\node[align=right,scale=0.9] at (0.5,-0.5) {-1};
\node[align=right,scale=0.9] at (1.5,-0.5) {0};
\node[align=right,scale=0.9] at (3,-0.5) {1};
\node[align=right,scale=0.9] at (6,-0.5) {2};
\node[align=right,scale=0.9] at (12,-0.5) {3};
\node[align=right,scale=0.9] at (-0.5,0.5) {-1};
\node[align=right,scale=0.9] at (-0.5,1.5) {0};
\node[align=right,scale=0.9] at (-0.5,3) {1};
\node[align=right,scale=0.9] at (-0.5,6) {2};
\node[align=right,scale=0.9] at (-0.5,12) {3};
\node[align=right,scale=0.9] at (8,-1) {$j_1$};
\node[align=right,scale=0.9] at (-1,8) {$j_2$};
\end{tikzpicture} \caption{Illustration of the number of 2-dimensional indices $\vec k$ where $|\vec j|\leq 3$.}
\label{fig:indices2d}
\end{figure}
In Figure \ref{fig:indices2d} every small square stands for one multi-index $\vec k$. 
The operator $P_3$ chooses those wavelet functions, for which the corresponding square is colored, i.e. 
all $\vec k$ in the index-set $\I_\vec j$ where $|\vec j|_1\leq3$.  
Using Theorem \ref{thm:norm-psi*}, we estimate the approximation error of this operator by
\begin{corollary}\label{cor:error_bound}
Let $f\in H^s_{\mix}(\T^d)$. For $s<m$ we have for the projection operator $P_n$ 
defined in~\eqref{eq:defPn}
$$\norm{f-P_nf}_{L_2(\T^d)}\lesssim 2^{-sn}\norm{f}_{H^s_\mix(\T^d)}. $$
\end{corollary}
\begin{proof}
Due to the wavelet decomposition of $f$, we have
\begin{align*}
\norm{f-P_nf}^2_{L_2(\T^d)}
&\lesssim\sum_{\stackrel{\vec j\geq -\vec 1}{|\vec j|_1>n}}\sum_{\vec k\in \I_\vec j}|\langle f,\psi_{\vec j,\vec k}^{\per *} \rangle|^2
=\sum_{\stackrel{\vec j\geq -\vec 1}{|\vec j|_1>n}}\sum_{\vec k\in \I_\vec j}2^{-2|\vec j|_1 s }2^{2|\vec j|_1 s }|\langle f,\psi_{\vec j,\vec k}^{\per *} \rangle|^2\\
&\lesssim 2^{-2n s }\norm{f}_{H^s_\mix(\T^d)}^2.
\end{align*}
Taking the square root gives the assertion.
\end{proof}
Note that this result can be compared with \cite[Theorem 3.25]{Bo17} for $m=2$. But we get a better approximation rate, 
since we proved the characterization in Theorem~\ref{thm:norm-psi*}, whereas in \cite{Bo17} only a characterization of type from Theorem~\ref{thm:d_charakpsi} was proven.
Related results also appeared in \cite[Theorem 3.2]{DKT98}, \cite[Proposition 6]{GrOsSc99}, \cite[Theorem 2.11]{SiUl09}, but with less transparent requirements on the wavelets.

We also give a relation between the number of necessary 
parameters (degrees of freedom) and the order of approximation. The content of the Lemma below is essentially known, see \cite[Lem. 3.6]{BuGr04}.

\begin{lemma}\label{lem:N}
Let $N:=\mbox{rank } P_n$ be the number of parameters, that we need to describe the space $V_n^\per(\T^d)$, defined in \eqref{eq:Vnper}, 
which is induced by a wavelet which fulfills \eqref{eq:support}, \eqref{eq:moments} and \eqref{eq:Riesz}. Then 
$$N=2^n\left(\frac{n^{d-1}}{(d-1)!}+\O(n^{d-2})\right)=\O(2^n\, n^{d-1}).$$ 
\end{lemma}
\begin{proof}
The number of parameters is
\begin{equation}\label{eq:Nsum}
N= \sum_{\vec j\in \J_n}|\I_\vec j|=\sum_{\vec j\in \J_n}2^{|\vec j|_1} = \sum_{\vec u\in\mathcal P([d]) } \sum_{\stackrel{\vec j_{\vec u}\geq \vec 0}{|\vec j_{\vec u}|_1\leq n} }2^{|\vec j_{\vec u}|_1},
\end{equation}
where $\vec u$ always denotes the index set $\vec u = \{i\in [d]\mid \vec j_i\geq 0\}$.
We consider each summand seperately. Therefore we consider the case where $\vec j\geq \vec 0$, 
\begin{align*}
\sum_{\stackrel{\vec j\geq \vec 0}{|\vec j|_1\leq n}}2^{|\vec j|_1}
&=\sum_{i=d}^{d+n} 2^{i-d}\sum_{|\vec j|_1=i}1
=\sum_{i=d}^{d+n} 2^{i-d}\binom{i-1}{d-1}
=\sum_{i=0}^{n} 2^{i}\binom{i+d-1}{d-1},
\end{align*}
since there are $\binom{i-1}{d-1}$ partitions of $i$ into non-zero natural numbers. For this sum holds
$$\sum_{i=0}^{n} 2^{i}\binom{i+d-1}{d-1}=2^n\left(\frac{n^{d-1}}{(d-1)!}+\O(n^{d-2})\right).$$
Summing over all $\vec u\in \mathcal P([d])$ gives us
$$N = \sum_{k=0}^d \binom{d}{k}2^n\left(\frac{n^{k-1}}{(k-1)!}+\O(n^{k-2})\right)
=\O(2^n\, n^{d-1}).\qedhere$$
\end{proof}

\begin{corollary}\label{cor:l2error_Pn}
Let $f\in H^s_{\mix}(\T^d)$. For $0<s<m$ we have for the projection operator $P_n$ 
defined in~\eqref{eq:defPn} with $N:=\mbox{rank } P_n$. Then
$$\norm{f-P_nf}_{L_2(\T^d)}\lesssim N^{-s}(\log N)^{s(d-1)}\norm{f}_{H^s_\mix(\T^d)}.$$
\end{corollary}
\begin{proof}
This follows from Corollary~\ref{cor:error_bound} together with Lemma~\ref{lem:N}.
\end{proof}
Note that the previous corollary deals with the case $s<m$, i.e. the smoothness $s$ is smaller than the order $m$ of vanishing moments of the wavelet. For the case $s=m$ we can only prove the following worse bound, which is based on the estimate in Theorem~\ref{thm:d_charakpsi}. We do not know whether this bound is optimal or can be improved. 
\begin{corollary}\label{cor:error_m=s}
Let $f\in H^m_\mix(\T^d)$ and $P_n$ being the approximation operator defined in~\eqref{eq:defPn}. 
Then
$$\norm{f-P_nf}_{L_2(\T^d)}\lesssim 2^{-mn}\, n^{(d-1)/2} \norm{f}_{H^m_\mix(\T^d)} \lesssim N^{-m}(\log N)^{(m+\frac 12)(d-1)}\norm{f}_{H^m_\mix(\T^d)}.$$
\end{corollary}
\begin{proof}
Like in Corollary~\ref{cor:error_bound} we have
\begin{align*}
\norm{f-P_nf}^2_{L_2(\T^d)}
&\lesssim \sum_{\stackrel{\vec j\geq -\vec 1}{|\vec j|_1>n}}\sum_{\vec k\in \I_\vec j}2^{-2|\vec j|_1 m }2^{2|\vec j|_1 m }|\langle f,\psi_{\vec j,\vec k}^{\per *} \rangle|^2\\
&\lesssim \norm{f}_{H^m_{\mix}(\T^d)}^2 \sum_{\stackrel{\vec j\geq -\vec 1}{|\vec j|_1>n}}2^{-2|\vec j|_1m}.
\end{align*}
In contrast to Corollary~\ref{cor:error_bound} we have to sum over the indices $\vec j$ instead of taking the supremum. 
By first considering the cases where $\vec j\geq \vec 0$, we have by \cite[Lemma 3.7]{BuGr04}, that 
$$\sum_{\stackrel{\vec j\geq \vec 0}{|\vec j|_1>n}}2^{-2|\vec j|_1m}\leq 2^{-2nm}\left(\tfrac{n^{d-1}}{(d-1)!}+\O(n^{d-2})\right).$$
Taking the scaling functions into account, which are constant, i.e. $\psi_{-1,0}=1$ we have
\begin{align*}
\sum_{\stackrel{\vec j\geq -\vec 1}{|\vec j|_1>n}}2^{-2|\vec j|_1m} 
& = \sum_{\vec u\in [d]} \sum_{\stackrel{\vec j_\vec u\geq \vec 0}{|\vec j_\vec u|_1>n}}2^{-2|\vec j|_1m}\\
&\lesssim 2^{-2nm}\sum_{\ell=0}^d\left(\tfrac{n^{\ell-1}}{(\ell-1)!}+\O(n^{\ell-1})\right) = 2^{-2nm}\left(\tfrac{n^{d-1}}{(d-1)!}+\O(n^{d-1})\right).
\end{align*}
The estimation regarding the number $N$ of parameters follows analogously as in Lemma~\ref{lem:N}.
\end{proof}
\begin{remark}\label{rem_besov_L2} With literally the same argument we obtain an analogous $L_2$-bound also for $f\in \bB^s_{2,\infty}(\T)$ if $s<m$. This is a direct consequence of Theorems \ref{thm:strong_characterization_nd}, \ref{thm:norm-psi*}.  
\end{remark}

The characterizations of our wavelet spaces also allow a bound on the $L_{\infty}$-error.
\begin{theorem}\label{thm:L_infty}
For $1/2<s<m$ we have for the projection operator $P_n$ 
defined in~\eqref{eq:defPn}
\begin{align*}
\norm{f-P_nf}_{L_\infty(\T^d)} &\lesssim 2^{-n(s-1/2)}n^{(d-1)/2} \norm{f}_{H^s_{\mix}(\T^d)},\\
\norm{f-P_nf}_{L_\infty(\T^d)} &\lesssim 2^{-n(s-1/2)} n^{d-1}\norm{f}_{\bB_{2,\infty}^s(\T^d)},
\end{align*}
whereas for $s=m$ we have
\begin{equation*}
\norm{f-P_nf}_{L_\infty(\T^d)} \lesssim 2^{-n(m-1/2)} n^{d-1}\norm{f}_{H^m_{\mix}(\T^d)}.
\end{equation*}
\end{theorem}
\begin{proof}
Using triangle inequality we obtain in case $1/2<s<m$ 
\begin{align}
\nonumber\norm{f-P_nf}_{L_\infty(\T^d)} 
&= \sup_{\vec x\in \T^d} \left|\sum_{|\vec j|_1>n}\sum_{\vec k\in \I_{\vec j}} \langle f, \psi_{\vec j,\vec k}^{*,\per}\rangle \psi_{\vec j,\vec k}^\per(\vec x)\right|\\
&\leq \sup_{\vec x\in \T^d} \left|\sum_{|\vec j|_1>n}  \left\|\left(\langle f, \psi_{\vec j,\vec k}^{*,\per}\rangle \psi_{\vec j,\vec k}^\per(\vec x)\right)_{\vec k\in \I_{\vec j}}\right\|_{\ell_1} \right|\,.\label{eq001}
\end{align}
Applying Cauchy-Schwarz-inequality yields
\begin{align*}
\eqref{eq001}&\leq \sup_{\vec x\in \T^d} \left|\sum_{|\vec j|_1>n}  \norm{(\langle f, \psi_{\vec j,\vec k}^{*,\per}\rangle)_{\vec k\in \I_{\vec j}}}_{\ell_2} \norm{( \psi_{\vec j,\vec k}^{\per}(\vec x))_{\vec k\in \I_{\vec j}}}_{\ell_2}\right|\\
&\leq \left|\sum_{|\vec j|_1>n} 2^{-|\vec j|_1s}2^{|\vec j|_1s}  \left(\sum_{\vec k\in \I_{\vec j}}|\langle f, \psi_{\vec j,\vec k}^{*,\per}\rangle|^2\right)^{1/2}\left(\sup_{\vec x\in \T^d} \sum_{\vec k\in \I_{\vec j }}|\psi_{\vec j,\vec k}(\vec x)|^2\right)^{1/2}\right|\,.
\end{align*}
Incorporating $|\psi_{\vec j,\vec k}(\vec x)|\lesssim 2^{j/2}$ implies 
\begin{align*}
\eqref{eq001}& \lesssim \left|\sum_{|\vec j|_1>n} 2^{-|\vec j|_1(s-1/2)}2^{|\vec j|_1s}  \left(\sum_{\vec k\in \I_{\vec j}}|\langle f, \psi_{\vec j,\vec k}^{*,\per}\rangle|^2\right)^{1/2} 
\right|\\
\intertext{and finally, with Hölder's inequality and Theorem~\ref{thm:norm-psi*},}
&\leq \left(\sum_{|\vec j|_1>n}  2^{-2|\vec j|_1(s-1/2)} \right)^{1/2}\norm{f}_{H^s_{\mix}(\T^d)}
\leq 2^{-n(s-1/2)}n^{(d-1)/2} \norm{f}_{H^s_{\mix}(\T^d)}\,.
\end{align*}
Note, that the last estimate boils down to estimate the sum, which has been already done in Corollary~\ref{cor:error_m=s}.

In the remaining case where $s = m$, we only have the weak characterization in Theorem~\ref{thm:d_charakpsi}. This gives a slightly worse bound for the $L_\infty$-error. Like in the previous estimates we have
 \begin{align*}
\norm{f-P_nf}_{L_\infty(\T^d)} 
&\lesssim \left|\sum_{|\vec j|_1>n} 2^{-|\vec j|_1(s-1/2)}2^{|\vec j|_1s}  \left(\sum_{\vec k\in \I_{\vec j}}|\langle f, \psi_{\vec j,\vec k}^{*,\per}\rangle|^2\right)^{1/2} 
\right|\\
\intertext{we extract the supremum in every summand,}
&\leq\left(\sum_{|\vec j|_1>n} 2^{-|\vec j|_1(s-1/2)}\right) \, \left( \sup_{|\vec j|>n}2^{|\vec j|_1s}  \left(\sum_{\vec k\in \I_{\vec j}}|\langle f, \psi_{\vec j,\vec k}^{*,\per}\rangle|^2\right)^{1/2} \right) \\
&\lesssim  2^{-n(s-1/2)} n^{d-1}\norm{f}_{H^s_{\mix}(\T^d)},
\end{align*}
where we bounded the last sum again by using \cite[Lemma 3.7]{BuGr04}. The estimation for the space $\bB_{2,\infty}^s(\T^d)$ follows similarly.
\end{proof}
Note that, using Lemma~\ref{lem:N} this Theorem can also be written in terms of the number of degrees of freedom $N$, which gives for $s<m$
\begin{align*}
\norm{f-P_nf}_{L_\infty(\T^d)} 
&\lesssim  N^{-s+1/2}\left(\log N\right)^{s(d-1)}\|f\|_{H^s_{\mix}(\T^d)},\\
\intertext{and for $s=m$}
\norm{f-P_nf}_{L_\infty(\T^d)} 
&\lesssim  N^{-m+1/2}\left(\log N\right)^{(m+1/2)(d-1)}\|f\|_{H^m_{\mix}(\T^d)}.\\
\end{align*}

\subsection{Tools from probability theory}\label{sec:prob}
In this subsection we collect some basic tools from probability theory, which we will apply later to our concrete settings.
Concentration inequalities describe how much a random variable spreads around the expectation value. 
One basic result about the spectral norm of sums of complex rank-$1$-matrices from~\cite[Theorem 1.1]{Tr12} is the following.
\begin{theorem}[Matrix Chernoff]\label{thm:matrix_chernoff}
Consider a finite sequence $\vec A_i\in \C^{N\times N}$ of independent, random, self-adjoint, positive definite matrices, 
where the eigenvalues satisfy $\mu_{\max}(\vec A_i)\leq R$ almost surely. Define 
$$\mu := \mu_{\min}\left(\sum_i\E(\vec A_i)\right) \quad \text{ and }\quad \tilde{\mu} := \mu_{\max}\left(\sum_i\E(\vec A_i)\right),$$ 
then 
\begin{align*}
\P\left(\mu_{\min}\left(\sum_i\vec A_i\right)\leq (1-\delta)\mu\right)&\leq N\left(\frac{\e^{-\delta}}{(1-\delta)^{1-\delta}}\right)^{\mu/R},\\
\P\left(\mu_{\max}\left(\sum_i\vec A_i\right)\geq (1+\delta)\tilde\mu\right)&\leq N\left(\frac{\e^{\delta}}{(1+\delta)^{1+\delta}}\right)^{\tilde\mu/R}.
\end{align*}
\end{theorem}
We will use this theorem in Theorem~\ref{thm:norm_MP} to prepare error estimates for the recovery of individual functions. 

Another basic inequality which we will use later is the Bernstein inequality, see~\cite[Theorem 6.12]{Stein08}.
\begin{theorem}\label{thm:bernstein}
Let $\P$ be a probability measure on $\T^d$, $B>0$ and $\sigma>0$ be real numbers and $M\geq 1$ be an integer. Furthermore, let $\xi_1,\ldots ,\xi_M:\T^d\to \R$ be independent random variables satisfying $\E\,\xi_i=0$, $\norm{\xi_i}_\infty\leq B$ and $\E\,\xi_i^2\leq \sigma^2$ for all $i=1,\ldots ,M$. Then we have 
\begin{equation*}
\P\left(\frac 1M \sum_{i=1}^M \xi_i \geq \sqrt{\frac{2\sigma^2\tau}{M}}+\frac{2B\tau}{3M}\right)\leq \e^{-\tau},\quad \quad \quad \tau>0.
\end{equation*}

\end{theorem}

\subsection{Hyperbolic wavelet regression}\label{sec:HWR}
So far we have bounded the error between a function and the approximation 
operator $P_nf$, defined in~\eqref{eq:defPn}, in Corollary~\ref{cor:error_bound}. Now we want to consider the case 
where we have random sample points $\vec x\in \X\subset \T^d$ with cardinality $|\X|=M$ together 
with the function values $\vec y = (f(\vec x))_{\vec x\in \X}$. In the sequel we always consider the case where the samples $\vec x\in \X$ are drawn i.i.d. at random according to the uniform Lebesgue measure. For that reason we introduce the $d$-dimensional probability measure $\d \P= \otimes_{i=1}^d \d x_i$. 
In this scenario we do not have the 
wavelet coefficients $\langle f, \psi_{\vec j,\vec k}^{\per,*}\rangle$ at hand.  
However, we study least squares solutions of the overdetermined system
\begin{equation}\label{eq:Aa=y}
\vec A\vec a=\vec y,
\end{equation}
where $\vec A=(\psi_{\vec j,\vec k}^\per(\vec x))_{\vec x\in \X,\stackrel{\vec j\in \J_n}{\vec k\in \I_\vec j}}\in \C^{M\times N}$ 
is the \textit{hyperbolic wavelet matrix} with $M>N$. At some point we will reduce the number of columns of the hyperbolic 
wavelet matrix $\vec A$. For that reason we will always denote the number of 
parameters, i.e. the number of columns of our wavelet matrix, by $N$.
In order to also minimize the $L_2(\T^d)$-error, we minimize the residual $\norm{\vec A\vec a-\vec y}_2$.  
Multiplying the system~\eqref{eq:Aa=y}
with $\vec A^*$ gives 
$$\vec A^*\vec A\vec a=\vec A^*\vec y.$$
If the hyperbolic wavelet matrix $\vec A$ has full rank, the unique solution of the least squares problem is
$$\vec a = \left(\vec A^*\vec A\right)^{-1}\vec A^*\vec y=:\vec A^+\vec y.$$
Computing these coefficients $\vec a$ gives us the wavelet coefficients of an approximation $S_n^\X f$ to $f$, i.e
\begin{equation}\label{eq:def_S_n^X}
S_n^\X f:=\sum_{\vec j\in \J_n}\sum_{\vec k\in \I_j}a_{\vec j,\vec k}\psi_{\vec j,\vec k}^\per. 
\end{equation}
To compute this approximant $S_n^\X f$ numerically, we have to ensure that the condition number of the 
matrix $\vec A^*\vec A$ is bounded away from zero. 
In order to apply Theorem~\ref{thm:matrix_chernoff} to our purposes, we use for $i=1,\ldots M$ and $\vec x_i\in \X$ the matrices 
\begin{equation*}\label{eq:def_A_i}
\vec A_i=\frac 1M \left(\left(\psi^\per_{\vec j,\vec k}(\vec x_i)\right)_{\stackrel{\vec j\in \J_n}{\vec k\in \I_\vec j}}\right)\left(\left(\psi^\per_{\vec j,\vec k}(\vec x_i)\right)_{\stackrel{\vec j\in \J_n}{\vec k\in \I_\vec j}}\right)^\top.
\end{equation*} 
Hence, we have $\sum_{i=1}^M \vec A_i=\frac 1M \vec A^*\vec A$. Additionally, these matrices fulfill the conditions 
in Theorem~\ref{thm:matrix_chernoff}. Since we will often consider the mass matrix 
\begin{equation}\label{eq:Lambda}
\vec \Lambda :=\frac 1M \,\E(\vec A^*\vec A),
\end{equation}
we have a closer look to its structure. In fact, the matrix $\vec \Lambda$ has entries 
$\langle\psi_{\vec j,\vec k}^\per,\psi_{\vec i,\vec \ell}^\per\rangle$, that are zero for $\vec j\neq \vec i$ 
because of the orthogonality of the one-dimensional wavelets for different scales $j$ and $i$, see~\eqref{eq:sca_psi_per}. We denote the entries of the matrix $\vec \Lambda$ for $\vec k\in \I_\vec j$ by
\begin{align*}\label{eq:def_lambda}
\lambda_{\vec j,\vec k}:&=\prod_{i=1}^{d} \langle\psi_{j_i, 0}^\per,\psi_{ j_i,k_i }^\per\rangle=\prod_{i=1}^{d }\int_{\T}\psi^\per_{j_i,0}(x_i)\psi^\per_{j_i,k_i}(x_i)\d x .
\end{align*}
Having a closer look at these entries, we see that there are only at most $ \left\lceil 2S-1\right\rceil$ ones of the $\lambda_{ j,k}$ non-zero 
for every one-dimensional index $j$. Additionally, these non-zero entries are the same for every index $ j$. Furthermore, the matrix $\vec \Lambda$ is symmetric.
Since the matrix $\vec \Lambda$ has the entries $\langle\psi_{\vec j,\vec k}^\per,\psi_{\vec i,\vec \ell}^\per\rangle$, 
it has a block structure and every block is dedicated to one index $\vec j$. Therefore we introduce the partial matrices 
\begin{equation}\label{eq:Lambda_j}
\vec \Lambda_{\vec j} := \left(\lambda_{\vec j,\vec k-\vec \ell}\right)_{\vec k\in \I_{\vec j},\vec \ell\in \I_{\vec j}} = \otimes_{i=1}^d \cir \left(\left(\lambda_{j_i,k_i}\right)_{k_i\in \I_{j_i}} \right),
\end{equation}
where the circulant matrices $\cir \vec y\in \C^{r\times r}$ are defined by $(\cir\vec y)_{i,j}=y_{i-j+1 \mod r}$ and $\otimes$ denotes the Kronecker product of matrices.
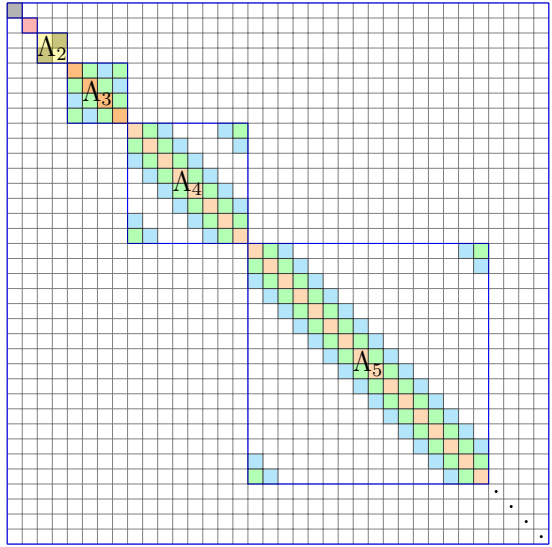
\begin{figure}[ht]\centering
\begin{tikzpicture} [scale=0.20]
\draw[step=1, gray, very thin] (0,0) grid (36,36);
\draw[blue] (0,0) -- (36,0);
\draw[blue] (0,0) -- (0,36);
\draw[blue] (36,0) -- (36,36);
\draw[blue] (0,36) -- (36,36);
\draw[blue] (0,35) -- (2,35);
\draw[blue] (1,36) -- (1,34);
\draw[blue] (2,35) -- (2,32);
\draw[blue] (1,34) -- (4,34);
\draw[blue] (4,34) -- (4,28);
\draw[blue] (2,32) -- (8,32);
\draw[blue] (8,32) -- (8,20);
\draw[blue] (4,28) -- (16,28);
\draw[blue] (16,28) -- (16,4);
\draw[blue] (8,20) -- (32,20);
\draw[blue] (32,20) -- (32,4);
\draw[blue] (16,4) -- (32,4);
\fill (32.5,3.5) circle (0.1);
\fill (33.5,2.5) circle (0.1);
\fill (34.5,1.5) circle (0.1);
\fill (35.5,0.5) circle (0.1);
\filldraw[fill=black,opacity=0.3] (0,35) rectangle (1,36);
\filldraw[fill=red,opacity=0.3] (1,34) rectangle (2,35);
\filldraw[fill=black,opacity=0.3] (2,32) rectangle (3,33);
\filldraw[fill=black,opacity=0.3] (3,33) rectangle (4,34);
\filldraw[fill=yellow,opacity=0.3] (2,32) rectangle (4,34);
\foreach \r in {4, 5,..., 31}
\filldraw[fill=orange,opacity=0.3] (\r,36-\r) rectangle (\r+1,36-\r-1);
\foreach \r in {4, 5,..., 7}
\filldraw[fill=orange,opacity=0.3] (\r,36-\r) rectangle (\r+1,36-\r-1);
\foreach \r in {5, 6,..., 7}
\filldraw[fill=green,opacity=0.3] (\r,37-\r) rectangle (\r+1,37-\r-1);
\foreach \r in {4, 5,..., 6}
\filldraw[fill=green,opacity=0.3] (\r,35-\r) rectangle (\r+1,35-\r-1);
\filldraw[fill=green,opacity=0.3] (4,28) rectangle (5,29);
\filldraw[fill=green,opacity=0.3] (7,31) rectangle (8,32);
\foreach \r in {9, 10,..., 15}
\filldraw[fill=green,opacity=0.3] (\r,37-\r) rectangle (\r+1,37-\r-1);
\foreach \r in {8, 9,..., 14}
\filldraw[fill=green,opacity=0.3] (\r,35-\r) rectangle (\r+1,35-\r-1);
\filldraw[fill=green,opacity=0.3] (8,20) rectangle (9,21);
\filldraw[fill=green,opacity=0.3] (15,27) rectangle (16,28);
\foreach \r in {17, 18,..., 31}
\filldraw[fill=green,opacity=0.3] (\r,37-\r) rectangle (\r+1,37-\r-1);
\foreach \r in {16, 17,..., 30}
\filldraw[fill=green,opacity=0.3] (\r,35-\r) rectangle (\r+1,35-\r-1);
\filldraw[fill=green,opacity=0.3] (16,4) rectangle (17,5);
\filldraw[fill=green,opacity=0.3] (31,19) rectangle (32,20);
\filldraw[fill=cyan,opacity=0.3] (4,29) rectangle (5,30);
\filldraw[fill=cyan,opacity=0.3] (5,28) rectangle (6,29);
\filldraw[fill=cyan,opacity=0.3] (6,31) rectangle (7,32);
\filldraw[fill=cyan,opacity=0.3] (7,30) rectangle (8,31);
\foreach \r in {10, 11,..., 15}
\filldraw[fill=cyan,opacity=0.3] (\r,38-\r) rectangle (\r+1,38-\r-1);
\foreach \r in {8, 9,..., 13}
\filldraw[fill=cyan,opacity=0.3] (\r,34-\r) rectangle (\r+1,34-\r-1);
\filldraw[fill=cyan,opacity=0.3] (9,20) rectangle (10,21);
\filldraw[fill=cyan,opacity=0.3] (8,21) rectangle (9,22);
\filldraw[fill=cyan,opacity=0.3] (14,27) rectangle (15,28);
\filldraw[fill=cyan,opacity=0.3] (15,26) rectangle (16,27);
\foreach \r in {18, 19,..., 31}
\filldraw[fill=cyan,opacity=0.3] (\r,38-\r) rectangle (\r+1,38-\r-1);
\foreach \r in {16, 17,..., 29}
\filldraw[fill=cyan,opacity=0.3] (\r,34-\r) rectangle (\r+1,34-\r-1);
\filldraw[fill=cyan,opacity=0.3] (17,4) rectangle (18,5);
\filldraw[fill=cyan,opacity=0.3] (16,5) rectangle (17,6);
\filldraw[fill=cyan,opacity=0.3] (30,19) rectangle (31,20);
\filldraw[fill=cyan,opacity=0.3] (31,18) rectangle (32,19);

\node[scale=0.9] at (3,33) {$\Lambda_{2}$};
\node[scale=0.9] at (6,30) {$\Lambda_{3}$};
\node[scale=0.9] at (12,24) {$\Lambda_{4}$};
\node[scale=0.9] at (24,12) {$\Lambda_{5}$};
\end{tikzpicture} \caption{Illustration of the non-zero entries of the matrix $\vec \Lambda$ for $d=1$.}
\label{fig:Lambda}
\end{figure}
Figure~\ref{fig:Lambda} shows the structure of $\vec \Lambda$ for $d=1$ and the Chui-Wang wavelets of order $m=2$, which 
mean, in every column of every block there are at most $4m-1$ non-zero entries, only $2m-1$ of them are different. Equal 
colors in the picture stand for equal matrix entries. In higher dimensions we get one block for every index $\vec j$, which 
is the Kronecker product of the one-dimensional circulant matrices.

Having these only few different entries of $\vec \Lambda$ in mind, we can bound the lowest eigenvalue of this matrix away from zero. 
\begin{lemma}\label{lem:bound_eigenvalues}
Let the matrix $\vec \Lambda$ be defined like in~\eqref{eq:Lambda}. Then we can bound the eigenvalues of this matrix by
\begin{align}
\mu_{\min}(\vec \Lambda)&\geq \gamma_m^d, \label{eq:lambda_min}\\
\mu_{\max}(\vec \Lambda)&\leq \max\{1,\delta_m^d\}, \label{eq:lambda_max}
\end{align}
where the constants $\gamma_m$ and $\delta_m$ are the Riesz bounds from~\eqref{eq:Riesz}.
\end{lemma}
\begin{proof}
As usual, we begin with the one-dimensional case. As mentioned before this lemma, this matrix has only few non-zero entries. In fact, for $d=1$ 
we have a block-diagonal matrix with blocks belonging to every $j$, see also \ref{fig:Lambda}. To be precise, the blocks are circulant matrices
$$ \cir \left((\lambda_{ j,k})_{k\in I_j}\right)=\cir (\vec \lambda_j).$$
For the case $j = -1$ this is only $1$, which is also the eigenvalue. For $j\geq 0$ we use \cite[Theorem 3.31]{PlPoStTa18} and we write a circulant matrix as
\begin{equation*}\label{eq:circ}
\cir (\vec \lambda_j) = \vec F_{2^j}^{-1} \diag(\vec F_{2^j}(\vec \lambda_j))\vec F_{2^j},
\end{equation*}
where $\vec F_{2^j}=\left(w_{2^j}^{k\ell}\right)_{k,\ell=0}^{2^j}$ is the Fourier matrix of dimension $2^j$ with the 
primitive $2^j$-th roots of unity $w_{2^j}:=\e^{-2\pi \im 2^{-j}}$. In order to bound the 
eigenvalues of the matrix $\vec \Lambda$, we have to determine the infimum and supremum of all eigenvalues of all blocks $\vec \Lambda_j$. 
Since the Fourier matrices are orthogonal, we have
$$\mu_{\min}(\cir (\vec \Psi_j))=\min |\vec F_{2^j}(\vec \lambda_{j})|,$$
and analog for the maximum. To bound this term we calculate the Fourier coefficients (see \eqref{eq:Fourier}) of the wavelets $\psi_{j,k}(x)$ using 
substitution in the integral by
\begin{align*}
c_\ell(\psi_{j,k}) 
&= 2^{-\tfrac{j}{2}}w_{2^j}^{\ell k}c_{\tfrac{\ell}{2^j}}(\psi).
\end{align*}
We begin with the case where $j$ is big enough, such that $2^j>S$, so that $\vec \lambda_j=\left(\langle\psi_{j,0},\psi_{j,k}\rangle\right)_{k\in \I_j}$. Therefore we get, using Parsevals' equality,
\begin{align*}
\langle\psi_{j,0},\psi_{j,k}\rangle 
&= \sum_{\ell\in \Z} c_\ell(\psi_{j,0}) c_{-\ell}(\psi_{j,k})
=\sum_{\ell\in \Z}  w_{2^j}^{-\ell k}2^{-j}c_{\tfrac{\ell}{2^j}}(\psi) c_{\tfrac{-\ell}{2^j}}(\psi)
= \sum_{r=0}^{2^j-1} w_{2^j}^{-r k}2^{-j}\sum_{s\in\Z} \left|c_{\tfrac{r}{2^j}+s}(\psi)\right|^2.
\end{align*}
We denote 
$$E(w^r_{2^j})=\sum_{s\in\Z} \left|c_{\tfrac{r}{2^j}+s}(\psi)\right|^2.$$
Hence, 
\begin{align*}
(\vec F_{2^j}(\vec \lambda_{j}))_i
=\sum_{k=0}^{2^j-1}w_{2^j}^{ik}\langle\psi_{j,0},\psi_{j,k}\rangle
=\sum_{k=0}^{2^j-1}\sum_{r=0}^{2^j-1} w_{2^j}^{ik}w_{2^j}^{-r k}2^{-j}E(w^r_{2^j})
=E(w^i_{2^j}).
\end{align*}
For the other case, where $2^j\leq S$, we get the same estimates in a similar way.
Therefore the eigenvalues of our block of the desired matrix $\vec \Lambda$ are bounded by
$$\min_{r}E(w^r_{2^j})\leq\mu_{\min}(\vec \Lambda_j)\leq\mu_{\max}(\vec \Lambda_j)\leq \sup_{r}E(w^r_{2^j}).$$
These extreme values coincide with the Riesz constants, which can be seen as follows
\begin{align*}
\left\|\sum_{k=0}^{2^j-1}d_{j,k}\psi_{j,k}\right\|_{L_2(\T)}^2
&=\sum_{\ell\in \Z}\sum_{k_1,k_2=0}^{2^j-1}d_{j,k_1}\,d_{j,k_2}c_{\ell}(\psi_{j,k_1})c_{-\ell}(\psi_{j,k_2})\\
&= \sum_{\ell\in \Z}\sum_{k_1,k_2=0}^{2^j-1}d_{j,k_1}\,d_{j,k_2}w^{\ell(k_1-k_2)}2^{-j}|c_{\tfrac{\ell}{2^j}}(\psi)|^2\\
&=\sum_{r=0}^{2^j-1}\sum_{k_1,k_2=0}^{2^j-1}d_{j,k_1}\,d_{j,k_2}w^{r(k_1-k_2)}2^{-j}\sum_{s\in \Z}|c_{\tfrac{r}{2^j}+s}(\psi)|^2\\
&=\sum_{r=0}^{2^j-1}\left|\hat{\vec d}_r\right|^2 \, 2^{-j} \, E(w^r_{2^j}),
\end{align*}
where $\hat{\vec d}_r$ is the $r$-th component of $\hat{\vec d}=\vec F_{2^j}\vec d=\vec F_{2^j}(d_{j,k})_{k\in \I_j}.$ Taking into account that $\norm{\hat{\vec d}}_{2}^2=2^j\norm{\vec d}_2^2$, the one-dimensional assertion follows. 

To generalize this to the multi-dimensional case, we have a closer look at the matrix $\vec \Lambda$ for $d>1$. 
Again we have a block diagonal matrix, because of the orthogonality of the wavelets $\psi_{\vec j,\vec k}^\per$ for 
different scales $\vec j$. So according to every $\vec j\geq -\vec 1$, we have a block in the matrix $\vec \Lambda$. 
Because of the tensor product form of our wavelet functions $\psi_{\vec j,\vec k}^\per$, we order the functions 
$\psi_{\vec j,\vec k}^\per$ in the matrix $\vec A$ such that the block belonging to $\vec j$ is equal to the Kronecker product 
$$\otimes_{i=1}^{d} (\cir \vec \lambda_{j_i}).$$  
Since the eigenvalues of the Kronecker product of a matrix are the products of the eigenvalues of the matrices, 
we can bound the smallest eigenvalue of every block matrix by $\gamma_m^d$ and the largest eigenvalues by $\delta_m^d$. 
The eigenvalue $\mu =1$ is explained by the first block for $\vec j= -\vec 1$, which is basically $1$.
\end{proof}
Let us again consider the example of the Chui-Wang wavelets from Example~\ref{ex:Chui-Wang}. The function $E$ 
in the previous proof is in this case the Euler-Frobenius polynomial $\Psi_{2m}$ from $\cite{PlTa94}$. From there 
we also get the Riesz-constants, which are summarized in Table~\ref{tab:Riesz-const}.

The previous Lemma gives us one constant in Theorem~\ref{thm:matrix_chernoff}. For the other constant $R$ let us introduce the spectral function
\begin{align}\label{eq:R(n)}
R(n):&=\sup_{\vec x\in \T^d}\sum_{\vec j\in \J_n}\sum_{\vec k\in \I_\vec j} |\psi_{\vec j, \vec k}^\per(\vec x)|^2.
\end{align}
In order to give an estimation of the complexity of $R(n)$ we denote for every $\vec j \geq -\vec 1$ the subset 
of indices $\vec u=\{i\in [d]\mid j_i\geq 0\}$. Hence, there holds
\begin{align}\label{eq:RN_bound}
R(n)&=\sup_{\vec x\in \T^d}\sum_{\vec j\in \J_n}\sum_{\vec k\in \I_\vec j}2^{|\vec j_\vec u|_1} |\sum_{\ell_\vec u\in \Z^{|\vec u|}}\psi(2^{\vec j_\vec u}(\vec x_\vec u+\vec \ell_\vec u)-\vec k_\vec u)|^2\notag\\
&\leq\sum_{\vec j\in \J_n}2^{|\vec j|_1} \sup_{\vec x\in \T^d}\sum_{\vec k_\vec u\in \Z^{|\vec u|}} |\psi(2^{\vec j_\vec u}\vec x_\vec u-\vec k_\vec u)|^2\notag\\
&=\sum_{\vec j\in \J_n} 2^{|\vec j|_1}\prod_{i=1}^{|\vec u|} \left(\sup_{x_i\in \T}\sum_{ k_i \in \Z} |\psi(2^{j_i} x_i-k_i)|^2\right)\notag\\
&=\sum_{\vec j\in \J_n} 2^{|\vec j|_1}\prod_{i=1}^{|\vec u|} \left(\sup_{y_i\in [-2^{j_i-1},2^{j_i-1})}\sum_{ k_i \in \Z} |\psi(y_i-k_i)|^2\right)\notag\\
&\leq  N \left(\sup_{x\in \R}\sum_{ k \in \Z} |\psi(x-k)|^2\right)^d=: N c_{\psi}^d,
\end{align}
where we use \eqref{eq:Nsum}.  
The supremum of $\sum_{ k \in \Z} |\psi(x-k)|^2$ is a constant, since the wavelet $\psi$ is compactly supported on $[0,S]$. 
In Table~\ref{tab:constants} we calculated these constants for the Chui-Wang wavelets of different orders.
\begin{table}[htb]\centering
\begin{tabular}{c|ccccc}
    \hline
    $m$   & $1$ &$2$&$3$&$4$ &$5$\\
    \hline
		$c_{\psi}$ & $1$ &$0.7083$	&$0.1479$ &$0.0662$&$0.0252$\\
		\hline

  \end{tabular}
\caption{Constants in~\eqref{eq:RN_bound} for the Chui-Wang wavelets.}
\label{tab:constants}
\end{table}

Now we are in the position to apply Theorem~\ref{thm:matrix_chernoff} for our setting.
\begin{theorem}\label{thm:norm_MP}
Let $\vec x\in \X$ drawn i.i.d. and uniformly at random, the wavelet function $\psi$ having vanishing moments of order $m$, 
$r>1$ and $\gamma_m$ the Riesz constant from~\eqref{eq:Riesz}. 
Then the matrix $\frac 1M \vec A^*\vec A$, where $\vec A$ is the hyperbolic wavelet matrix from~\eqref{eq:Aa=y},   
only has eigenvalues greater than $\tfrac{ \gamma_m^d}{2}$ with probability at least $1-M^{-r}$ if 
\begin{equation}\label{eq:RN_leq}
R(n)\leq \frac{ c_{m,d}\,M}{(r+1)\log M},
\end{equation}
with the $m$- and $d$-dependent constant 
\begin{equation}\label{eq:cm}
c_{m,d} = \gamma_m^d\log\left(\sqrt{\tfrac{\e}{2}}\right)\approx 0.153\, \gamma_m^d.
\end{equation}
Especially, we have for the operator norm 
\begin{equation}\label{eq:norm_MP}
\norm{\vec A^+}_2\leq \sqrt{\frac{ 2}{ M\,\gamma_m^d}}.
\end{equation}
\end{theorem}

\begin{proof}
We have that 
$$\mu_{\max}(\vec A_i)= \frac 1M \left\|\left(\psi^\per_{\vec j,\vec k}(\vec x_i)\right)_{\stackrel{\vec j\in \J_n}{\vec k\in \I_\vec j}}\right\|_2^2\leq \frac{R(n)}{M},\quad \text{ for all } i = 1,\ldots, M.$$
Hence, we use $R= \frac{R(n)}{M}$, $\mu_{\min}=\gamma_m^d$ and $\delta = \tfrac{1}{2}$ in Theorem~\ref{thm:matrix_chernoff}.
Using the bound in \eqref{eq:RN_leq} and $N\leq M$, we confirm 
\begin{align*}
R(n)&\leq \frac{\gamma_m^d\, M \log\left(\sqrt{\tfrac{\e}{2} }\right)}{(r+1)\log M}
= c_{m,d}\,\frac{M}{(r+1)\log M}\\
R(n) &\leq \frac{-\gamma_m^d\, M \log\left(\sqrt{\tfrac{2}{\e}}\right)}{r\log M+\log N}\\
\log N + \frac{\gamma_m^d M}{R(n)} \log\left(\sqrt{\tfrac{2}{\e} }\right)&\leq -r\log M\\
N\left(\frac{\e^{-\delta}}{(1-\delta)^{1-\delta}}\right)^{\frac{\gamma_m^d M}{R(n)}}&\leq\frac{1}{M^r}.
\end{align*}
Therefore, it follows that $\P\left(\mu_{\min}\left(\vec A^*\vec A\right)\leq \frac{M\,\gamma_m^d}{2} \right)\leq \tfrac{1}{M^r}$.
Hence, $\vec A$ has singular values at least $\sqrt{\frac{ 2\,\gamma_m^d}{M}}$ with high probability. This yields an upper bound for 
the norm of the Moore-Penrose-inverse $\norm{\vec A^+}_2=\norm{(\vec A^*\vec A)^{-1}\vec A^*}_2$ by using Proposition 3.1 in \cite{KaUlVo19}, i.e. \eqref{eq:norm_MP} follows.
\end{proof}
It remains to estimate the number of samples $M$, such that \eqref{eq:RN_leq} is fulfilled. In \eqref{eq:RN_bound} we 
estimated the complexity of the spectral function $R(n)$, hence we have to require $N\leq \frac{c_{m,d}\, M}{c_{\psi}^d\,(r+1)\,\log M}$, 
which yields that
\begin{equation}\label{eq:bound_M}
M \geq \frac{c_{\psi}^d\,(r+1)}{c_{m,d}}\, N\log N,
\end{equation}
where $c_{m,d}$ is the constant from \eqref{eq:cm}. We receive an estimation of the number of samples $M$ in terms of the number of parameters $N$. 
In order to recover individual functions, we get a bound of the individual error 
$\norm{f-S_n^\X f}_{L_2(\T^d)}$ with high probability.

\begin{theorem}\label{thm:one_f_new}
Let $M$ be the number of samples satisfying \eqref{eq:bound_M}, $(\vec x_i)_{i=1}^M$ drawn i.i.d. and uniformly at random, $r>1$ and $f \in C(\T^d)$ a continuous function. Then 
\begin{equation*}
\P\left(\norm{f-S_n^\X f}_{L_2(\T^d)}^2 \leq e_2^2 + \tfrac{2}{\gamma_m^d}\,\left(e_2^2+e_2e_{\infty}\sqrt{\tfrac{r\,\log M}{M}}+ e_{\infty}^2\tfrac{r\,\log M}{M}\right)\right)
\geq 1- 2\,M^{-r},
\end{equation*}
where we define $e_2:=\norm{f-P_nf}_{L_2(\T^d)}$ and $e_\infty:=\norm{f-P_nf}_{L_\infty(\T^d)}$.
That means, the $L_2(\T^d)$-error of our approximation can be bounded with high probability by rates of the $L_2(\T^d)$-and the $L_\infty(\T^d)$-error of the projection $P_n$.
\end{theorem}
\begin{proof}
Using the orthogonality of $\psi_{\vec j,\vec k}^\per$ and $\psi_{\vec i,\vec \ell}^\per$ for $\vec j\neq \vec i$ we have 
\begin{align}\label{eq:22}
\norm{f-S_n^\X f}_{L_2(\T^d)}^2
&=e_2^2+\norm{P_n f-S_n^\X f}_{L_2(\T^d)}^2
=e_2^2+\norm{S_n^\X (P_n f- f)}_{L_2(\T^d)}^2\notag\\
&\leq e_2^2+\norm{S_n^\X}_2^2 \norm{P_n f- f}_{\ell_2(\X)}^2.
\end{align}

We apply Theorem~\ref{thm:norm_MP} to bound the operator norm $\norm{S_n^\X}_2$.
For the $\ell_2$-norm we give a bound with high probability by using Bernstein inequality. Therefore we introduce the random variables 
$$\xi_i = |f(\vec x_i)-P_nf(\vec x_i)|^2- e_2^2= \eta_i-\E(\eta_i),$$
where $\eta_i =|f(\vec x_i)-P_nf(\vec x_i)|^2$. These random variables are centered, i.e. $\E(\xi_i)=0$. The variances of these random 
variables can be bounded by
\begin{align*}
\E(\xi_i^2) &= \E(\eta_i^2)- \E(\eta_i)^2 = \int |f(\vec x_i)-P_nf(\vec x_i)|^4\d \P - \left(\int |f(\vec x_i)-P_nf(\vec x_i)|^2\d \P\right)^2\\
&\leq \norm{f-P_nf}_{L_\infty(\T^d)}^2\int |f(\vec x_i)-P_nf(\vec x_i)|^2\d \P- \left(\int |f(\vec x_i)-P_nf(\vec x_i)|^2\d \P\right)^2\\
&\leq \left(\norm{f-P_nf}_{L_\infty(\T^d)}^2- \int |f(\vec x_i)-P_nf(\vec x_i)|^2\d \P\right) \norm{f-P_nf}_{L_2(\T^d)}^2\\
&\leq \norm{f-P_nf}_{L_\infty(\T^d)}^2\norm{f-P_nf}_{L_2(\T^d)}^2=e_2^2\,e_\infty^2.
\end{align*}
Furthermore, we have
\begin{align*}
\norm{\xi_i}_\infty = \norm{\eta_i-\E_\P(\eta_i)}_\infty \leq \sup_{\vec x\in \T^d}\left||f(\vec x)-P_nf(\vec x)|^2-\norm{f-P_nf}^2_{L_2(\T^d)}\right|\leq e_\infty^2.
\end{align*} 
For the last estimation we used that for positive $y_1,y_2$ it holds $|y_1-y_2|\leq \max\{y_1,y_2\}$.

Now we are in the position to merge all inequalities in order to apply Bernstein’s inequality from Theorem~\ref{thm:bernstein}. Using $\tau = r\log M$, this yields
\begin{align*}
\P\left(\tfrac 1M \sum_{i=1}^M\xi_i \geq \sqrt{\tfrac{2\, e_\infty^2e_2^2 r\log M}{M} } + \tfrac{2\,e_\infty^2 r \log M}{3M}\right)&\leq M^{-r}.
\end{align*}  
Because of our choice for the random variables $\xi_i$, we have that 
$$\sum_{i=1}^M\xi_i +e_2^2=\sum_{i=1}^M\eta_i=\norm{f(\vec x_i)-P_nf(\vec x_i)}_{\ell_2(\X)}^2.$$
Hence, we add the mean $e_2^2$ and get
\begin{equation}\label{eq:23}
\P\left(\norm{P_n f- f}_{\ell_2(\X)}^2\leq M\left(e_2^2 + \sqrt{\tfrac{2\, e_\infty^2e_2^2 r\log M}{M} } + \tfrac{2\,e_\infty^2 r \log M}{3M} \right) \right)\geq 1-M^{-r}.
\end{equation}
The terms in~\eqref{eq:22} are bounded with high probability. Let us define the events
\begin{align*}
A&:=\left\{\X\left\vert \norm{f(\vec x_i)-P_nf(\vec x_i)}_{\ell_2(\X)}^2\leq M\left(e_2^2 + \sqrt{\tfrac{2\, e_\infty^2e_2^2 r\log M}{M} } + \tfrac{2\,e_{\infty}^2 r \log M}{3M}\right)\right.\right\},\\
B&:=\left\{\X\left\vert  \norm{S_n^\X}_2\leq \sqrt{\tfrac{2}{M\gamma_m^d}}\right.\right\}.
\end{align*}
Due to \eqref{eq:23} and Theorem~\ref{thm:norm_MP} we know that 
\begin{align*}
\P(A)>1-M^{-r} \text{ and } \P(B)>1-M^{-r}, 
\end{align*}
which implies that 
$$ \P(A\cap B )\geq 1- \P(A^c)-\P(B^c)\geq 1-2M^{-r}.$$
This gives us the assertion.
\end{proof}

In a similar way we give an estimation for the $L_\infty$-error.
\begin{theorem}\label{thm:one_f_new_infty}
Let $M$ be the number of samples satisfying~\eqref{eq:bound_M}, $(\vec x_i)_{i=1}^M$ drawn i.i.d. and uniformly at random, $r>1$ and $f \in C(\T^d)$ a continuous function. Then 
\begin{align*}
\P\left(\norm{f-S_n^\X f}_{L_\infty(\T^d)} \leq e_\infty +      \left(\tfrac{2 c_{m,d}\delta_m^d}{(r+1)\,\gamma_m^d}\right)^{1/2} \left(e_2^2 \tfrac{M}{\log M} + e_2e_\infty \sqrt{\tfrac{r\, M}{\log M}} + r\,e_\infty^2 \right)^{1/2}\right)
&\geq 1- 2\,M^{-r},
\end{align*}
where we define as in the previous theorem $e_2:=\norm{f-P_nf}_{L_2(\T^d)}$ and $e_\infty:=\norm{f-P_nf}_{L_\infty(\T^d)}$.
\end{theorem}
\begin{proof}
This proof is similar to the proof of the previous theorem. Triangle inequality gives
\begin{equation*}
\norm{f-S_n^\X f}_{L_\infty(\T^d)}
=e_\infty+\norm{P_n f-S_n^\X f}_{L_\infty(\T^d)}.
\end{equation*}
We denote the function $g =P_n f-S_n^\X f  = \sum_{\vec j\in \J_n}\sum_{\vec k\in \I_{\vec j} }\langle g, \psi_{\vec j,\vec k}^{\per,*}\rangle  \psi_{\vec j,\vec k}^{\per}$, which gives 
\begin{align*}
|g(\vec x)| &\leq \left(\sum_{\vec j\in \J_n}\sum_{\vec k\in \I_{\vec j}} |\langle g, \psi_{\vec j,\vec k}^{\per,*}\rangle|^2 \right)^{1/2} \left(\sum_{\vec j\in \J_n}\sum_{\vec k\in \I_{\vec j}} |\psi_{\vec j,\vec k}(\vec x)|^2\right)^{1/2}\\
&\leq \delta_m^{d/2}\norm{P_nf -S_n^\X f}_{L_2(\T^d)} \sqrt{R(n)}.
\end{align*}
The condition~\eqref{eq:RN_leq} gives a bound for $R(n)$ and the estimation of $\norm{P_nf -S_n^\X f}_{L_2(\T^d)} $ follows the same lines as the proof of Theorem~\ref{thm:one_f_new}.
\end{proof}
We use the general Theorem~\ref{thm:one_f_new} to give a bound for the approximation error with high probability 
for our settings, where $s = m$ and where $s<m$, using our estimates for the errors of $\norm{f-P_nf}$ for both cases.

\begin{corollary}\label{cor:one_f}
Let the assumptions be like in Theorem~\ref{thm:one_f_new}. Let $m$ be the order of vanishing moments of the wavelets, the 
number of samples satisfying~\eqref{eq:bound_M}, and $\gamma_m, \delta_m$ the Riesz constants from~\eqref{eq:Riesz}. In the case where $1/2<s<m$ we have
\begin{equation}
\P\left(\norm{f-S_n^\X f}_{L_2(\T^d)}^2 \lesssim  (1+\tfrac{2}{\gamma_m^d}(r+\sqrt r+1))\, 2^{-2ns}n^{d-1}\norm{f}^2_{\bB^{s}_{2,\infty}(\T^d)}\right)\geq 1- 2\,M^{-r},
\end{equation}
and in the case where $s = m$ we have
\begin{equation}
\P\left(\norm{f-S_n^\X f}_{L_2(\T^d)}^2 \lesssim  (1+\tfrac{2}{\gamma_m^d}(r+\sqrt r+1)) \,2^{-2nm}n^{d-1}\norm{f}^2_{H^m_{\mix}(\T^d)}\right)\geq 1- 2\,M^{-r}.
\end{equation}

\end{corollary}
\begin{proof}
In order to apply the previous theorem, let us collect bounds for the occurring terms. We have for the number of samples 
$\frac{\log M}{M} \lesssim 2^{-n} n^{-d+1}$, see~\eqref{eq:bound_M}. The $L_2$-error is bounded in Corollary~\ref{cor:error_bound} respectively 
Corollary~\ref{cor:error_m=s} and the subsequent remark.
Theorem~\ref{thm:L_infty} gives us a bound for the ${L_\infty}$-error.
Hence, we have in the case $s<m$,
\begin{align*}
\norm{f-P_nf}_{L_\infty(\T^d)}\norm{f-P_nf}_{L_2(\T^d)}\sqrt{\tfrac{r\,\log M}{M}}&\lesssim 
\sqrt{r}\, 2^{-2ns}n^{d-1}\norm{f}^2_{\bB^s_{2}(\T^d)}, \\
\norm{f-P_nf}^2_{L_\infty(\T^d)} \tfrac{r\,\log M}{M}&\lesssim 
r\, 2^{-2ns}n^{d-1}\norm{f}^2_{\bB^s_{2,\infty}(\T^d)},
\end{align*}
and for $s=m$ 
\begin{align*}
\norm{f-P_nf}_{L_\infty(\T^d)}\norm{f-P_nf}_{L_2(\T^d)}\sqrt{\tfrac{r\,\log M}{M}}&\lesssim \sqrt{r}\,2^{-2nm} n^{d-1}\norm{f}^2_{H^m_{\mix}(\T^d)}, \\
\norm{f-P_nf}^2_{L_\infty(\T^d)} \tfrac{r\,\log M}{M}&\lesssim r\, 2^{-2nm} n^{d-1}\norm{f}^2_{H^m_{\mix}(\T^d)}.
\end{align*} 
Theorem~\ref{thm:one_f_new} gives the assertion.                                     
\end{proof}

\begin{remark}
Note, that this theorem establishes a bound for the error of the least squares approximation with high probability , which has the same 
rate like the best approximation with the projection operator $P_n$ in Corollarys~\ref{cor:error_bound} and \ref{cor:error_m=s}. 
The projection operator $P_n$ is the optimal approximation in the wavelet spaces. Hence, with high probability we also get 
this optimal rate using the operator $S_n^\X$. 
Furthermore, also the $L_\infty$-error of $S_n^\X$ allows such an optimal bound by applying Theorem~\ref{thm:one_f_new_infty}. 
Note, in this case choosing the sampling number $M$ according to~\eqref{eq:RN_leq}, gives that $$
    \frac{\log M}{M}\sup\limits_{\|f\|\leq 1}\|f-P_nf\|^2_\infty  \asymp
    \frac{1}{N}\sup\limits_{\|f\|\leq 1}\|f-P_nf\|^2_\infty  \asymp
    \sup\limits_{\|f\|\leq 1}\|f-P_nf\|^2_2\,.
$$
\end{remark}

All the theoretical considerations in this section result in Algorithm~\ref{alg:1}, which determines the approximant~\eqref{eq:def_S_n^X} 
from given samples $\X$ by solving a least squares algorithm with the hyperbolic wavelet matrix $\vec A$.

\begin{algorithm}[ht]
\caption{Hyperbolic wavelet regression}
	\vspace{2mm}
	\begin{tabular}{ l l l }
		\textbf{Input:} & $n \in \N$ & maximal level \\
		&$M\in\N\text{ with }    \tfrac{M}{\log M} \gtrsim 2^n n^{d-1}$& number of samples\\
		&		$\X = (x_i)_{i=1}^M\in \T^d$ & sampling nodes \\
		& $\vec y = (f(x_i))_{i=1}^M$ & function values at sampling nodes 
	\end{tabular}
	\begin{algorithmic}[1]
			\STATE{Construct the sparse hyperbolic wavelet matrix 
			$$\vec A=(\psi_{\vec j,\vec k}^\per(\vec x))_{\vec x\in \X,\stackrel{\vec j\in \J_n}{\vec k\in \I_\vec j}}\in \C^{M\times N},$$
			where the number $N$ of parameters is 
			$N := \text{dim} \lin \{\psi_{\vec j,\vec k}^\per \mid \vec j\in \J_n,\vec k\in \I_j\}$}
	    \STATE{Minimize the residual 
			$\norm{\vec A\vec a-\vec y}$
			via an LSQR-algorithm.
			}
			
	\end{algorithmic}
	\begin{tabular}{ l l l }
		\textbf{Output:} 
		 &  $\left(a_{\vec j,\vec k}\right)_{\vec j,\vec k}\in \C^N$ &coefficients of the approximant $S_n^\X f =\sum_{\vec j\in \J_n}\sum_{\vec k\in \I_j}a_{\vec j,\vec k}\psi_{\vec j,\vec k}^\per$
	\end{tabular}
	\label{alg:1}

\end{algorithm}

\subsubsection*{Comparison to other work}
Error estimates with piecewise linear wavelet functions, i.e. the case $m=2$ are considered in \cite{Bo17}, where the wavelets are called ``prewavelets'' and the non-periodic setting is treated. To compare, in \cite[p. 117]{Bo17} 
an approximation rate $2^{-ns}n^{(d-1)/2}$ was proven for $s\leq m$ in case of $H^s_{\mix}(\T^d)$-functions. This is due to the sub-optimal analysis of the projection operator. We obtain the same bound for the larger space $\bB^s_{2,\infty}(\T^d)$ additionally with high probability. In the case $s=m$ our results match the results in \cite{Bo17}. Furthermore, we consider quasi-optimal $L_\infty(\T^d)$-error bounds with high probability.  

In~\cite{CoMi17} the authors also studied the approximation error of a least squares operator like $S_n^\X$. But they used an 
orthonormal system of basis functions and bounded expectation of the approximation error $\norm{f-S_n^\X f}_{L_2(\T^d)}$. A recent improvement was done in~\cite{CoDo21}.
In contrast to that, we give in Corollary~\ref{cor:one_f} a concentration inequality for the approximation error based on the probabilistic Bernstein inequality. 

In~\cite{KaUlVo19} also hyperbolic wavelet regression was considered. In contrast to our work they studied the worst-case setting for the whole function class.

 \section{Computing the ANOVA decomposition}\label{sec:wav_anova}
The hyperbolic wavelet regression in Algorithm~\ref{alg:1} already reduces the curse of dimensionality because of the 
hyperbolic structure of our index set. This is reflected in the number of parameters of the wavelet spaces $N=\O(2^nn^d)$.
But we still have the dimension $d$ in the exponent, which grows fast for high dimensions $d$. Therefore we want to 
reduce the number of necessary parameters further by taking into account which variable interactions play a role for describing the function. 

In this section we calculate the global sensitivity indices $\rho(\vec u,S_n^\X f)$ defined in~\eqref{eq:def_gsi} for the approximated functions 
$S_n^\X f$, which we introduced in the last section. Knowing these indices $\rho(\vec u,S_n^\X f)$, we can reduce the number of parameters by 
omitting the ANOVA terms which do not play a role in describing the variance of a function $f\in L_2(\T^d)$.
Solving the least-squares problem
$$\min_{\vec a\in \C^N}\norm{\vec A \vec a-\vec y}_2,$$
where the hyperbolic wavelet matrix $\vec A\in \R^{M\times N}$ has the form 
$\vec A=(\psi_{\vec j,\vec k}^\per(\vec x))_{\vec x\in \X,\stackrel{\vec j\in \J_n}{\vec k\in \I_\vec j}}$ ,
leads to a coefficient vector $\vec a=(a_{\vec j,\vec k})_{\vec j\in \J_n,\vec k\in \I_\vec j}$, which describes the approximant by
$$g=S_n^\X f=\sum_{\vec j\in \J_n}\sum_{\vec k\in \I_{\vec j}}a_{\vec j,\vec k}\psi_{\vec j,\vec k}^\per.$$
We calculate the global sensitivity indices $\rho(\vec u,g)$
in terms of the coefficient vector $\vec a$.
\begin{theorem}\label{thm:anova_psi}
Let $g\in L_2(\T^d)$ be a function in the periodic wavelet space written as 
\begin{equation}\label{eq:g=apsi} 
g(\vec x)= \sum_{\vec j\geq -\vec 1}\sum_{\vec k\in \I_\vec j} a_{\vec j,\vec k} \psi_{\vec j,\vec k}^\per(\vec x).
\end{equation}
For a function $f\in H^s_\mix(\T^d)$ let $g$ be the projection $g = P_nf$ defined in~\eqref{eq:defPn} 
or the approximation $g = S_n^\X f$ from~\eqref{eq:def_S_n^X}. In these cases the sum for the index $\vec j$ reduces to $\vec j\in \J_n$.
Furthermore, let $g$ be decomposed in ANOVA terms $g_{\vec u}$ as in Definition~\ref{def:anova-terms}. For these terms yields
\begin{equation*}
g_{\vec u}(\vec x_{\vec u})=\sum_{\vec j_\vec u\geq\vec 0}\sum_{\vec k\in \I_{\uinv{j}{u}}}a_{\uinv{j}{u},\vec k} \,\psi^\per_{\vec j_\vec u,\vec k_\vec u}(\vec x_\vec u),
\end{equation*}
where we define the notion for up-sampling $\uparrow\colon \R^{|\vec u|}\to \R^d$, given by
$$(\uinv{j}{u})_i=\begin{cases}j_i&\text{ if } i\in \vec u,\\ -1 &\text{ otherwise.}\end{cases}$$
Then the variances $\sigma^2(g_\vec u)$ of these ANOVA terms for $\varnothing\neq\vec u\subseteq [d]$ are given by
\begin{equation}\label{eq:gsi_of_g}
\sigma^2(g_\vec u)
=\sum_{\vec j_\vec u\geq\vec 0}\vec a_{\uinv{j}{u}}^\top\vec \Lambda_{\uinv{j}{u}}\vec a_{\uinv{j}{u}},
\end{equation}
where the matrices $\vec \Lambda_{\vec j}$ are defined in \eqref{eq:Lambda_j} and we denote the vectors
$\vec a_{\vec j} = \left(a_{\vec j,\vec k}\right)_{\vec k\in \I_{\vec j}}$.
\end{theorem}
\begin{proof}
We use Definition~\ref{def:anova-terms} to determine the ANOVA terms. We take into account that the wavelets have the property that 
\begin{align*}
\int_{\T} \psi_{j,k}^\per\d x=\delta_{j,-1}.
\end{align*}
We use induction over $|\vec u|$ and begin with
\begin{align*}
g_\varnothing = \int_{\T^d}\sum_{\vec j\geq-\vec 1}\sum_{\vec k\in \I_{\vec j}}a_{\vec j,\vec k}\psi_{\vec j,\vec k}^\per(\vec x)\d \vec x=a_{-\vec 1,\vec 0}.
\end{align*}
The induction step reads
\begin{align*}
g_{\vec u}&=\int_{\T^{d-|\vec u|}}\sum_{\vec j\geq-\vec 1}\sum_{\vec k\in \I_\vec j}a_{\vec j,\vec k}\psi_{\vec j,\vec k}^\per(\vec x)\d \vec x_{\vec u^c}-\sum_{\vec v\subset\vec u}g_\vec v(x_\vec v)\\
&=\sum_{\vec j_\vec u\geq-\vec 1} \sum_{\vec k\in \I_{\uinv{j}{u}}}a_{\uinv{j}{u},\vec k}\psi_{\vec j_\vec u,\vec k_\vec u}^\per(\vec x_\vec u)-\sum_{\vec j_\vec v\geq\vec 0}\sum_{\vec k\in \I_{\uinv{j}{v}}}a_{\uinv{j}{v},\vec k} \psi^\per_{\vec j_\vec v,\vec k_\vec v}(\vec x_\vec v)\\
&=\sum_{\vec j_\vec v\geq\vec 0}\sum_{\vec k\in \I_{\uinv{j}{u}}}a_{\uinv{j}{u},\vec k} \psi^\per_{\vec j_\vec u,\vec k_\vec u}(\vec x_\vec u).
\end{align*}
For determining the variances of these terms, we use the fact that $\langle\psi_{j_1,k_1}^\per,\psi_{j_2,k_2}^\per\rangle =0,\text{ if } j_1\neq j_2$, see~\eqref{eq:sca_psi_per}.
Hence, 
\begin{align*}
\sigma^2(g_\vec u)&=\int_{\T^{|\vec u|}}g_\vec u(\vec x_\vec u)\d x_\vec u=\int_{\T^{|\vec u|}}\left(\sum_{\stackrel{\vec j_\vec u\geq\vec 0}{|\vec j_\vec u|_1\leq n}}\sum_{\vec k\in \I_{\uinv{j}{u}}}a_{\uinv{j}{u},\vec k} \psi^\per_{\vec j_\vec u,\vec k_\vec u}(\vec x_\vec u)\right)^2\d \vec x_\vec u\\
&=\sum_{\vec j_\vec u\geq\vec 0}\int_{\T^{|\vec u|}}\left(\sum_{\vec k\in \I_{\uinv{j}{u}}}a_{\uinv{j}{u},\vec k} \psi^\per_{\vec j_\vec u,\vec k_\vec u}(\vec x_\vec u)\right)^2\d \vec x_\vec u\\
&=\sum_{\vec j_\vec u\geq\vec 0}\sum_{\vec k\in \I_{\uinv{j}{u}}}\sum_{\vec \ell\in \I_{\uinv{j}{u}}}a_{\uinv{j}{u},\vec k_\vec u} a_{\uinv{j}{u},\vec \ell_\vec u}\int_{\T^{|\vec u|}}\psi^\per_{\vec j_\vec u,\vec k_\vec u}(\vec x_\vec u)\psi^\per_{\vec j_\vec u,\vec \ell_\vec u}(\vec x_\vec u)\d \vec x_\vec u\\
&=\sum_{\vec j_\vec u\geq\vec 0}\vec a_{\uinv{j}{u}}^\top\vec \Lambda_{\uinv{j}{u}}\vec a_{\uinv{j}{u}}.\qedhere
\end{align*}
\end{proof}
This theorem tells us that the description of a function $g\in L_2(\T^d)$ in terms of wavelets like in~\eqref{eq:g=apsi} inherits the ANOVA structure of the function, since we have for every subset $\vec u \in \mathcal{P}([d])$,
$$\langle g(\vec x),\psi_{\uinv{j}{u},\vec k}^{\per,*}(\vec x)\rangle=\langle g_\vec u(\vec x_\vec u),\psi_{\vec j_\vec u,\vec k_\vec u}^{\per,*}(\vec x_\vec u)\rangle.$$
Note that there are only few different matrix entries in the matrices $\vec \Lambda_{\vec j}$. 
Hence, these few entries only have to be precomputed and with \eqref{eq:gsi_of_g} the global sensitivity indices $\rho(\vec u,S_n^\X f)$ can be computed in a fast way.

\subsection{Truncating the ANOVA decomposition}\label{sec:truncate_ANOVA}
So far we used in Algorithm~\ref{alg:1} all ANOVA terms to approximate a function $f\in H^s_\mix(\T^d)$. The number of ANOVA terms of a
function is equal to $2^d$  and therefore grows exponentially in the dimension $d$. This reflects the curse of dimensionality 
in a certain way and poses a problem for the approximation of a function, even while we use a wavelet decomposition, 
which decreases the number of used parameters in comparison to using a full grid approximation. For that reason, we want to truncate
the ANOVA decomposition, i.e., removing certain terms $f_{\vec u}$, and creating certain form of sparsity. 

To this end we introduce the notion of effective dimension, see \cite{CaMoOw97}.  
\begin{Definition}\label{def:effective_dim}
For $0<\epsilon_s\leq 1$ the \textit{effective dimension} of $f$, in the \textit{superposition sense}, is the smallest integer $\nu\leq d$, such that 
$$\sum_{|\vec u|\leq \nu}\sigma^2(f_\vec u)\geq \epsilon_{\nu} \sigma^2 (f).$$ 
\end{Definition}
This means, we can describe a function with low effective dimension in the superposition sense by only using a low dimensional approximant very well. 
For that reason, we introduce the set $U_{\nu}$, where we use all ANOVA terms up to the superposition dimension $\nu$, i.e. 
\begin{equation}\label{eq:Unu}
U_{\nu}:=\{\vec u\in[d] \mid |\vec u|\leq \nu\}.
\end{equation}
Furthermore, in~\cite{PoSc19a} was shown that functions of dominating mixed smoothness have low effective dimension. 
There they bounded the truncation error $\norm{\sum_{\vec u\notin U_{\nu}}f_{\vec u}}_{L_2(\T^d)},$
which we accept in the following by supposing that a function has only dimension interactions up to order $\nu$. \\

To gain from a low effective dimension, we introduce the following ANOVA inspired Sobolev spaces of dominating mixed derivatives with superposition dimension $\nu$
\begin{align}
H^{s,\nu}_{\mix}(\T^d)&=\{f\in H_{\mix}^s(\T^d)\mid f_{\vec u}=0 \text{ for all } \vec u\notin U_{\nu} \},\label{eq:Hsnu}\\
H^{s,U}_{\mix}(\T^d)&=\{f\in H_{\mix}^s(\T^d)\mid f_{\vec u}=0 \text{ for all } \vec u\notin U \}.\label{eq:HsU}
\end{align}
For functions in these subspaces of $H^s_{\mix}(\T^d)$ we adapt our algorithm to benefit from the structure of $f$. The first spaces were already introduced in~\cite{DuUl13} in terms of Fourier coefficients. But there they used trigonometric polynomials for approximation.

Theorem \ref{thm:anova_psi} tells us, which coefficients of the wavelets representation of a function in $L_2(\T^d)$
coincide to which ANOVA terms. We truncate the operator $P_n$, defined in~\eqref{eq:defPn} to a set 
$\varnothing \in U\subseteq \mathcal P([d])$ by
\begin{equation*}
P_{n,U}f := \langle f, 1_{\T^d}\rangle\, 1_{\T^d}+\sum_{\varnothing\neq\vec u\in U} \sum_{\vec j\in \J_n^{\{\vec u\}}} \sum_{\vec k\in \I_{\uinv{j}{u}}} \langle f,\psi_{\uinv{j}{u},\vec k}^{\per *} \rangle \psi_{\uinv{j}{u},\vec k}^\per,
\end{equation*}
where we define analog to~\eqref{eq:J_n} the index-sets
$$\J_n^{\{\vec u\}} := \{\vec j\in \Z^d\mid \vec j\geq -\vec 1, \vec j_{\vec u^c}=-\vec 1,|\vec j_{\vec u}|_1\leq n\}$$
and $\J_n^{U} = \bigcup_{\vec u\in U}\J_n^{\{\vec u\}}$.
Note that for the whole power set $U = \mathcal P([d])$ we obtain the untruncated projection $P_nf$ from~\eqref{eq:defPn}.
In order to truncate the ANOVA decomposition of $f$, we have to know which terms we can omit. If the function $f$ has 
low superposition dimension $\nu$, we use $U_{\nu}$ as truncation index set, where we have to know $\nu$ in advance or 
we have to make a suitable guess. It turns out that in many real world problems the superposition dimension is low, 
see \cite{CaMoOw97, DePeVo10, KuSlWaWo09, Wu2011, PoSc21}. 
Therefore we have to determine in a first step the variances of the ANOVA terms. Then we omit in the second step these ANOVA terms for which 
$\rho(\vec u, S_n^\X f)< \varepsilon$ for some threshold-parameter $\varepsilon$.

Algorithm~\ref{alg:1} allows us to restrict our approximation to some index-set $U\subset\mathcal P([d])$, 
while using the decomposition in Theorem~\ref{thm:anova_psi}. This coincides with deleting columns in the matrix $\vec A$, 
which belong to ANOVA terms $g_\vec u$with $\vec u\notin U$. Instead of the matrix $\vec A$ of the first step we use the reduced matrix 
$$\vec A_U=(\psi_{\vec j,\vec k}^\per(\vec x))_{\vec x\in \X,\stackrel{\vec j\in \J_n^U}{\vec k\in \I_\vec j}},$$
which allows us to increase the maximal level $n$. The whole algorithm is summarized in Algorithm~\ref{alg:anova}. 
One remaining question is, how much samples do we need for the approximation if we only use some columns $\vec A_U$ of 
the hyperbolic wavelet matrix. Again, we denote by $N$ the number of parameters, i.e. 
$N = \text{dim }\lin \{\psi_{\vec j,\vec k}^\per\mid \vec j\in \J, \vec k\in \I_\vec j\}$.
If we do require nothing to the index set $\J$, we have
$$\sup_{\vec x\in \T^d} \sum_{\vec j\in \J} \sum_{\vec k\in \I_{\vec j}}\left|\psi^\per_{\vec j,\vec k}(\vec x)\right|^2\lesssim  \sum_{\vec j\in \J}2^{|\vec j|_1} = N,$$
which follows by the same arguments as in \eqref{eq:RN_bound}.
This shows that our theory in Section~\ref{sec:HWR} also applies if we choose other index sets than 
$\J_n$ as for the hyperbolic wavelet regression. Especially we can also apply Theorem~\ref{thm:norm_MP}. 
If we use a number of $N$ wavelets for the approximation, we have to use $M \approx N\log N$ samples. 
Especially, if we choose an index set $\J =\J_n^{U_{\nu}}$, this has cardinality $\binom{d}{\nu}2^n n^{\nu-1}$. Therefore, we do in a first step the hyperbolic wavelet approximation with the matrix $\vec A_{U_\nu}$. Then we calculate the global sensitivity indices of the resulting approximant. In the second step  we omit ANOVA terms with low variances. This reduction allows us to increase the accuracy, i.e. to increase $n$. Algorithm~\ref{alg:anova} summarizes that approach. \\

The theory in Section~\ref{sec:HWR} suffers from the truncation to low dimensional terms. All proofs can be done in the same way, but 
instead of the $d$-dependence we receive the a $\nu$-dependance. The number $N$ of necessary 
parameters reduces in this case to $N=\O\left(\binom{d}{\nu}2^n n^{\nu-1}\right)$. The number of samples has to fulfill
$$M \geq \frac{c_\psi^\nu(r+1)}{\gamma_m^\nu \log((e/2)^{(1/2)})}N\log N  = \O\left(\binom{d}{\nu}2^n n^\nu\right). $$
Similarly to Corollary~\ref{cor:one_f} we get 
\begin{equation*}
\P\left(\norm{f-S_n^{\X,U_\nu} f}_{L_2(\T^d)}^2 \lesssim  (1+\tfrac{2}{\gamma_m^\nu}(r+\sqrt r+1)) \,2^{-2ns}\norm{f}^2_{H^s_{\mix}(\T^d)}\right)\geq 1- 2\,M^{-r},
\end{equation*}
where we eliminate the $d$-dependance. The approximation operator $S_n^{\X,U_\nu}$ is defined in Algorithm~\ref{alg:anova}. Analog estimates can be done for the $L_\infty$-error as well as for the spaces $\bB_{2,\infty}^s(\T^d)$. \\

To conclude this, in Table~\ref{tab:Ms_and_Ns} we summarize the asymptotic behavior of full grid approximation, hyperbolic wavelet 
approximation as well as
 approximation of functions with low effective dimensions using ANOVA ideas. For the comparison with the full grid see~\cite[Section 3.5]{Bo17}. 
The hyperbolic wavelet regression coincides with Algorithm~\ref{alg:1}. The truncated hyperbolic regression coincides with steps $2,3$ of 
Algorithm~\ref{alg:anova}. Finally, the ANOVA-hyperbolic wavelet regression coincides with steps $6,7$ of 
Algorithm~\ref{alg:anova}. In all cases we end up with the approximation error from Corollary~\ref{cor:one_f}.

\begin{table}\centering
\begin{tabular}{ccc|ccc}
    \hline
    regression type&space   &defined in & $N$ &$M$ \\
    \hline
		 full grid &$H^s(\T^d)$ && $\O(2^{nd})$ & $\O(2^{nd}nd)$\\
    hyberbolic &$H^s_{\mix}(\T^d)$ &\eqref{eq:Hsmix}& $\O(2^n n^{d-1})$ &$\O(2^n n^d)$\\
		truncated hyperbolic&$H^{s,{\nu}}_\mix(\T^d)$ &\eqref{eq:Hsnu}&$\O(\binom{d}{\nu} 2^n n^{\nu-1})$&$\O(\binom{d}{\nu} 2^n n^{\nu})$\\
		ANOVA-hyperbolic&$H^{s,U}_\mix(\T^d)$ &\eqref{eq:HsU}&$\O(2^n\, |U|\,n^{\max_{\vec u\in U}|\vec u|-1})$&$\O(2^n \,|U|\,n^{\max_{\vec u\in U}|\vec u|})$\\
		\hline

  \end{tabular}
\caption{Number of needed wavelet functions $N$ and needed cardinality $M$ of samples in different settings. }
\label{tab:Ms_and_Ns}
\end{table}

\begin{algorithm}[ht]
\caption{ANOVA - Hyperbolic wavelet regression}
	\vspace{2mm}
	\begin{tabular}{ l l l }
		\textbf{Input:} & $ d $ & dimension\\
		& $ \nu $ & superposition dimension\\
		&	$\X = (x_i)_{i=1}^M\in \T^d$ & sampling nodes \\
		& $\vec y = (f(x_i))_{i=1}^M$ & function values at sampling nodes \\
		& $0<\varepsilon<1$ & threshold parameter
	\end{tabular}
	\begin{algorithmic}[1]
			\STATE{Choose $n$ such that for $N=\bigcup_{\vec j\in \J}|\I_{\vec j}|$ holds $M>N\log N$, where we define $\J =\J^{U_{\nu}}_n$.}
			\STATE{Construct the sparse matrix 
			$$\vec A_{U_{\nu}}=(\psi_{\vec j,\vec k}^\per(\vec x))_{\vec x\in \X,\stackrel{\vec j\in \J }{\vec k\in \I_\vec j}}\in \C^{M\times N}.$$}
	    \STATE{Solve the overdetermined linear system
			$\vec A_{U_{\nu}} \left(a_{\vec j,\vec k}\right)_{\vec j,\vec k} =\vec y$ 
			via an LSQR-algorithm. This gives us the approximation
			$$S_n^{\X,U_\nu} f :=\sum_{\vec j\in \J_n^{U_\nu}}\sum_{\vec k\in \I_j}a_{\vec j,\vec k}\psi_{\vec j,\vec k}^\per$$
			}
			\STATE{Determine $\rho(\vec u,S_n^{\X,U_{\nu}}f)$ using Theorem~\ref{thm:anova_psi}. }
			\STATE{$U\gets \{\vec u\mid \rho(\vec u,S_n^{\X,U_{\nu}}f)>\varepsilon\}\cup \varnothing$ }
			\STATE{Choose $n$ such that for $N=\bigcup_{\vec j\in \J_n^U}|\I_{\vec j}|$ holds $M>N\log N $.}
			\STATE{Construct the sparse matrix 
			$$\vec A_{U}=(\psi_{\vec j,\vec k}^\per(\vec x))_{\vec x\in \X,\stackrel{\vec j\in \J^U_n}{\vec k\in \I_\vec j}}\in \C^{M\times N}.$$}
	    \STATE{Solve the overdetermined linear system 		
				$\vec A_U \left(a_{\vec j,\vec k}\right)_{\vec j,\vec k} =\vec y$
			via an LSQR-algorithm.}
	\end{algorithmic}
	\begin{tabular}{ l l l }
		\textbf{Output:} 
		 &  $\left(a_{\vec j,\vec k}\right)_{\vec j,\vec k}\in \C^N$ coefficients of the approximant $S_n^{\X,U} f :=\sum_{\vec j\in \J_n^{U}}\sum_{\vec k\in \I_j}a_{\vec j,\vec k}\psi_{\vec j,\vec k}^\per$
	\end{tabular}
	\label{alg:anova}

\end{algorithm}

\begin{remark}\label{rem:fourier}
To compare our theory with the results in~\cite{PoSc19a}, in the Fourier setting a full grid approximation 
with polynomial degree at most $2^n$ has the same approximation rate, number of parameters $N$ and number 
of needed samples $M$ as the full grid approximation in the wavelet case. Choosing hyperbolic cross index 
sets in frequency domain gives a similar cardinality of the index set as the hyperbolic wavelet regression. 
But for this case no fast algorithms are available so far. However, in~\cite{PoSc19a} they use full index 
sets of dimension $\nu$, which give the worse estimates $N=\O(\binom{d}{\nu}\,2^{n\nu})$ in comparison to 
the third line of Table~\ref{tab:Ms_and_Ns}. 
\end{remark}

\section{Numerical results}\label{sec:num}
After deriving our theoretical statements in the previous chapters, in this section we now underpin our
findings by several numerical results. We use the Chui-Wang wavelets from Example~\ref{ex:Chui-Wang}, 
which fulfill the properties~\eqref{eq:support}, \eqref{eq:moments} and \eqref{eq:Riesz}.
First we give estimates about the computational cost. We begin with a kink 
function to illustrate our theorems from Section~\ref{sec:wavelets} by using Algorithm~\ref{alg:1}.
Our second example shows the benefit of using wavelets with vanishing moments of higher order for function with higher regularity. The example in Section~\ref{sec:ex3} of 
a high-dimensional function shows that Algorithm~\ref{alg:anova} is a powerful method. The last example in 
Section~\ref{sec:pyramid} shows that even functions in a Sobolev space with low regularity benefit from the ANOVA ideas.

\subsection{Computational cost}
In order to determine the complexity of our algorithms we first calculate the complexity of one matrix multiplication with the hyperbolic wavelet matrix. 
For fixed $\vec j\in \J_n$ there are at most $\lceil 2m-1\rceil ^d$ non-zero entries in every block of columns in every row. Since
$\sum_{\vec u\in U}\sum_{|\vec j_{\vec u}|\leq n} 1=\sum_{\vec u}\O(n^{|\vec u|}),$
every matrix multiplication with the hyperbolic wavelet matrix $\vec A_{U}$ has complexity 
$$\O(M \,\sum_{\vec u}\O(n^{|\vec u|}))=\O(2^n |U| n^{\max|\vec u|} )=\O(M(\log M)^{\max|\vec u|} ).$$ 
The cases where we choose $U=\mathcal P([d])$ in Algorithm~\ref{alg:1} and $U=U_\nu$ in Algorithm~\ref{alg:anova} are included in this consideration.
Note that in contrast to the Fourier setting, see~\cite{PoSc19a}, our index set has a lower cardinality $N$, which gives a better complexity of the algorithm.\\

The second factor which plays a role is the number of iterations $r^*$. The whole algorithm has a complexity of
$$\O(r^*\,M(\log M)^{\max|\vec u|}).$$ 
Theorem~\ref{thm:norm_MP} gives us an estimation about the minimal eigenvalues of the matrix $\vec A^*\vec A$. 
In order to bound the maximal eigenvalues of the matrix $\vec A^*\vec A$, we can use the same argumentation as in the 
proof of Theorem \ref{thm:norm_MP} together with the bound for the maximal eigenvalues in Theorem~\ref{thm:matrix_chernoff}. This gives us for $r>1$
$$ \P\left(\mu_{\max}\left(\vec A^*\vec A\right)\geq \tfrac{3\,M}{2} \right)\leq \tfrac{1}{M^r},$$
if we are in the setting of logarithmic oversampling. Hence, with high probability we can bound the condition number of the matrix $\vec A^*\vec A$ by
$$\kappa(\vec A^*\vec A):=\frac{\mu_{\max}(\vec A^*\vec A)}{\mu_{\min}(\vec A^*\vec A)}\leq  \frac{3}{\gamma_m^d}. $$
 
Following \cite[Example 13.1]{ax} the maximal number of iteration $r^*$, to achieve an accuracy of $\varepsilon$, can be bounded by
\begin{align*}
r^*\leq \log\left(\frac 2\varepsilon\right)\,\left(\log\left(\frac{1+(\gamma_m^{d}/3)^{1/4}}{1-(\gamma_m^{d}/3)^{1/4}}\right)\right)^{-1} .
\end{align*}
As an example, if we aim an accuracy of $\varepsilon = 10^{-4} $ and we choose as parameters the order of vanishing moments $m =3$ and the dimension $d =2$, this bound gives us $r^*\leq 85$ with high probability.\\

The calculation of the global sensitivity indices in step $4$ of Algorithm~\ref{alg:anova} does not play role compared to the LSQR-algorithm, since this
can be done in a fast way. There are only few different matrix entries in the matrices $\vec \Lambda_{\vec j}$ in \eqref{eq:gsi_of_g}, so these 
few entries can be precomputed. Also, these matrices are sparse circulant matrices of size smaller than $2^{n\nu}$.

\subsection{Kink test function}
We start with an example of an $L_2(\T^d)$-normalized kink function,
\begin{equation}\label{eq:kink}
f:\T^d \to \R,\quad f(x) = \prod_{i=1}^d\left(\sqrt{\frac{98415}{32}}\max\left(\frac19 - x^2,0\right)\right)\in \bB^{3/2}_{2,\infty}(\T^d),
\end{equation}
which has the Fourier-coefficients 
\begin{equation*}
c_{\vec k}(f) = \prod_{i=1}^d\begin{cases}
\frac{27}{8}\,\sqrt{\tfrac{15}{2}}\,\frac{3\,\sin(2k_i \pi /3)-2k_i\pi \cos(2k_i\pi/3)}{\pi^3 k_i^3}&\text{ for } k_i\neq 0,\\
\sqrt{\tfrac{15}{2}} & \text{ for } k_i=0,
\end{cases}
\end{equation*}
i.e. the Fourier coefficients of this function decay like $|c_{\vec k}(f)|\sim\prod_{i=1}^d\left(1+|k_i|^2\right)^{-1}$. 
Consequently, $f\in H^s_\mix(\T^d)$ with $s = 1+\tfrac 12 -\varepsilon=\tfrac 32 -\varepsilon  $. In addition it follows that $f \in \bB^{3/2}_{2,\infty}(\T^d)$. Hence, we use wavelets with vanishing moments of order $m>\tfrac 32$.

\begin{figure}[ht]

  \centering
\begin{tikzpicture}[scale=1]
\begin{axis}[xmin=-0.5,xmax=0.5,scale only axis, width=5cm]
\addplot[domain = -1/9^(1/2):1/9^(1/2),blue]{98415^(1/2)/112^(1/2)*(1/9 - x^2)};

\addplot[domain = -0.5:-1/9^(1/2),blue]{  0};  
\addplot[domain = 1/9^(1/2):0.5,blue]{  0};\end{axis}
\end{tikzpicture}
 
\caption{The kink function \eqref{eq:kink} for $d=1$.}
	\label{fig:function}
	\end{figure}
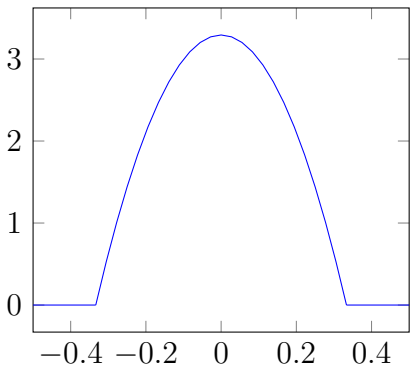
We will begin with the one-dimensional function, which is plotted in Figure \ref{fig:function}.  
For our approximation we use $n= 9$, i.e we have $N = 1024$ parameters. To apply our theory, 
we have to choose logarithmic oversampling. For that reason we sample the function at the i.i.d. 
samples $\X$ with $|\X|=M=20000$. Then we use Algorithm~\ref{alg:1} to approximate the 
coefficients $a_{j,k}\approx \langle f,\psi_{j,k}^{\per,*}\rangle$. Figure~\ref{fig:1d_wavelets_coeff1} shows the resulting coefficients.  
\begin{figure}[ht]
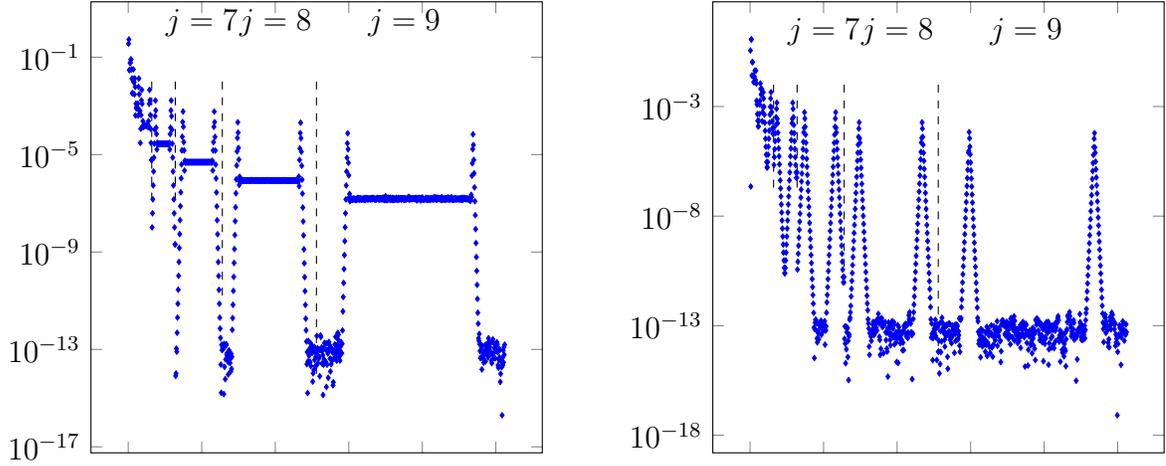

\centering
\begin{subfigure}[c]{0.49\textwidth}
\raggedleft
 \end{subfigure}
\caption{Wavelet coefficients of the kink function after approximation for $m = 2$ (left) and $m = 3$ (right) for $d=1$.}
\label{fig:1d_wavelets_coeff1}
\end{figure}
Our test function is piecewise polynomial of degree $2$, only at the two kink points we have regularity $\tfrac 32$. Figure~\ref{fig:1d_wavelets_coeff1} shows that the wavelet coefficients detect locally lower regularity. This is due to the compact support of the wavelets. 

In Figure \ref{fig:1d_wavelets_coeff2} we plotted the sums $\sum_{|\vec j|_1=n}\sum_{\vec k\in \I_{\vec j}}|a_{\vec j,\vec k}|^2$. In Theorem~\ref{thm:norm-psi*} we proved that $\sum_{\vec k\in \I_{\vec j}}|\langle f, \psi_{\vec j,\vec k}^{*,\per}\rangle|^2$ decay like $2^{-2ns}$. Lemma~\ref{lem:N} gives the estimation that there are $\binom{n+d-1}{d-1}=\O(n^{d-1})$ indices $\vec j$, such that $|\vec j|_1=n$. This gives us a proposed decay rate $2^{-3n}n^{d-1}$. We see this decay even though we approximated the wavelet coefficients $\langle f,\psi_{\vec j,\vec k}^{*,\per}\rangle$ by the solutions $a_{\vec j,\vec k}$ of the hyperbolic wavelet regression.   

\begin{figure}[ht]
\begin{subfigure}[l]{0.5\textwidth}
\raggedleft
\begin{tikzpicture}[scale=1]
\begin{semilogyaxis}[scale only axis,width = 4.4cm, height = 4.4cm,
xlabel = {level $n$},ylabel={sum wavelet coefficients},
legend entries={$d=1$ ,$d=2$ ,$d=3$ , $2^{-3 n }$,$2^{- 3n  }\,n$,$2^{-3 n  }\,n^2$},
legend style={at={(axis cs:9.8,10^-7)},anchor=south west,scale =0.6},
grid=major
]

\addplot+[blue,mark = *]  coordinates {
(1,0.007420063242759986)
(2,0.08390254796927547)
(3,0.010696309292201575)
(4,0.001914971341935989)
(5,0.00021499480004481694)
(6,2.859600290040552e-5)
(7,3.4598394883977946e-6)
(8,4.4192855362084734e-7)
(9,4.554633333194939e-8)
};

\addplot[red,mark=*]  coordinates {
(1,0.02605977990464884)
(2,0.29472569665706005)
(3,0.03881217203162489)
(4,0.013924611922015984)
(5,0.002578879546530748)
(6,0.0005396357919130869)
(7,8.97722637366388e-5)
(8,1.4680772133621016e-5)
(9,2.2042792615882147e-6)

};

\addplot[green,mark=*]  coordinates {
(1,0.06864205212945357)
(2,0.7764582745753281)
(3,0.10551034727122)
(4,0.05565125886358685)
(5,0.01175308057190851)
(6,0.0032086490384034727)
(7,0.0006761140514349357)
(8,0.00014478636815719583)
(9,2.8000145680203137e-5)

};

\addplot[black,mark=none,dashed] coordinates {(4,1*2^-4*3) (9,1*2^-3*9)};
\addplot[orange,mark=none,dashed] coordinates {(4, 10*4 * 2^-4*3 ) (9, 10 * 9  * 2^-3*9 )};
\addplot[cyan,mark=none,dashed] coordinates {(4, 10*4^3 * 2^-4*3 ) (9, 10 * 9^3  * 2^-9*3 )};

\end{semilogyaxis}
\end{tikzpicture} \subcaption{Decay of the wavelet coefficients for $m=2$.}
\label{fig:1d_wavelets_coeff2}
\end{subfigure}
\begin{subfigure}[c]{0.5\textwidth}
\raggedleft
\begin{tikzpicture}[scale=1]
\begin{semilogyaxis}[scale only axis,width = 4.4cm, height = 4.4cm,
xlabel = {level $n$},ylabel ={RMSE},
legend entries={$d=1$ ,$d=2$,$d=3$,$2^{- n \, 3/2}$,$2^{- n \, 3/2}\,n^{1/2}$,$2^{- n \, 3/2}\,n$},
grid=major,
legend style={at={(axis cs:13.1,10^-6)},anchor=south west,scale =0.6}
]
\addplot[blue,mark=*]  coordinates {
(3,0.05111863650215732)
(4,0.017215332542154075)
(5,0.005691907961835877)
(6,0.002399019868404038)
(7,0.0006429678005481621)
(8,0.00023821863949560293)
(9,8.336269795759345e-5)
(10,2.8710997130589267e-5)
(11,1.0007319659266092e-5)
(12,3.5600705304164336e-6)
};

\addplot[red,mark=*]  coordinates {
(3,0.16856421475200833)
(4,0.07089288280406299)
(5,0.02624329124676823)
(6,0.010167210102632482)
(7,0.003942528258178008)
(8,0.0014441024076449257)
(9,0.0005394330184956737)
(10,0.00018886829588602564)
(11,6.715275607665141e-5)
};

\addplot[green,mark=*]  coordinates {

(1,19.832951875445758)
(2,5.765613719396113)
(3,3.1673664718743755)
(4,1.2614449968382548)
(5,0.5706041580507363)
(6,0.23580773858796403)
(7,0.09962562030897586)

};

\addplot[black,mark=none,dashed] coordinates {(4,0.1*1/2^6) (12,0.1*1/2^18)};
\addplot[orange,mark=none,dashed] coordinates {(3, 5*3^1/2 * 2^-3/2*3 ) (12, 5 * 12^1/2  * 2^-3/2*12 )};
\addplot[cyan,mark=none,dashed] coordinates {(3, 10*3 * 2^-3/2*3 ) (7, 10 * 7  * 2^-3/2*7 )};

\end{semilogyaxis}
\end{tikzpicture} \subcaption{Decay of the RMSE for $m=2$. }
\label{fig:1d_conv}
\end{subfigure}
\caption{Approximation of the kink function.}
\end{figure}
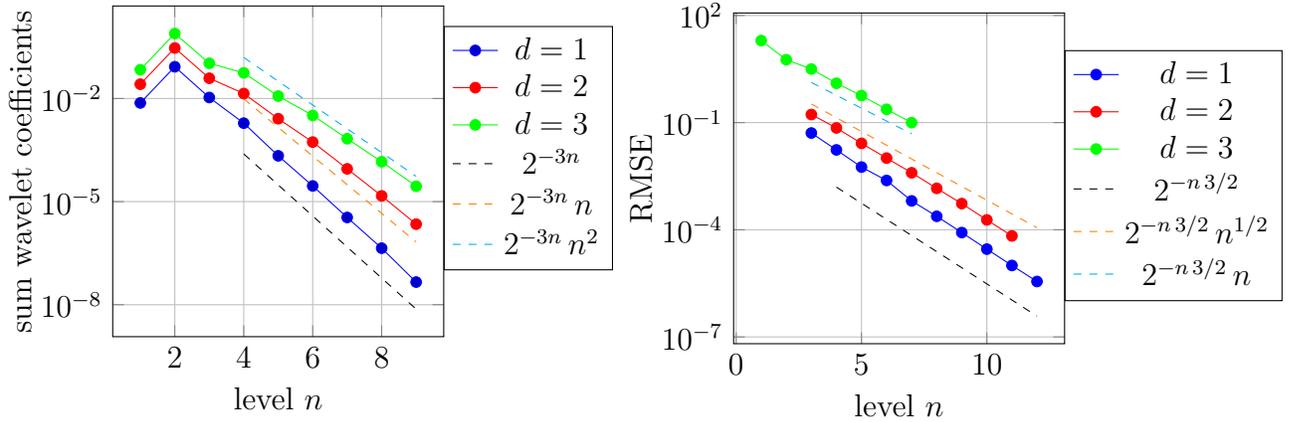

Using different parameters $n$, 
while always ensuring logarithmic oversampling, we see the decay of the $L_2(\T^d)$-error in Figure~\ref{fig:1d_conv}. It matches the proposed error bound $2^{-3/2n}n^{(d-1)/2}$ from Corollary~\ref{cor:one_f}. To measure the error we use the \textit{root mean squared error} (RMSE), which is defined by
\begin{equation*}\label{eq:MSE}
\text{RMSE} = \left(\frac{1}{|\X_{\text{test}}|}\sum_{\vec x\in \X_{\text{test}}}|f(\vec x)-(S_n^\X f(\vec x))|^2\right)^{\tfrac 12},
\end{equation*}
for some sample points $\X_{\text{test}}\subset \T^d$, 
which gives us a good estimator for the $L_2(\T^d)$-error $\norm{f-S_n^\X f}_{L_2(\T^d)}$. Since we always use $L_2(\T^d)$-normalized test functions, 
the RMSE can be interpreted as a relative error.
For the RMSE we use random points $\X_{\text{test}}\subset \T^d$ with $|\X_{\text{test}}|=10^6$ as test samples.

\subsection{A function with higher mixed regularity}\label{sec:ex3}
We consider the following test function 
\begin{equation}\label{eq:ex2}
f:\T^3 \to \R,\quad f(\vec x) = \prod_{i=1}^3 B_3(4x_i-\tfrac 1\pi), \quad \vec x\in [-\tfrac 12,\tfrac 12)^3,
\end{equation}
where we use the B-spline from~\eqref{eq:BSpline}. This function is in $H^{5/2-\varepsilon}_\mix(\T^3)$ and in $\bB^{5/2}_{2,\infty}(\T^d)$. 
We do an approximation using Algorithm~\ref{alg:1}, where we use Chui-Wang wavelets of order $m=2$ and $m=3$. 
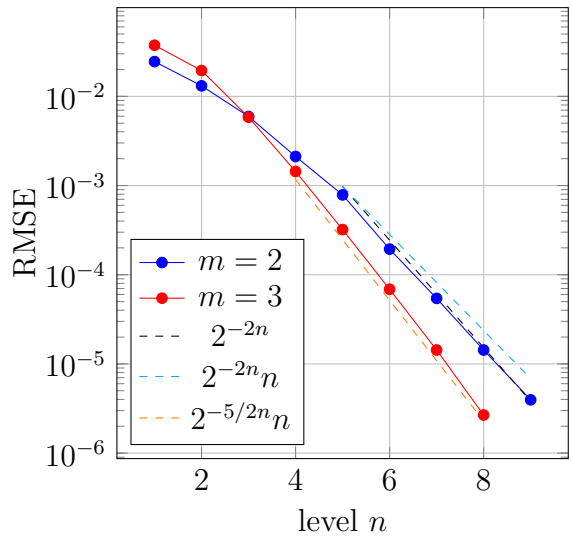
\begin{figure}[htb]
\centering
\begin{tikzpicture}[scale=1]
\begin{semilogyaxis}[scale only axis,width = 6cm, height = 6cm,
xlabel = {level $n$},ylabel ={RMSE},
legend entries={$m = 2$, $m = 3$, $2^{-2n}$,$2^{-2 n }n$,$2^{-5/2 n }n$},
grid=major,
legend pos=south west,
]
\addplot[blue,mark=*]  coordinates {

(1,0.02453002031561157)
(2,0.013125555463623787)
(3,0.005950874471086693)
(4,0.002118353263104248)
(5,0.0007854944568757404)
(6,0.00019373301595564362)
(7,5.428381669121783e-5)
(8,1.4315719102259882e-5)
(9,3.956760930675728e-6)

};

\addplot[red,mark=*]  coordinates {

(1,0.037439314538755514)
(2,0.019478209284568882)
(3,0.005836622364221294)
(4,0.0014369814273951253)
(5,0.0003210210349160611)
(6,6.872014743947136e-5)
(7,1.4327812916367332e-5)
(8,2.677051115960276e-6)

};

\addplot[black,mark=none,dashed] coordinates {(5,1 * 2^-5*2) (9, 1 * 2^-9*2)};
\addplot[cyan,mark=none,dashed] coordinates {(5,  0.2 * 5^1  * 2^-5*2  ) (9,0.2 *9^1 * 2^-9*2  )};
\addplot[orange,mark=none,dashed] coordinates {(4,  0.3*4^1  * 2^-4*2.5  ) (8,0.3*8^1 * 2^-8*2.5  )};

\end{semilogyaxis}
\end{tikzpicture} 
\caption{Decay of the RMSE for the test function~\eqref{eq:ex2}.}
\label{fig:RMSE_ex2}
\end{figure}

While always ensuring logarithmic oversampling, we use $|\X_{\text{test}}|= 3M$ samples for calculation of the RMSE. 
The results can be seen in Figure~\ref{fig:RMSE_ex2}. This confirms our proposed error decay from Corollary~\ref{cor:one_f}. 
For $m=2$ we are in the setting where $s=m$ and the error decays a bit faster than $2^{-2n}n$ but slower than $2^{-2n}$ (recall that $d=3$). 
If we use wavelets of higher regularity, i.e. $m=3$, we reduce the error as well as the decay rate compared to $m=2$. 
We are in the case where $s<m$ and since the test function is in $\bB_{2,\infty}^{5/2}(\T^d)$, we proved that the error decays like 
$2^{-5/2n}n$, which is confirmed by the numerical experiments.

\subsection{A high-dimensional function with low effective dimension}
The Ishigami function \cite{Is90} is used as an example for uncertainty and sensitivity analysis methods, 
because it exhibits strong non-linearity and non-monotonicity. Since we are in the periodic setting, we consider the suited 
periodized version, add a term which consists of a B-spline term and add additional dimensions, which do not contribute to the function 
$f:\T^8\to \R$ with 
\begin{equation}\label{eq:f}
f(\vec x) = c\,(-\tfrac 72 + \sin(2\pi \,x_1)+7\,\sin^2(2\pi\, x_2)+0.1\, x_3^4\sin(2 \pi \, x_1)+ 10^3\,g(\vec x_{\vec v})),
\end{equation}
where $\vec v = \{6,7,8\}$ and $g$ is a tensor product of a three-dimensional B-spline, see \eqref{eq:BSpline}, of order $6$, given by $g(\vec x_{\vec v})=\prod_{i=6}^8 \left(B_6(2^4 x_i)-\frac{1}{16}\right)$ and the constant $c$ is 
such that the function is normalized to $\norm{f}_{L_2(\T^d)}=1$.
The ANOVA terms and their variances can be computed analytically. The variances of $f_\vec u$ are non-zero only 
for the indices $\vec u\in \{\{1\},\{2\},\{1,3\},\{6,7,8\}\}$. 
It follows easily that the effective dimension of $f$, in the superposition sense, see Definition~\ref{def:effective_dim}, for 
$\epsilon_s=1$ is $\nu = 3$. 
We use this function to test Algorithm~\ref{alg:anova} and initially choose $n=2$, which means that we use $N = 2269$ parameters. 
Therefore we randomly draw $M=10^5$ sample points on $\T^8$. Figure~\ref{fig:gsi} 
shows all $|U_3|=92$ resulting global sensitivity indices from the first step of our 
algorithm as proposed in Theorem~\ref{thm:anova_psi} in comparison to the analytically calculated global sensitivity indices. 
Although we have low maximal level $n$, we can detect the correct ANOVA terms. 
Numerical experiments showed that even with a maximal level $0$ or $1$, we can detect the correct ANOVA terms. 
In the second step we use this information and build an adapted model, i.e. we use $U = \{\varnothing,\{1\},\{2\},\{1,3\},\{6,7,8\}\}$ 
and increase the maximal level to $n = 6$. 

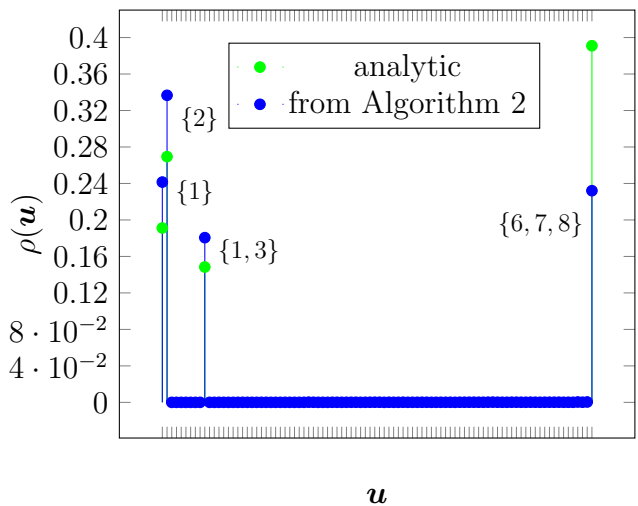
\begin{figure}[ht]
\centering
\begin{tikzpicture}
\begin{axis}[ylabel={$\rho(\vec u)$},
			xlabel={$\vec u$},
			xtick={1,...,92},
			xticklabel=\empty,
			ytick distance=0.04,
			xtick distance=10,
			legend entries={analytic,from Algorithm~\ref{alg:anova}},
legend style={at={(axis cs:15,0.3)},anchor=south west,scale =0.8},
			]
			
			\addplot+[ycomb, green,mark=*, mark options={green}] plot coordinates
			{(1,0.1912)
(2,0.2694)
(10,0.1484)
(92,0.3910)
};
			\addplot+[ycomb, blue,mark=*, mark options={blue}] plot coordinates
			{

(1,0.24153597599559753)
(2,0.33665970671982126)
(3,1.7767810247926656e-5)
(4,2.1025185632957515e-5)
(5,2.168330603101333e-5)
(6,7.565107421285341e-6)
(7,1.4257331677310236e-5)
(8,8.322516651874926e-6)
(9,3.2223899485512406e-5)
(10,0.18060319414467976)
(11,4.444103325715087e-5)
(12,2.3768691756249132e-5)
(13,4.8370615592550214e-5)
(14,4.21863147801197e-5)
(15,2.882947626081978e-5)
(16,3.790756026457199e-5)
(17,3.5883525807961565e-5)
(18,6.538804246852549e-5)
(19,6.12099850284412e-5)
(20,4.237741685353411e-5)
(21,3.743388028240746e-5)
(22,5.762950543527411e-5)
(23,6.247088541988502e-5)
(24,3.576372859487476e-5)
(25,4.63487127680813e-5)
(26,3.89396476417618e-5)
(27,6.005492553441951e-5)
(28,5.293376154848434e-5)
(29,4.405371345884129e-5)
(30,5.129290745116281e-5)
(31,6.933678795866517e-5)
(32,0.0002134282030492247)
(33,0.00011976382094980741)
(34,4.1217960097260803e-5)
(35,4.297927040335825e-5)
(36,3.0027669177974564e-5)
(37,0.00010252480871021029)
(38,0.00010644072806978013)
(39,6.657345658091251e-5)
(40,9.38375960077142e-5)
(41,0.00011067492568236068)
(42,0.00013341430378186248)
(43,5.775472393629593e-5)
(44,0.00010053620486701932)
(45,0.00012287008121354118)
(46,0.0001703474690404934)
(47,0.00015989499162004565)
(48,4.239309375583368e-5)
(49,0.00012844830242332098)
(50,7.50878568318896e-5)
(51,7.921817601480944e-5)
(52,0.0001490185670659867)
(53,7.231732778246743e-5)
(54,7.827869451478764e-5)
(55,0.00010094182606218802)
(56,0.00012381556585984828)
(57,0.00011460711229281735)
(58,6.784979299154731e-5)
(59,8.476642205066531e-5)
(60,0.00010358894111445812)
(61,0.00010168659357947023)
(62,0.00010759128965989494)
(63,7.324044045341094e-5)
(64,0.00014561162843732496)
(65,0.00015028351915478407)
(66,0.0001648795105360265)
(67,0.0001667153200836816)
(68,0.0001406611405351508)
(69,0.00012568557503432587)
(70,0.00018864205964196128)
(71,0.0002460764288009544)
(72,0.00015495686345288174)
(73,0.00011993412988190183)
(74,0.00011226774893775424)
(75,0.00017363399267357233)
(76,0.00017970455301309312)
(77,0.0001317922467458065)
(78,0.000122177984148131)
(79,0.00018107846497296108)
(80,0.0001631358290455714)
(81,0.00011985796268541476)
(82,0.00016291732357666869)
(83,0.00016704385247647506)
(84,0.00015821596386713616)
(85,0.00019385774183386332)
(86,0.00012131093005794921)
(87,0.0001405967446351259)
(88,0.00010623936985838819)
(89,0.0004088221480325772)
(90,0.00017512665068217405)
(91,0.00042000963439373876)
(92,0.2320752853317532)
};

\node at (axis cs:2,0.23) [anchor=west,scale =0.8] {$\{1\}$};
			\node at (axis cs:3,0.31) [anchor=west,scale =0.8] {$\{2\}$};
			\node at (axis cs:11,0.165) [anchor=west,scale =0.8] {$\{1,3\}$};
				\node at (axis cs:92,0.195) [anchor=east,scale =0.8] {$\{6,7,8\}$};

		\end{axis}

\end{tikzpicture} \caption{Global sensitivity indices $\rho(\vec u)$ of the function~\eqref{eq:f} analytically computed and from Algorithm~\ref{alg:anova} with $n=2$.}
\label{fig:gsi}
\end{figure}
To approximate the $L_2(\T^d)$-error, we use the RMSE with random test samples with $|\X_{\text{test}}|= 10^6$. 
The second step of Algorithm~\ref{alg:anova} allows us to decrease this RMSE from $0.479$ after the first step to $0.064$.
We stress the fact, that this strategy, finding the unimportant dimension interactions and building an adapted model, allows an interpretation of the data. See also~\cite{PoSc21} for real world application in combination with a Fourier basis.

\subsection{A function with small mixed regularity}\label{sec:pyramid}
\label{pyramid}
\begin{figure}[hb]
\begin{subfigure}[t]{0.48\linewidth}
\centering
 \raggedleft
\begin{tikzpicture}[scale=1]

\begin{loglogaxis}[
xlabel = { $N$},ylabel ={RMSE},
legend entries={$U_2$,$U$ in \eqref{eq:U}, $N^{-s} (\log N )^{s}$},
grid = major,
scale only axis, 
width = 6cm, height =6cm,
      enlarge x limits = 0,
			legend pos = south west,
			]

\addplot[red,mark=*, mark options={red}]  
coordinates {

(94,0.687223438540751)
(298,0.3204880246591772)
(826,0.30059076618859415)
(2122,0.14544633065552567)
(5194,0.13015787819900934)
(12298,0.05670854335190603)
(28426,0.050141340422413076)

};

\addplot[blue,mark=*, mark options={blue}]  
coordinates {

(34,0.7332500762501897)
(94,0.32019702903330644)
(238,0.3036899765438543)
(574,0.14738901901377185)
(1342,0.13122711009420043)
(3070,0.057131914828883364)
(6910,0.05045809272383272)

};

\addplot[domain = 100:3000,black,dashed]{ 2*x^(-0.75)*ln(x)^(0.75)};

\end{loglogaxis}
\end{tikzpicture} 
\subcaption{With Algorithm~\ref{alg:anova} achieved RMSE for different levels $n$ and $s=3/4$. }
\label{fig:pyramid1}
\end{subfigure}
\begin{subfigure}[t]{0.48\linewidth}
\centering
 \raggedleft
\begin{tikzpicture}[scale=1]

\begin{semilogyaxis}[xlabel = {level $n$}, ylabel={},
legend entries={$|\vec u|=1$,,,,,,,,,,,,,,,,,,,, $|\vec u|=2$,$2^{-2 n}$},
scale only axis, 
width = 6cm, height = 6cm,
enlarge x limits = 0,
grid = major,
legend pos = south east,
]

\addplot[black]
coordinates {
(1,0.6495362412667611)
(2,0.16234342936793944)
(3,0.04059546569045314)
(4,0.010173018464709409)
(5,0.0025434711440037706)
(6,0.000719490988346319)
(7,0.0005441877042348867)
(8,0.0006692443364490774)
};
\addplot[black]
coordinates {
(1,0.6495922147735796)
(2,0.16236918355487673)
(3,0.0405454024492381)
(4,0.010209789242228814)
(5,0.0025387724729860985)
(6,0.000800447266796756)
(7,0.0005165691220450681)
(8,0.0006832425338754297)
};
\addplot[black]
coordinates {
(1,0.649594426893255)
(2,0.16231939996593595)
(3,0.040629464310695786)
(4,0.010149992119249242)
(5,0.002616390965404838)
(6,0.0007178806379056718)
(7,0.0004809277484926765)
(8,0.0007138683900536173)
};
\addplot[black]
coordinates {
(1,0.6495489754549153)
(2,0.16233957168678417)
(3,0.0406018684527865)
(4,0.0101182212035183)
(5,0.0025727225672041846)
(6,0.0007116293298710109)
(7,0.00042315554349395073)
(8,0.0006628237972755088)
};
\addplot[black]
coordinates {
(1,0.6494720240371122)
(2,0.162368456740353)
(3,0.04052801482435007)
(4,0.010186572884529052)
(5,0.002595296320544535)
(6,0.0007371093215377501)
(7,0.0005392185248880095)
(8,0.0006250381970834683)
};
\addplot[black]
coordinates {
(1,0.6495132763271753)
(2,0.16238820002977852)
(3,0.040613548698388596)
(4,0.010158664864587951)
(5,0.0025661490858287215)
(6,0.0007438500140199381)
(7,0.0005187332028858305)
(8,0.0007022006593729644)
};
\addplot[magenta]
coordinates {
(1,0.00012585140312006702)
(2,0.9434206885744736)
(3,0.10872237471021778)
(4,0.5223580371951858)
(5,0.1612754324943838)
(6,0.30737150137355773)
(7,0.06741890734806953)
(8,0.1330109618388737)
};
\addplot[magenta]
coordinates {
(1,9.949122939625308e-5)
(2,0.00023868854842432982)
(3,0.00042034320685479094)
(4,0.0006851700682648239)
(5,0.001116160044485033)
(6,0.0019124628062473382)
(7,0.002929367201110399)
(8,0.004042195716777138)
};
\addplot[magenta]
coordinates {
(1,0.0001169624177804665)
(2,0.0002757880415284124)
(3,0.00041554006679394906)
(4,0.0007618305218013496)
(5,0.001241567714709735)
(6,0.0019305327370450374)
(7,0.0028098343004106965)
(8,0.0040638480320883495)
};
\addplot[magenta]
coordinates {
(1,0.00010916724342187812)
(2,0.00024904990277345435)
(3,0.000419936827122662)
(4,0.0007230799932701224)
(5,0.0011521347834423025)
(6,0.001678631818940915)
(7,0.002787022393026891)
(8,0.004163546191526254)
};
\addplot[magenta]
coordinates {
(1,6.409042849083584e-5)
(2,0.00026629468517591566)
(3,0.000462660752307487)
(4,0.0007772286466687606)
(5,0.0010695201064184394)
(6,0.001811742592885263)
(7,0.0029144554755266153)
(8,0.004143101134802565)
};
\addplot[magenta]
coordinates {
(1,4.937281564439739e-5)
(2,0.00017997919907153801)
(3,0.0003691380887512218)
(4,0.0007071798495478836)
(5,0.0011823003128040273)
(6,0.0018231336421421949)
(7,0.0027546122545079113)
(8,0.0042251570304116205)
};
\addplot[magenta]
coordinates {
(1,0.00017904264511462624)
(2,0.0003424450102477361)
(3,0.0003678061043236637)
(4,0.0006074069246974958)
(5,0.0011999665416983062)
(6,0.0018287339570815127)
(7,0.002709135853384796)
(8,0.004089674015702717)
};
\addplot[magenta]
coordinates {
(1,0.00022486690549415066)
(2,0.00033257628262424937)
(3,0.0004968607498513115)
(4,0.0006994666446562211)
(5,0.0012756446599973177)
(6,0.0017562225607583779)
(7,0.0028457377477527487)
(8,0.004064419231568448)
};
\addplot[magenta]
coordinates {
(1,8.965515130640907e-5)
(2,0.0002883004047804202)
(3,0.0004107311372244769)
(4,0.000826593211290376)
(5,0.0012757094697537458)
(6,0.0016438907660424956)
(7,0.002965791742116197)
(8,0.004231524322409198)
};
\addplot[magenta]
coordinates {
(1,4.91685533462647e-5)
(2,0.9434136530283833)
(3,0.10870520896346596)
(4,0.5222460273957965)
(5,0.161323299896875)
(6,0.30722285691201634)
(7,0.06749920259662301)
(8,0.13298110799377375)
};
\addplot[magenta]
coordinates {
(1,0.00022113706112924828)
(2,0.00030820710496091394)
(3,0.0004578359742872467)
(4,0.000729344204429171)
(5,0.0012217936206551305)
(6,0.0018317086615514872)
(7,0.002834975261237489)
(8,0.00426991392022043)
};
\addplot[magenta]
coordinates {
(1,0.0001235162968984801)
(2,0.0003057082819618695)
(3,0.000669063084157357)
(4,0.0008386413608456297)
(5,0.0012633669895491114)
(6,0.001843205658946189)
(7,0.0027757773604931138)
(8,0.004317394861392546)
};
\addplot[magenta]
coordinates {
(1,0.0002056909026251987)
(2,0.0001369217614547618)
(3,0.0004611759990234885)
(4,0.0007534301635048921)
(5,0.0011879890131873505)
(6,0.0018300798058074757)
(7,0.0027119112210899376)
(8,0.00416550287982399)
};
\addplot[magenta]
coordinates {
(1,9.102634091529204e-5)
(2,0.00026696342435775895)
(3,0.0004505120768409744)
(4,0.0008501881218895173)
(5,0.0011599385726226168)
(6,0.0017668016936417885)
(7,0.0028913007039814112)
(8,0.004130763184153135)
};
\addplot[magenta]
coordinates {
(1,7.567257478924667e-5)
(2,0.9433877893145828)
(3,0.10884657031686946)
(4,0.5222643025014942)
(5,0.161178657269903)
(6,0.3070536830752586)
(7,0.0674850192353431)
(8,0.13295685499805657)
};

\addplot[domain = 1:5,black,dashed]{ 0.5*2^(-2*x)};

\end{semilogyaxis}
\end{tikzpicture} \subcaption{decay of the wavelet coefficients for different ANOVA terms. }
\label{fig:pyramid2}
\end{subfigure}
\begin{subfigure}[t]{0.5\linewidth}
\begin{tikzpicture}
\begin{axis}[scale only axis,width = 6cm, height = 6cm,
ylabel={$\rho(\vec u)$},
			xlabel={$\vec u$},
			xtick={1,...,21},
			xticklabel=\empty,
			ytick distance=0.04,
			xtick distance=10,
			legend entries={analytic,from Algorithm~\ref{alg:1}},
legend style={at={(axis cs:7,0.09)},anchor=south west,scale =0.8},
			]
			
			\addplot+[ycomb, green,mark=*, mark options={green}] plot coordinates
			{(1,2/15)
(2,2/15)
(3,2/15)
(4,2/15)
(5,2/15)
(6,2/15)
(7,1/15)
(16,1/15)
(21,1/15)

};
			\addplot+[ycomb, blue,mark=*, mark options={blue}] plot coordinates
			{

(1,0.13447779188546122)
(2,0.1351483977033039)
(3,0.13743017292446147)
(4,0.13656972109675397)
(5,0.13615004905662637)
(6,0.13690793294139264)
(7,0.06012929970744424)
(8,1.3685627857931502e-5)
(9,1.6867784274436945e-5)
(10,1.1488171154329058e-5)
(11,9.249669943767697e-6)
(12,1.1143238475691404e-5)
(13,3.2099717181230255e-6)
(14,1.3152923700658927e-5)
(15,6.3089157275583496e-6)
(16,0.06199351901191231)
(17,1.1672438284073698e-5)
(18,4.455660622753515e-6)
(19,1.631064868088903e-5)
(20,1.0244104348631551e-5)
(21,0.06106532651785512)

};

\node at (axis cs:-0.5,0.145) [anchor=west,scale =0.7] {$\{1\}$};
\node at (axis cs:1.5,0.145) [anchor=west,scale =0.7] {$\{3\}$};
\node at (axis cs:3.5,0.145) [anchor=west,scale =0.7] {$\{5\}$};
\node at (axis cs:7,0.06) [anchor=west,scale =0.7] {$\{1,2\}$};
\node at (axis cs:16,0.06) [anchor=east,scale =0.7] {$\{3,4\}$};
\node at (axis cs:21,0.06) [anchor=east,scale =0.7] {$\{5,6\}$};

		\end{axis}

\end{tikzpicture} \caption{Global sensitivity indices of function~\eqref{eq:pyramid} analytic and from Algorithm~\ref{alg:anova} with $n=1$.}
\label{fig:gsi2}
\end{subfigure}
\begin{subfigure}[t]{0.5\linewidth}
\centering
\begin{tikzpicture}[scale=1]

\begin{semilogyaxis}[xlabel = {level $n$},ylabel ={$M$},
grid=major,
scale only axis, 
width = 6cm, height = 6cm,
      enlarge x limits = 0,
			legend pos = north west
]

\addplot[green,mark=*]  
coordinates {

(1,3946)
(2,24127)
(3,115775)
(4,479956)
(5,1800603)
(6,6277371)
(7,20680713)

};

\addlegendentry{Alg. \ref{alg:1}}
\addplot[red,mark=*]  
coordinates {

(1,617)
(2,2450)
(3,8004)
(4,23451)
(5,64108)
(6,167083)
(7,420561)

};
 \addlegendentry{$U_2$}

\addplot[blue,mark=*]  
coordinates {
(1,173)
(2,617)
(3,1879)
(4,5261)
(5,13944)
(6,35563)
(7,88134)

};
\addlegendentry{$U$ in \eqref{eq:U}}

\end{semilogyaxis}
\end{tikzpicture} \subcaption{Comparison of needed samples between hyperbolic wavelet regression and ANOVA hyperbolic wavelet regression.}
\label{fig:pyramid3}
\end{subfigure}
\caption{Numerical results for the pyramid function~\eqref{eq:pyramid}.}
\label{fig:pyramid}
\end{figure}
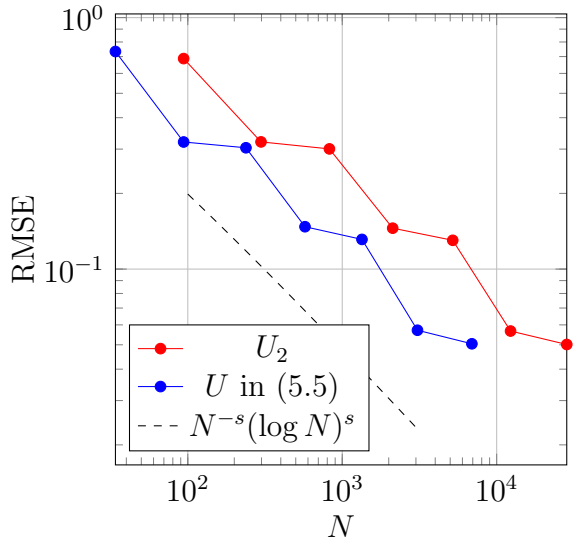
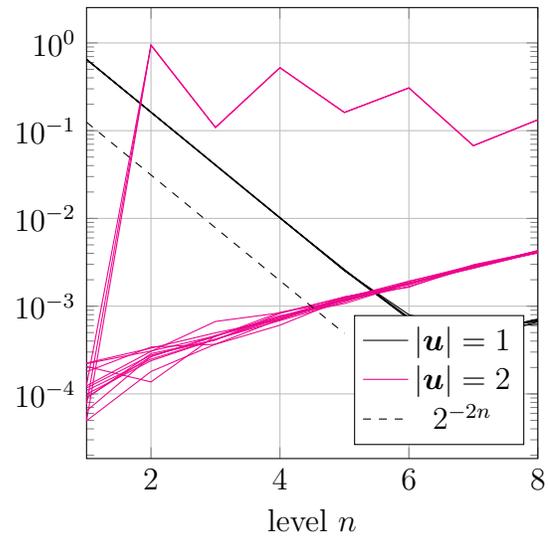
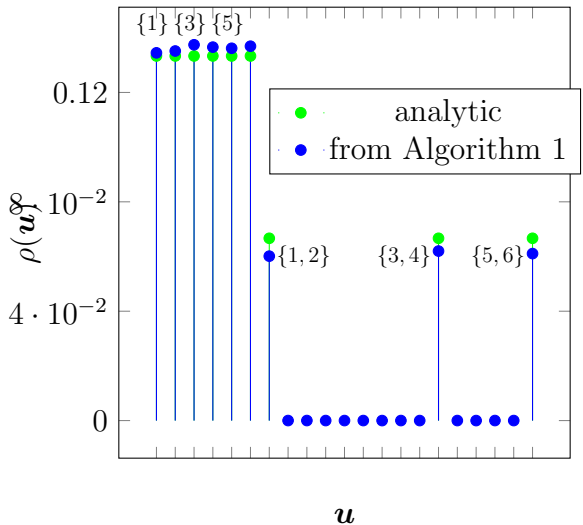
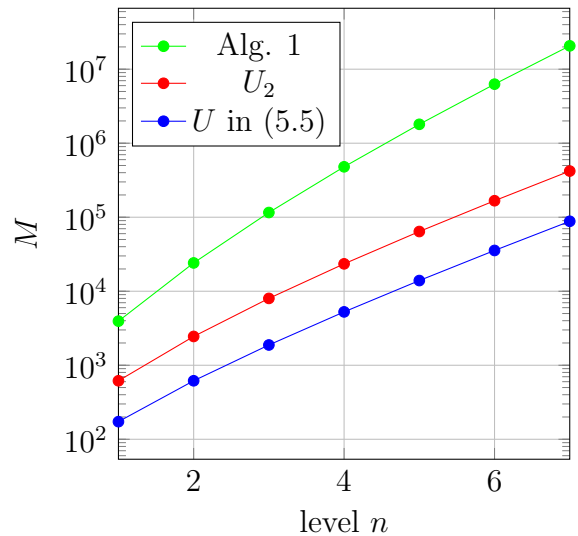
In the sequel we consider the following pyramid-like function having diagonal kinks on the one hand and a coupling of only two variables in each summand on the other hand.   
\begin{equation}\label{eq:pyramid}
f(\vec x) = 2\sqrt{6}\, \sum_{i=1}^3 \left(\tfrac 13 -\max\{|x_{2i-1}|,|x_{2i}|\}\right)\quad,\quad {\vec x} \in [-1/2,1/2)^6.
\end{equation}
This function can be periodically extended since it is constant on the boundary. It takes some efforts (but it is possible) to show that that this functions has a mixed Besov-Nikolskij regularity of $s = 1/2+1/(2p)$ in the sense of $\bB^s_{p,\infty}(\T^6)$, see Definition \ref{diff}, where $p$ may vary in $[1,\infty]$. In fact, the mixed regularity is the same as for the function $g(x_1,x_2) = (x_1-x_2)_+$. So we even have smoothness $s=1$ if $p=1$ which would be relevant for integration problems and lead to a rate of $n^{-1}$. This might have some relation to \cite{GrKuSl10}, where functions of ``type'' $g$ have been considered. \\ 

In fact, we choose the superposition dimension $\nu=2$ and use Algorithm~\ref{alg:anova}
for different levels $n$. We always consider the logarithmic oversampling where $M\gtrsim N\log N$. In a first experiment 
we test steps $1$ to $3$ of Algorithm~\ref{alg:anova}, i.e. we choose as index set $U_2=\{\vec u\in \mathcal{P}([6])\mid |\vec u|\leq 2\}$, see~\eqref{eq:Unu}.
For the RMSE we use a random 
test sample $\X_{\text{test}}$ of size $3M$. The results can be seen in Figure~\ref{fig:pyramid1}, where we plot the 
RMSE in relation to the number of parameters $N$. The regularity in $\bB^{s}_{2,\infty}(\T^6)$ determines the rate $s=3/4$ which coincides with the proposed rate from Corollary~\ref{cor:one_f}. 

Calculating the global sensitivity indices $\rho(\vec u,S_n^{\X,U_\nu})$, see~Figure~\ref{fig:gsi2} in step $4$ of Algorithm~\ref{alg:anova}, gives us 
(already for the low maximal index $n=1$) the index set 
\begin{equation}\label{eq:U}
U = \{\varnothing,\{1\},\{2\},\{3\},\{4\},\{5\},\{6\},\{1,2\},\{3,4\},\{5,6\}\}\subset U_2.
\end{equation}
Actually, the test function $f$ is for this set $U$ in $H^{s,U}_{\mix}(\T^6)$, see~\eqref{eq:HsU}.
In a second experiment we use the steps $6$ to $8$ of our algorithm with this index set $U$ for different maximal levels $n$. 
The resulting RMSE in comparison to the approximation with the index-set $U_\nu$ is plotted in Figure~\ref{fig:pyramid1}. As expected, 
the same approximation error needs less parameters and therefore less samples compared to the bigger set $U_2$.\\

In Figure~\ref{fig:pyramid2} we study the solution vector $a_{\vec j,\vec k}$ from the finest approximation using the index set $U_2$. We 
plot for every 
$\vec u\in U_2$ the sum 
$$\big{(}\sum_{\stackrel{\uinv{j}{u} }{|\vec j_{\vec u}|_1=n} }\sum_{\vec k\in \I_{\vec j}} |a_{\vec j,\vec k}|^2\big{)}^{1/2}.$$
The ANOVA terms of order $1$ are black, the terms of order $2$ are magenta. This example shows that the ANOVA terms of lower order 
can be smoother than the function $f$ itself, see also~\cite{GrKuSl10}. The functions $f_{\vec u}$ with $|\vec u|=1$ are of the 
form $f_{\{1\}}(x_1) = c\, (\tfrac{1}{12}-x_1^2)$ and are in $H^{3/2-\varepsilon}_\mix(\T)$. But the only kink position is at $-\tfrac 12$, 
which is a point where the wavelets also have lower regularity. Therefore, we see that the sum of the wavelet coefficients 
of the one-dimensional terms decay like $2^{-2 n}$, which coincides with the regularity $s = 2$. Regarding the two-dimensional terms, 
the three biggest sums belong to the index 
sets $\{1,2\},\{3,4\}$ and $\{5,6\}$, which are part of $U$ in \eqref{eq:U}. Note, that for $n=1$ the plotted sums do not distinguish the three relevant two-dimensional terms from the other ones, but the computation of the global sensitivity indices with~\eqref{eq:gsi_of_g} does, see Figure~\ref{fig:gsi2}. 

In Figure~\ref{fig:pyramid3} we see the number of necessary samples for different wavelet indices. 
The classical hyperbolic wavelet regression in Algorithm~\ref{alg:1} needs $\O(2^n n^{6})$ samples. The 
restriction to low dimensional terms using the index set $U_2$ reduces this number to $\O(\binom{6}{2} \,2^n n^2)$. 
The further reduction to the index set $U$ in \eqref{eq:U}, which reduces the number of two-dimensional terms to three, 
again reduces the number of necessary samples significantly to $\O(3\cdot 2^n n^2)$, by inheriting the same approximation error.

\appendix{}
\section{Besov-Nikolskij-Sobolev spaces of mixed smoothness on the $d$-torus}

Here we summarize some relevant results from \cite[Chapt.\ 3]{DuTeUl16}. In particular we give the standard definition of the used function spaces. Let us first define Besov-Nikolskij spaces of mixed smoothness. We 
will use the classical definition via mixed moduli of smoothness. Let us
first recall the basic concepts. For univariate functions $f:\T \to \C$ the
$m$-th difference operator $\Delta_h^{m}$ is defined by
\begin{equation*}
\Delta_h^{m}(f,x) := \sum_{j =0}^{m} (-1)^{m - j} \binom{m}{j} f(x +
jh)\quad,\quad x\in \T, h\in [0,1]\,.
\end{equation*}

Let $\vec u$ be any subset of $\{1,...,d\}$. For multivariate functions $f:\T^d\to
\C$ and $\vec h\in [0,1]^d$ the mixed $(m,\vec u)$-th difference operator $\Delta_{\vec h}^{m,\vec u}$
is defined by 
\begin{equation*}
\Delta_{\vec h}^{m,\vec u} := \
\prod_{i \in \vec u} \Delta_{h_i,i}^m\quad\mbox{and}\quad \Delta_{\vec h}^{m,\varnothing} =  \operatorname{Id},
\end{equation*}
where $\operatorname{Id}f = f$ and $\Delta_{h_i,i}^m$ is the univariate
operator applied to the $i$-th coordinate of $f$ with the other variables kept
fixed. 

\begin{definition}\label{diff} 
Let $s > 0$ and $1 \le p \le \infty$. Fixing an integer $m > s$, we define the space $\bB^s_{p,\infty}(\T^d)$ as the set of all $f\in L_p(\T^d)$ such that for any $\vec u \subset \{1,...,d\}$
\[
\big\|\Delta_\bh^{m,\vec u}(f,\cdot)\big\|_{L_p(\T^d)}
\ \le \  
C\, \prod_{i \in \vec u} |h_i|^s
\]
for some positive constant $C$ and introduce the norm in this space
\[
\| \, f \, \|_{\bB^s_{p,\infty}} :=
\sum_{\vec u \subset \{1,...,d\}}
\, | \, f \, |_{\bB^s_{p,\infty}(\vec u)},
\]
where 
\[
| \, f \, |_{\bB^s_{p,\infty}(\vec u)} :=
\sup_{0 < |h_i| \le 1, \ i \in \vec u} \,  
\left(\prod_{i \in \vec u} |h_i|^{-s} \right) \,\big\| \, \Delta_\bh^{m,\vec u}(f,\cdot) \, \big\|_{L_p(\T^d)}\,.
\]
\end{definition}

Let us proceed to the related (univariate) Sobolev spaces with smoothness $s \in \N$. 
\begin{equation*}\label{eq:Hm}
H^s_p(\T):=\left\{f:\T\to \C\mid \norm{f}_{H^s_p(\T)}<\infty\right\},
\end{equation*}
where the norm is defined by
$$\norm{f}_{H^s_p(\T)}^p=\sum_{0\leq k\leq s}\norm{\Dx^{k}f}_{L_p(\T)}^p,$$
with $\Dx^{k}f = \frac{\d^k }{\d x^k}f$.
An equivalent norm for this function space is 
$$\norm{f}_{H^s_p(\T)}^p=\norm{f}_{L_p(\T)}^p+\norm{\Dx^{s}f}_{L_p(\T)}^p.$$

In this paper we are mainly interested in the case $p=2$ but higher dimensions $d$. We introduce the 
so-called Sobolev spaces with dominating mixed derivatives $H^m_{\mix}(\T^d)$ in the usual way:
\begin{equation}\label{eq:Hsmix}
H^s_{\mix}(\T^d):=\left\{f:\T^d\to \C\mid \norm{f}_{H^s_{\mix}(\T^d)}<\infty\right\},
\end{equation}
where the norm is defined by
\begin{equation}\label{eq:Hmnorm}
\norm{f}_{H^s_{\mix}(\T^d)}=\sum_{0\leq \norm{\vec k}_\infty\leq s}\norm{\Dx^{\vec k}f}_{L_2(\T^d)},
\end{equation}
with the partial derivatives $\Dx^{\vec k} f = \tfrac{\partial^{k_1+\ldots+k_d} }{\partial x_1^{k_1}\cdots\partial x_d^{k_d} }$.

It clearly holds for $d=1$
\begin{equation*}\label{univ}
      H^s_{\mix}(\T) = H^s_2(\T) =: H^s(\T)\,.
\end{equation*}
The case $p=2$ allows for a straight-forward extension to fractional smoothness parameters.

\begin{definition}Let $s>0$. Then we define
\begin{equation*}\label{eq:Hsmix2}
H^s_{\mix}(\T^d):=\left\{f:\T^d\to \C\mid \norm{f}_{H^s_{\mix}(\T^d)}<\infty\right\},
\end{equation*}
where the norm is defined by
\begin{equation*}
\label{Fourier_char}\norm{f}_{H^s_{\mix}(\T^d)}^2=\sum_{\vec k\in\Z^d}|c_{\vec k}(f)|^2\prod_{i=1}^d(1+|k_i|^2)^{s},
\end{equation*}
\end{definition}
This norm  is equivalent to the norm in \eqref{eq:Hmnorm} for $s\in \N$, see \cite{KuSiUl14}. We will consider the case where $s>\tfrac 12$, since in this case we have that $H^s_{\mix}(\T^d)	\hookrightarrow C(\T^d)$, which is necessary to sample the function.

There is a further useful equivalent norm which is based on a decomposition of $f$ in dyadic blocks. We introduce the dyadic blocks 
\begin{equation*}
J_j=\begin{cases}
\{k\in \Z\mid 2^{j-1}\leq |k|<2^{j}\}& \text{ if } j\geq 1,\\
\{0\}& \text{ if } j=0.\end{cases}
\end{equation*}
For $\bj \in \N_0^d$ we define 
$$
J_{\bj}:= J_{j_1} \times ... \times J_{j_d}\,. 
$$
if all components belong to $\N_0$. Using these dyadic blocks, we decompose the Fourier series of the function $f$ into 
$$f=\sum\limits_{\bj \in \Z_0^d} f_{\bj}({\vec x}) \text{ with } f_{\bj}({\vec x})=\sum\limits_{\bk \in J_{\bj}} c_{\bk}\e^{2\pi \im\langle\bk,{\vec x}\rangle}$$
in case that $\bj\in \N_0^d$ and $f_{\bj}:=0$ otherwise. This immediately gives
\begin{equation}\label{dyadic}
    \|f\|_{H^s_{\mix}}^2 \asymp \sum\limits_{\bj \in \N_0^d} 2^{2|\bj|_1s} \|f_{\bj}\|_{L_2(\T^d)}^2.    
\end{equation}

Interestingly, there is also a Fourier analytic characterization of the above defined Besov-Nikolskij spaces $\bB^s_{p,\infty}(\T^d)$ which even works for $1<p<\infty$. Instead of taking the $\ell_2$-norm of the weighted sequence $(2^{|\bj|_1s}\|f_{\bj}\|_{L_p(\T^d)})_{\bj \in \N_0^d}$ we take the $\ell_{\infty}$-norm. We have  
\begin{equation*}\label{dyadic_besov}
    \|f\|_{\bB^s_{p,\infty}} \asymp \sup\limits_{\bj \in \N_0^d} 2^{|\bj|_1 s }\|f_j\|_{L_p(\T^d)}\,.
\end{equation*}

\section*{Acknowledgments}
Laura Lippert and Daniel Potts acknowledge funding by Deutsche Forschungsgemeinschaft (DFG, German Research Foundation) - Project-ID 416228727 - SFB 1410. Tino Ullrich would like to acknowledge support by the DFG Ul-403/2-1.

\bibliographystyle{abbrv}

\end{document}